\documentclass{amsart}

\usepackage{amssymb}
\usepackage{latexsym}
\usepackage{amsmath}

\def\wt{\widetilde}
\def\wh{\widehat}
\def\ov{\overline}
\def \im{{\rm Im}}
 \def\up{\upharpoonright}
\def\cH{\mathcal H}

\def\cB{\mathcal B}   
\def\cK{\mathcal K} \def\cL{\mathcal L}
\def\cM{\mathcal M} \def\cN{\mathcal N}  
 \def\cT{\mathcal T} \def\cI{\mathcal I}

\def \gH{\mathfrak H}   \def \gN{\mathfrak N}

\def \bC{\mathbb C}    \def\bR{\mathbb R}
\def\bH{\mathbb H} 
\def \l{\lambda}
\def \a{\alpha} \def \b{\beta}  \def \L{\Lambda}  \def \s{\sigma}
  \def\g {\gamma}
\def\d {\delta}   \def\F{\Phi}
\def \f{\varphi}  \def \G{\Gamma} \def\D {\Delta}\def\Si{\Sigma}

\def \C{\widetilde {\mathcal C}}
\def \CA{\C(\cH_0,\cH_1)}

\def \cd {\cdot}

\def\AC {AC(\cI; \bH)}   \def\LI {L_\Delta^2(\cI)}
\def\lI {\cL_\Delta^2(\cI)}
\def\tma{\cT_{\max}} \def\tmi{\cT_{\min}} \def\Tma{T_{\max}} \def\Tmi{T_{\min}}
\def\Lma{L_{\max}} \def\Lmi{L_{\min}}

\def \dom {{\rm dom}\,}  \def \ran {{\rm ran}\,}  \def \ker{{\rm ker\,}}
 \def \mul {{\rm mul}\,} \def \sign {{\rm sign}\,}

\def \exa { {Ext}_A}
  \def\tm{\times}

\def \pair {\tau=\{\tau_+,\tau_-\}}
\def \CR {\bC\setminus\bR}
\def\St{\Sigma_\tau(\cdot)} \def\LS {L^2(\Sigma_\tau ; H_0)}
\def\LSB {L^2(\Si_\tau;\bold H_0)}
\def\Tt {\wt T^\tau}

\newcommand {\lo}[1] {\cL_\D^2[#1,\bH ]}
\newcommand {\Ca}[1] {\emph {Case #1}}

\def\bta{\{\cH_0\oplus \cH_1,\Gamma _0,\Gamma _1\}}

\newtheorem{theorem}{Theorem}[section]
\newtheorem{proposition}[theorem]{Proposition}
\newtheorem{corollary}[theorem]{Corollary}
\newtheorem{lemma}[theorem]{Lemma}

\theoremstyle{definition}

\theoremstyle{definition}
\newtheorem {definition} [theorem]{Definition}
\theoremstyle{remark}
\newtheorem{remark}[theorem]{Remark}
\numberwithin{equation}{section}
\begin{document}
\title[On eigenfunction expansions]
{On eigenfunction expansions of first-order symmetric systems and
ordinary differential operators of an odd order}
\author {Vadim Mogilevskii}
\address{Department of Mathematical Analysis, Lugans'k National University,
 Oboronna Str. 2, 91011  Lugans'k,   Ukraine}
\email{vim@mail.dsip.net}



\begin{abstract}
We study general (not necessarily Hamiltonian) first-order
symmetric  systems $J y'-B(t)y=\D(t) f(t)$ on an interval
$\cI=[a,b\rangle $ with the regular endpoint $a$. It is assumed
that the deficiency indices $n_\pm(\Tmi)$ of the   minimal
relation $\Tmi$  satisfy $n_+(\Tmi)< n_-(\Tmi)$. We define
 $\l$-depending  boundary
conditions which are analogs  of separated self-adjoint boundary
conditions for Hamiltonian systems. With a boundary value problem
involving such conditions we associate an exit space self-adjoint
extension $\wt T$ of $\Tmi$ and the $m$-function $m(\cd)$, which
is an analog of the Titchmarsh-Weyl coefficient for the
Hamiltonian system. By using $m$-function we obtain the
eigenfunction expansion with the spectral function $\Si(\cd)$ of
the minimally possible dimension and  characterize the case when
spectrum of $\wt T$ is defined by  $\Si(\cd)$. Moreover, we
parametrize all spectral functions  in terms of a Nevanlinna type
boundary parameter.  Application of these results to ordinary
differential operators of an odd order enables us to complete the
results by Everitt and Krishna Kumar on the Titchmarsh-Weyl theory
of such operators.
\end{abstract}
\maketitle
\section{Introduction}
Let  $H$ and $\wh H$ be finite dimensional Hilbert spaces and let
\begin{gather}\label{1.1}
H_0:=H\oplus\wh H, \qquad \bH:=H_0\oplus H=H\oplus\wh H \oplus H.
\end{gather}

In the paper we study  first-order symmetric  systems of
differential equations defined on an interval $\cI=[a,b\rangle,
-\infty<a <b\leq\infty,$ with the regular endpoint $a$ and regular
or singular endpoint $b$. Such a system is of the form
\cite{Atk,GK}
\begin {equation}\label{1.2}
J y'-B(t)y=\D(t) f(t), \quad t\in\cI,
\end{equation}
where $B(t)=B^*(t)$ and $\D(t)\geq 0$ are the $[\bH]$-valued
functions on $\cI$ and
\begin {equation} \label{1.3}
J=\begin{pmatrix} 0 & 0&-I_H \cr 0& i I_{\wh H}&0\cr I_H&
0&0\end{pmatrix}:H\oplus\wh H\oplus H \to H\oplus\wh H\oplus H.
\end{equation}
System \eqref{1.2} is called a Hamiltonian system if $\wh
H=\{0\}$.

Throughout the paper we  assume  that  system \eqref{1.2} is
definite. The latter means  that   for any $\l\in\bC$ each
solution $y(\cd)$ of the equation
\begin {equation}\label{1.4}
J y'-B(t)y=\l \D(t) y
\end{equation}
satisfying $\D(t)y(t)=0$  (a.e. on $\cI$) is trivial, i.e., $
y(t)=0,\;t\in\cI$.

In what follows  we denote by $\gH:=\LI$ the Hilbert space of
$\bH$-valued Borel  functions $f(\cd)$ on $\cI$   (in fact,
equivalence classes) satisfying $||f||_\D^2:= \int\limits_\cI
(\D(t)f(t),f(t))_\bH\,dt<\infty$.

Studying of symmetric systems  is basically motivated by the fact
that system \eqref{1.4} is a more general objet than a formally
self-adjoint differential equation of an arbitrary order with
matrix coefficients. In fact, such equation is reduced  to the
system  \eqref{1.4} of a special form with $J$ given by
\eqref{1.3}; moreover, this system is Hamiltonian precisely in the
case when the differential equation is of an even order (we will
concern this questions below).

As  is known,  the extension theory of symmetric linear relations
gives a natural framework   for investigation of the boundary
value problems for symmetric systems (see
\cite{BHSW10,DLS93,Kac03,LesMal03,Orc} and references therein).
According to \cite{Kac03,LesMal03, Orc}  the system \eqref{1.2}
generates the minimal linear relation $\Tmi$ and the maximal
linear relation $\Tma$ in $\gH$. It turns out that $\Tmi$ is a
closed symmetric relation with not necessarily equal deficiency
indices $n_\pm(\Tmi)$ and $\Tma=\Tmi^*$. Since system \eqref{1.2}
is  definite,  $n_\pm(\Tmi)$ can be defined as a number of
$L_\D^2$-solutions of  \eqref{1.4} for $\l\in\bC_\pm$.

A description of various classes of extensions of $\Tmi$
(self-adjoint, $m$-dissipative, etc.) in terms of boundary
conditions is an important problem in the spectral theory of
symmetric systems. For Hamiltonian system \eqref{1.2} self-adjoint
separated boundary conditions were described in \cite{HinSch93}.
Moreover, the Titchmarsh--Weyl coefficient $M_{TW}(\l)$ of the
boundary value problem for  Hamiltonian system with self-adjoint
separated boundary conditions  was defined by various methods in
\cite{HinSch93,Kra89,Khr06}. Using $M_{TW}(\cd)$ one obtains the
Fourier transform with the spectral function $\Si(\cd)$ of the
minimally possible dimension $N_\Si=\dim H$ (see
\cite{DLS93,HinSch98,Kac03}). At the same time according to
\cite{Mog12} \emph{non-Hamiltonian system \eqref{1.2} does not
admit  self-adjoint separated boundary conditions}. Moreover,  the
inequality $n_+(\Tmi)\neq n_-(\Tmi)$, and hence absence  of
self-adjoint boundary conditions is a typical situation for such
systems. Therefore the following problems are of certain interest:

$\bullet$\; To find (may be $\l$-depending) analogs of
self-adjoint separated boundary conditions for general (not
necessarily Hamiltonian) systems \eqref{1.2} and describe such
type conditions;

$\bullet$\; To describe in terms of boundary conditions all
spectral matrix functions that have the minimally possible
dimension and  investigate the corresponding Fourier transforms.

In the paper \cite{Mog13.1} these problems were considered for
symmetric systems \eqref{1.2} satisfying $n_-(\Tmi)\leq
n_+(\Tmi)$. In the present paper we solve these problems in the
opposite case $n_+(\Tmi) < n_-(\Tmi)$. It turns out that this case
requires a somewhat another approach in comparison with
\cite{Mog13.1}, although the ideas of both the papers are similar.
Moreover, we show that in the case $n_+(\Tmi) < n_-(\Tmi)$ there
is a class of exit space self-adjoint extensions of $\Tmi$ with
special spectral properties.  We  apply also the obtained results
to ordinary differential operators of an odd order with matrix
valued coefficients and arbitrary deficiency indices.

Let $\nu_{b+}$ and $\nu_{b-}$ be indices of inertia of the
skew-Hermitian bilinear form $[y,z]_b$ defined on $\dom\Tma$ by
\begin {equation*}
[y,z]_b=\lim _{t\uparrow b} (J y(t), z(t)), \quad y,z\in\dom\Tma,
\end{equation*}
It turns out that system \eqref{1.2} with $n_+(\Tmi) < n_-(\Tmi)$
has different properties depending on the sign of
$\nu_{b+}-\nu_{b-}$. Within this section we present the results of
the paper assuming a simpler case $\nu_{b+}-\nu_{b-}\leq 0$.

Let a function $y\in\dom\Tma$ be represented as $y(t)=\{y_0(t),
\wh y(t), y_1(t) \}(\in H\oplus\wh H\oplus H)$. A crucial role in
our considerations is played by a symmetric linear relation
$T(\supset \Tmi)$ in $\gH$ given by means of boundary conditions
as follows:
\begin {equation*}
T=\{\{ y,  f\}\in\Tma: \, y_1(a)=0, \;\wh y(a)=0,\; \wt\G_{0b}y
=\G_{1b}y=0 \}
\end{equation*}
Here $\wt\G_{0b}:\dom\Tma\to\wt \cH_b$ and $\G_{1b}: \dom\Tma\to
\cH_b$ are linear mappings with special properties, $\wt \cH_b$
and $\cH_b(\subset \wt \cH_b)$ are auxiliary finite dimensional
Hilbert spaces. In fact , $\wt\G_{0b}y$ and $\G_{1b}y$ are
singular boundary values of a function $y\in\dom\Tma$ at the
endpoint $b$.

Recall that a linear relation $\wt T=\wt T^*$ in a wider Hilbert
space $\wt \gH\supset \gH$ satisfying $T\subset \wt T$ is called
an exit space self-adjoint extension of $T$. Moreover,  a
generalized resolvent $R(\cd)$ and a spectral function $F(\cd)$ of
$T$ are defined by
\begin {equation*}
R(\l)=P_\gH(\wt T -\l)^{-1}\up \gH, \quad \l\in\CR,\;\;\;
\text{and}\;\;\; F(t)=P_{\gH}E(t)\up \gH,\quad t\in\bR,
\end{equation*}
where $E(\cdot)$ is the orthogonal spectral function (resolution
of identity) of $\wt T$. As is known \cite{AkhGla} exit space
self-adjoint extensions exist for  symmetric linear relations with
arbitrary (possibly unequal) deficiency indices.

 To describe
the set of all generalized resolvents of $T$ we use the Nevanlinna
type class $\wt R_- (\wh H\oplus\wt\cH_b,\cH_b)$ of holomorphic
operator pairs $\tau=\{D_0(\l), D_1(\l)\}$. Such a pair is formed
by defined on $\bC_-$ holomorphic operator functions
\begin {equation}\label{1.8}
D_0(\l)=(\wh D_0(\l),\wt D_{0b}(\l)):\wh H\oplus \wt\cH_b\to (\wh
H\oplus \wt\cH_b)\;\;\; \text{and}\;\;\;  D_1(\l)(\in [\cH_b,\wh
H\oplus \wt\cH_b])
\end{equation}
with special properties (see \cite{Mog13.2}). We show that each
generalized resolvent $y=R(\l)f, \; f\in\gH,$ is given as the
$L_\D^2$-solution of the following boundary-value  problem with
$\l$-depending boundary conditions:
\begin{gather}
J y'-B(t)y=\l \D(t)y+\D(t)f(t), \quad t\in\cI,\label{1.9}\\
y_1(a)=0, \quad i \wh D_0(\l)\wh y(a)+\wt
D_{0b}(\l)\wt\G_{0b}y+D_1(\l)\G_{1b}y=0, \quad \l\in\bC_-.
\label{1.10}
\end{gather}
Here $(\wh D_0(\l),\wt D_{0b}(\l))=:D_0(\l)$ and $D_1(\l)$ are
components of a pair $\tau=\{D_0(\l), D_1(\l)\}\in\wt R_- (\wh
H\oplus\wt\cH_b,\cH_b) $ (see \eqref{1.8}), so that the second
equality in \eqref{1.10} is a Nevanlinna type boundary condition
involving boundary values of a function $y$ at  both endpoints $a$
and $b$. Thus, \emph{investigation of boundary value problems for
the system \eqref{1.2} in the case $n_+(\Tmi)<n_-(\Tmi)$ require
 use of boundary conditions of another class in comparison
with the case $n_-(\Tmi)\leq n_+(\Tmi)$}(cf.\cite{Mog13.1}).  One
may consider a pair $\tau$   as a boundary parameter, since
$R(\l)$ runs over the set of all generalized resolvents of $T$
when $\tau$ runs over the set of all holomorphic operator pairs of
the class $\wt R_- (\wh H\oplus\wt\cH_b,\cH_b)$. To indicate this
fact explicitly we write $R(\l)=R_\tau(\l)$ and $F(t)=F_\tau(t)$
for the generalized resolvents and spectral functions of $T$
respectively. Moreover, we denote by $\wt T=\wt T^\tau$ the exit
space self-adjoint extension of $T$ generating $R_\tau(\cd)$ and
$F_\tau(\cd)$.

Next assume that $\f(\cd,\l)$ and $\psi (\cd,\l)$ are
$[H_0,\bH]$-valued operator solutions of the equation  \eqref{1.4}
satisfying  the initial conditions
\begin {equation*}
\f(a,\l)=\begin{pmatrix} I_{H_0}\cr 0\end{pmatrix}(\in [H_0,
H_0\oplus H]),\quad \psi(a,\l)=\begin{pmatrix} -\tfrac i 2 P_{\wh
H}\cr -P_H\end{pmatrix}(\in [H_0, H_0\oplus H]).
\end{equation*}
(here $P_H$ and $P_{\wh H}$ are the orthoprojectors in $H_0$ onto
$H$ and $\wh H$ respectively). We show that, for each Nevanlinna
boundary parameter $\tau = \{D_0(\l),D_1(\l)\}$ of the form
\eqref{1.8}, there exists a unique operator function
$m_\tau(\cd):\CR\to [H_0]$ such that the operator solution
\begin {equation*}
v_\tau(t,\l):=\f(t,\l)m_\tau(\l) +\psi (t,\l)
\end{equation*}
of Eq. \eqref{1.4} has the following property: for every $h_0\in
H_0$ the function $y=v_\tau (t,\l)h_0$ belongs to $\LI$ and
satisfies the boundary condition
\begin{equation*}
\wh D_0(\l)(i\wh y(a)-P_{\wh H}h_0)+ \wt D_{0b}(\l)\wt\G_{0b}y
+D_1(\l)\G_{1b}y=0,\quad \l\in\bC_-.
\end{equation*}
We call $m_\tau(\cd)$  the $m$-function corresponding to the
boundary value problem \eqref{1.9}, \eqref{1.10}; really,
$m_\tau(\cd)$ is an analog of the Titchmarsh-Weyl coefficient
$M_{TW}(\l)$ for Hamiltonian systems. It turns out that
$m_\tau(\cd) $ is a Nevanlinna operator function satisfying the
inequality
\begin {equation*}
(\im \,\l)^{-1}\cd \im\, m_\tau(\l)\geq \int_\cI
v_\tau^*(t,\l)\D(t) v_\tau(t,\l)\, dt, \quad\l\in\CR.
\end{equation*}

Next  we study eigenfunction expansions of the boundary value
problems for symmetric systems. Namely, let $\tau $ be a boundary
parameter and let $F_\tau(\cd)$ be the spectral function of $T$
generated by the boundary value problem \eqref{1.9}, \eqref{1.10}.
A distribution operator-valued function $\Si_\tau(\cd):\bR\to
[H_0]$ is called a spectral function of this problem  if, for each
function $f\in\gH$ with compact support, the Fourier transform
\begin {equation*}
\wh f(s)=\int _\cI \f^*(t,s)\, \D(t)\,f(t)\, dt
\end{equation*}
satisfies
\begin {equation}\label{1.14}
((F_\tau(\b)-F_\tau(\a))f,f)_{\gH}=\int_{[\a,\b)}
(d\Si_{\tau}(s)\wh f(s), \wh f(s))
\end{equation}
for any compact interval $[\a,\b)\subset\bR$. We show that for
each boundary parameter $\tau$ there exists a unique spectral
function $\Si_\tau(\cd)$ and it is recovered  from the
$m$-function $m_\tau(\cd)$ by means of the Stieltjes inversion
formula
\begin {equation}\label{1.15}
\Si_\tau(s)=-\lim\limits_{\delta \to +0}\lim\limits_{\varepsilon
\to +0}\frac 1 \pi \int_{-\delta}^{s-\delta} \im\, m_\tau(\sigma
-i\varepsilon)\, d\sigma.
\end{equation}
Below (within this section) we assume  for simplicity that $T$ is
a  (not necessarily densely defined) operator, i.e., $\mul
T=\{0\}$.

It follows from \eqref{1.14} that  the mapping  $V f=\wh f$ (the
Fourier transform) admits a continuous extension to a contractive
linear mapping $V: \gH \to \LS$ (for the strict definition of the
Hilbert space $\LS$ see e.g. \cite[Ch.13.5]{DunSch}. As in
\cite{Mog13.1} one proves that $V$ is an isometry (that is, the
Parseval equality $||\wh f||_{\LS}=||f||_\gH$ holds for every
$f\in\gH$) if and only if the exit space extension $\wt T^{\tau}$
in $\wt \gH\supset \gH$ is an operator, i.e., $\mul \wt
T^\tau=\{0\}$. In this case one may define the inverse Fourier
transform in the explicit form (see \eqref{6.30}). Moreover, if
$V$ is an isometry, then there exists a unitary extension  $\wt
V\in [\wt\gH,\LS]$ of  $V$ such that the operator $\wt T^\tau$ and
the multiplication operator $\L$ in $\LS$ are unitarily equivalent
by means of $\wt V$. Hence, the operators $\Tt$ and $\L$ have the
same spectral properties; for instance,   multiplicity of the
spectrum of $\Tt$ does not exceed $\dim H_0(=\dim H+\dim \wh H)$.

Now assume that the boundary parameter $\tau=\{D_0(\l),D_1(\l)\}$
is  (cf. \eqref{1.8})
\begin {equation}\label{1.16}
D_0(\l)={\rm diag} \, (I_{\wh H},\, \ov D_0(\l))(\in [\wh
H\oplus\wt\cH_b]),\qquad D_1(\l)=(0,\,\ov D_1(\l))^\top(\in
[\cH_b,\wh H\oplus\wt\cH_b]).
\end{equation}
 We show that in this
case the corresponding $m$-function $m_\tau(\cd)$ is of the
triangular form
\begin {equation*}
m_{\tau}(\l)=\begin{pmatrix} m_{\tau, 1}(\l)  & m_{\tau, 2}(\l)
\cr 0 & -\tfrac i 2 I_{\wh H}\end{pmatrix}:H\oplus\wh H\to
H\oplus\wh H, \quad \l\in\bC_-,
\end{equation*}
so that the spectral function $\St$ has the block matrix
representation
\begin {equation}\label{1.18}
\Si_{\tau}(s)=\begin{pmatrix} \Si_{\tau,1}(s) & \Si_{\tau, 2}(s)
\cr \Si_{\tau, 3}(s) & \tfrac 1 {2\pi} s I_{\wh H}
\end{pmatrix}:H\oplus\wh H\to H\oplus\wh H, \quad s\in\bR.
\end{equation}
Here $\Si_{\tau,1}(\cd)$ is an $[H]$-valued distribution function,
which can be defined by means of the Stieltjes inversion formula
for $m_{\tau, 1}(\l)$. It follows from \eqref{1.18}, that in the
case when a boundary parameter $\tau=\{D_0(\l),D_1(\l)\}$ is of
the form \eqref{1.16} and the Fourier transform is an isometry,
the corresponding exit space self-adjoint extension $\wt T^\tau$
of $T$ has the following spectral properties (for more details see
Theorem \ref{th6.13}):

(S1) $\;\s_{ac}(\wt T^\tau)=\bR$, where $\s_{ac}(\wt T^\tau)$ is
the absolutely continuous spectrum of $\Tt$.

(S2) $\s_s(\Tt)=Ss(\Si_{\tau,1})$, where $\s_s(\Tt)$ is the
singular spectrum of $\Tt$ and $Ss(\Si_{\tau,1})$ is a closed
support of the measure generated by the singular component of
$\Si_{\tau,1} $. Hence the multiplicity of the singular spectrum
of $\Tt$ does not exceed $\dim H$, which yields the same estimate
for multiplicity of each eigenvalue $\l_0$ of $\Tt$.

Next, we show that all spectral functions $\Si_\tau(\cd)$ can  be
parametrized immediately in terms of the  boundary parameter
$\tau$. More precisely, we show that there exists an operator
function
\begin {equation}\label{1.19}
X_-(\l)=\begin{pmatrix}m_0(\l) & \Phi_-(\l) \cr \Psi_-(\l) & \dot
M_-(\l)\end{pmatrix}:H_0\oplus (\wh H\oplus\wt\cH_b)\to
H_0\oplus\cH_b , \quad \l\in\bC_-,
\end{equation}
such that for each  boundary parameter $\tau =\{D_0(\l), D_1(\l)
\}$ the corresponding $m$-function $m_\tau(\cd)$ is given by
\begin {equation}\label{1.20}
m_\tau(\l)=m_0(\l)+\F_-(\l)(D_0(\l)-D_1(\l)\dot M_-(\l))^{-1}
D_1(\l) \Psi_-(\l), \quad\l\in\bC_-.
\end{equation}
Thus, formula \eqref{1.20} together with the Stieltjes inversion
formula \eqref{1.15} defines a (unique) spectral function
$\Si_\tau(\cd)$ of the boundary value problem \eqref{1.9},
\eqref{1.10}. Note  that entries of the matrix \eqref{1.19} are
defined in terms of the boundary values of respective operator
solutions of Eq.\eqref{1.4}. We also describe boundary parameters
$\tau $ for which the Fourier transform $V$ is an isometry and
characterize the case when $V$ is an isometry for every boundary
parameter $\tau$ (see Theorem \ref{th6.14}). Note that a
description of spectral functions for various classes of boundary
problems in the form close to \eqref{1.20}, \eqref{1.15} can be
found in \cite{Ful77,Gor66,HinSha82,Hol85,KacKre,Mog07}.

Clearly,  all the foregoing results can be reformulated (with
obvious simplifications) for Hamiltonian systems \eqref{1.2}.

We suppose that the above results about exit space extensions
$\Tt$ of $T$ may be useful in the case when $\Tt$ is a
self-adjoint relation in $\wt\gH=L_\D^2(\bR)$ induced by the
symmetric system on the whole line $\bR$. More precisely, we
assume that spectral properties of such $\wt T^\tau$ may be
characterized  in terms of the objects ($m$-function, boundary
parameter etc.) associated with the restriction of this system
onto the semi-axis $\cI=[0,\infty)$. These problems will be
considered elsewhere.

If $n_+(\Tmi)$ takes on  the minimally possible value
$n_+(\Tmi)=\dim H$, then $n_+(\Tmi)\leq n_-(\Tmi) $ and the above
results can be rather simplified. Namely, in this case $T$ is a
maximal symmetric relation and hence there exists a unique
generalized resolvent of $T$. Therefore there is a unique
$m$-function $m(\cd)$ which has the triangular form
\begin {equation}\label{1.21}
m(\l)=\begin{pmatrix} M(\l)  & N_-(\l) \cr 0 & -\tfrac i 2 I_{\wh
H}\end{pmatrix}:H\oplus\wh H\to H\oplus\wh H, \quad \l\in\bC_-.
\end{equation}
Moreover, the spectral function $\Sigma(\cd)$ of the corresponding
boundary value problem is of the form \eqref{1.18} and the Fourier
transform $V$ is an isometry. Therefore a (unique) exit space
self-adjoint extension $T_0$ of $T$ has the spectral properties
(S1) and (S2).

Note that systems \eqref{1.2} with $\wh H\neq \{0\}$ and both minimal
deficiency indices $n_+(\Tmi)=\dim H$ and $n_-(\Tmi)=\dim H+\dim\wh H$  were
studied in the paper by Hinton and Schneider \cite{HinSch06}, where the concept
of the ''rectangular'' Titchmarsh-Weyl coefficient $M_{TW}(\l)(\in
[H,H\oplus\wh H]), \; \l\in\bC_+, $ was introduced. One can easily show that in
fact $M_{TW}(\l)=(M^*(\ov\l), \, N_-^*(\ov\l))^\top$, where $M(\l)$ and
$N_-(\l)$ are taken from \eqref{1.21}.

In the final part of the paper we consider the operators generated
by a differential expression $l[y]$ of an odd order $2m+1$ with
$[H]$-valued coefficients ($H$ is a  Hilbert space with $r:=\dim
H<\infty $) defined on an interval $\cI=[a, b\rangle $ (see
\eqref{7.1}). In the particular case of scalar coefficients such
differential operators have been investigated in the papers by
Everitt and Krishna  Kumar \cite{EveKum76.1,EveKum76.2,Kum82},
where the limiting process from the compact intervals
$[a,\b]\subset \cI$ was used for construction of the
Titchmarsh-Weyl matrix $M_{TW}(\l)= (m_{jk}(\l))_{j,k=1}^{m+1}$.
Note that the results of these papers can not be considered to be
completed; in particular, an attempt to define self-adjoint
boundary conditions in \cite{EveKum76.2} gave rise to hardly
verifiable assumptions even in the case of minimally possible
equal  deficiency indices $n_\pm(L_{min})=m+1$ of the minimal
operator $\Lmi$.

Our approach is based on the known fact \cite{KogRof75} that the
equation $l[y]=\l y$ is equivalent to a special  symmetric
non-Hamiltonian system \eqref{1.4}. This enables us to extend the
obtained results concerning symmetric systems to the expression
$l[y]$ with arbitrary (possibly unequal) deficiency indices of
$\Lmi$. In particular, we describe self-adjoint and $\l$-depending
boundary conditions for $l[y]$, which are analogs of self-adjoint
separated boundary conditions for differential expressions of an
even order. This makes it possible to construct eigenfunction
expansion with spectral matrix function $\Si (\cd)$ of the
dimension $(m+1)r\times (m+1)r$ and to describe all $\Si(\cd)$
immediately in terms of boundary conditions (for operators of an
even order and separated boundary conditions such a description
was obtained in \cite{Mog07}).

In conclusion note that the above results  are obtained with the
aid of the method of boundary triplets and the corresponding Weyl
functions in the extension theory of symmetric linear relations
(see \cite{DM91,GorGor,Mal92,MalNei12,Mog06.2} and references
therein).

\section{Preliminaries}
\subsection{Notations}
The following notations will be used throughout the paper: $\gH$,
$\cH$ denote Hilbert spaces; $[\cH_1,\cH_2]$  is the set of all
bounded linear operators defined on the Hilbert space $\cH_1$ with
values in the Hilbert space $\cH_2$; $[\cH]:=[\cH,\cH]$; $A\up
\cL$ is the restriction of an operator $A$ onto the linear
manifold $\cL$; $P_\cL$ is the orthogonal projector in $\gH$ onto
the subspace $\cL\subset\gH$; $\bC_+\,(\bC_-)$ is the upper
(lower) half-plane  of the complex plane.

Recall that a closed linear   relation   from $\cH_0$ to $\cH_1$
is a closed linear subspace in $\cH_0\oplus\cH_1$. The set of all
closed linear relations from $\cH_0$ to $\cH_1$ (in $\cH$) will be
denoted by $\C (\cH_0,\cH_1)$ ($\C(\cH)$). A closed linear
operator $T$ from $\cH_0$ to $\cH_1$  is identified  with its
graph $\text {gr}\, T\in\CA$.

For a linear relation $T\in\C (\cH_0,\cH_1)$  we denote by $\dom
T,\,\ran T, \,\ker T$ and $\mul T$  the domain, range, kernel and
the multivalued part of $T$ respectively. Recall also that the
inverse and adjoint linear relations of $T$ are the relations
$T^{-1}\in\C (\cH_1,\cH_0)$ and $T^*\in\C (\cH_1,\cH_0)$ defined
by
\begin{gather}
T^{-1}=\{\{h_1,h_0\}\in\cH_1\oplus\cH_0:\{h_0,h_1\}\in T\}\nonumber\\
T^* = \{\{k_1,k_0\}\in \cH_1\oplus\cH_0:\, (k_0,h_0)-(k_1,h_1)=0,
\; \{h_0,h_1\}\in T\}\label{2.0}.
\end{gather}

  In the
case $T\in\CA$ we write $0\in \rho (T)$ if $\ker T=\{0\}$\ and\
$\ran T=\cH_1$, or equivalently if $T^{-1}\in [\cH_1,\cH_0]$;
$0\in \wh\rho (T)$\ \ if\ \ $\ker T=\{0\}$\ and\   $\ran T$ is a
closed subspace in $\cH_1$. For a linear relation $T\in \C(\cH)$
we denote by $\rho (T):=\{\l \in \bC:\ 0\in \rho (T-\l)\}$ and
$\wh\rho (T)=\{\l \in \bC:\ 0\in \wh\rho (T-\l)\}$ the resolvent
set and the set of regular type points  of $T$ respectively.

Recall also the following definition.
\begin{definition}\label{def2.0}
A holomorphic operator function $\Phi (\cd):\bC\setminus\bR\to
[\cH]$ is called a Nevanlinna function  if $\im\, \l\cd \im \Phi
(\l)\geq 0 $ and $\Phi ^*(\l)= \Phi (\ov \l), \;
\l\in\bC\setminus\bR$.
\end{definition}
\subsection{Symmetric and self-adjoint linear relations}\label{sub2.1a}
A linear relation $A\in\C (\gH)$ is called symmetric
(self-adjoint) if $A\subset A^*$ (resp. $A=A^*$). For each
symmetric relations $A\in \C (\gH)$ the following decompositions
hold
\begin {equation}\label{2.0.1}
\gH=\gH_0\oplus \mul A, \qquad A=\rm{gr}\, A'\oplus\wh {\mul} A,
\end{equation}
where $\wh {\mul} A =\{0\}\oplus \mul A$ and $A'$ is a closed
symmetric (not necessarily densely defined)  operator in $\gH_0$
(the operator part of $A$); moreover, $A=A^* $ if and only if
$A'=(A')^*$.

A spectral function $E(\cd):\bR\to [\gH]$ of the relation
$A=A^*\in \C (\gH)$ is defined as $E(t)=E'(t)P_{\gH_0}$, where
$E'(\cd):\bR\to [\gH_0]$ is the orthogonal spectral function of
$A'$

For an operator $A=A^*$ we denote by $\s (A), \; \s_{ac}(A)$ and
$\s_s(A)$ the spectrum, the absolutely continuous spectrum and the
singular spectrum of $A$ respectively  \cite[Section 10.1]{Kat}.

Next assume that $\cH$ is a finite dimensional Hilbert space. A
non-decreasing operator function $\Si(\cd): \bR\to [\cH]$ is
called a distribution function if it is left continuous and
satisfies  $\Si(0)=0$. For each distribution function $\Si(\cd)$
there is a unique pair of distribution functions $\Si_{ac}(\cd)$
and $\Si_s (\cd)$ such that $\Si_{ac}(\cd)$ is absolutely
continuous, $\Si_s (\cd)$ is singular and
$\Si(t)=\Si_{ac}(t)+\Si_s(t)$ (the Lebesgue decomposition of
$\Si$). For  a distribution function $\Si(\cd)$ we denote by
$S(\Si)$ the set of all $t\in\bR$ such that $\Si(t-\delta)\neq
\Si(t+\delta)$ for any $\delta>0$. Moreover, we let
$S_{ac}(\Si)=S(\Si_{ac})$ and $S_{s}(\Sigma)=S(\Si_s)$.

With each distribution function $\Si(\cd)$ one associates the
Hilbert space $L^2(\Si;\cH)$ of all functions $f(\cd):\bR\to \cH$
such that $\int\limits_\bR (d\Si(t)f(t),f(t))<\infty $ (for more
details see e.g. \cite[Section 13.5]{DunSch}). The following
theorem is well known.
\begin{theorem}\label{th2.1}
Let $\Si(\cd)$ be a $[\cH]$-valued distribution function. Then the
relations
\begin {equation*}
\dom \Lambda_\Si=\{f\in L^2(\Si;\cH):t f(t)\in L^2(\Si;\cH)\}
,\;\;\; (\Lambda_\Si f)(t)=tf(t), \;\; f\in\dom\Lambda_\Si
\end{equation*}
define a self-adjoint operator $\L=\L_\Si$ in $L^2(\Si;\cH)$ (the
multiplication operator) and
\begin {equation*}
\s(\L_\Si)=S(\Si), \quad \s_{ac}(\L_\Si)=S_{ac}(\Si), \quad
\s_{s}(\L_\Si)=S_{s}(\Si).
\end{equation*}
\end{theorem}

\subsection{The class $\wt R_-(\cH_0,\cH_1)$}\label{sub2.2}
Let $\cH_0$ be a Hilbert space, let $\cH_1$ be a subspace in
$\cH_0$ and let $\tau =\{\tau_+,\tau_-\}$ be a collection of
holomorphic functions $\tau_\pm(\cd):\bC_\pm\to\CA$. In the paper
we systematically deal with collections $\pair$ of the special
class $\wt R_-(\cH_0,\cH_1)$ introduced in \cite{Mog13.2}. In the
case $\cH_0=\cH_1=:\cH$ this class turns into the well known class
$\wt R(\cH)$ of Nevanlinna functions $\tau(\cd):\CR\to \C (\cH)$
(see for instance \cite{DM00}). If $\dim\cH_0<\infty$, then
according to \cite{Mog13.2} the collection $\pair\in \wt
R_-(\cH_0,\cH_1)$ admits the representation
\begin {equation}\label{2.1}
\tau_+(\l)=\{(C_0(\l),C_1(\l));\cH_1\}, \;\;\l\in\bC_+; \;\;\;\;
\tau_-(\l)=\{(D_0(\l),D_1(\l));\cH_0\}, \;\;\l\in\bC_-
\end{equation}
by means of holomorphic operator pairs
\begin {equation*}
(C_0(\l),C_1(\l)):\cH_0\oplus\cH_1\to\cH_1, \;\;\l\in\bC_+;
\;\;\;\;(D_0(\l),D_1(\l)):\cH_0\oplus\cH_1\to\cH_0, \;\;\l\in\bC_-
\end{equation*}
(more precisely, by equivalence classes of such pairs). The
equalities \eqref{2.1} mean that
\begin {equation}\label{2.2}
\begin{array}{c}
\tau_+(\l)=\{\{h_0,h_1\}\in\cH_0\oplus\cH_1:
C_0(\l)h_0+C_1(\l)h_1=0\},
\;\;\;\l\in\bC_+\\
\tau_-(\l)=\{\{h_0,h_1\}\in\cH_0\oplus\cH_1:
D_0(\l)h_0+D_1(\l)h_1=0\}, \;\;\;\l\in\bC_-.
\end{array}
\end{equation}
In \cite{Mog13.2} the class $\wt R_-(\cH_0,\cH_1)$ is
characterized both in terms of $\CA$-valued functions
$\tau_\pm(\cd)$ and in terms of operator functions $C_j(\cd)$ and
$D_j(\cd), \; j\in\{0,1\},$ from \eqref{2.1}.

\subsection{Boundary triplets and Weyl functions}
Here we recall some  definitions and results from our paper
\cite{Mog13.2}.

 Let $A$ be a closed  symmetric linear relation in the Hilbert space $\gH$,
let $\gN_\l(A)=\ker (A^*-\l)\; (\l\in\wh\rho (A))$ be a defect
subspace of $A$, let $\wh\gN_\l(A)=\{\{f,\l f\}:\, f\in
\gN_\l(A)\}$ and let $n_\pm (A):=\dim \gN_\l(A)\leq\infty, \;
\l\in\bC_\pm,$ be deficiency indices of $A$. Denote by $\exa$ the
set of all proper extensions of $A$, i.e., the set of all
relations $\wt A\in \C (\gH)$ such that $A\subset\wt A\subset
A^*$.

Next assume that $\cH_0$ is a Hilbert space,  $\cH_1$ is a
subspace in $\cH_0$ and $\cH_2:=\cH_0\ominus\cH_1$, so that
$\cH_0=\cH_1\oplus\cH_2$. Denote by $P_j$ the orthoprojector  in
$\cH_0$ onto $\cH_j,\; j\in\{1,2\} $.
\begin{definition}\label{def2.5}
 A collection $\Pi_-=\bta$, where
$\G_j: A^*\to \cH_j, \; j\in\{0,1\},$ are linear mappings, is
called a boundary triplet for $A^*$, if the mapping $\G :\wh f\to
\{\G_0 \wh f, \G_1 \wh f\}, \wh f\in A^*,$ from $A^*$ into
$\cH_0\oplus\cH_1$ is surjective and the following Green's
identity
\begin {equation}\label{2.10}
(f',g)-(f,g')=(\G_1  \wh f,\G_0 \wh g)_{\cH_0}- (\G_0 \wh f,\G_1
\wh g)_{\cH_0}-i (P_2\G_0 \wh f,P_2\G_0 \wh g)_{\cH_2}
\end{equation}
 holds for all $\wh
f=\{f,f'\}, \; \wh g=\{g,g'\}\in A^*$.
\end{definition}
\begin{proposition}\label{pr2.7}
Let  $\Pi_-=\bta$ be a boundary triplet for $A^*$. Then:
\begin{enumerate}
\item
$\;\dim \cH_1=n_+(A)\leq n_-(A)=\dim \cH_0$.
\item
$\ker \G_0\cap\ker\G_1=A$ and $\G_j$ is a bounded operator from
$A^*$ into $\cH_j, \;  j\in\{0,1\}$.
\item
The equality
\begin {equation}\label{2.11}
A_0:=\ker \G_0=\{\wh f\in A^*:\G_0 \wh f=0\}
 \end{equation}
defines the maximal symmetric extension $A_0\in\exa$ such that
$\bC_-\subset \rho (A_0) $.
\end{enumerate}
\end{proposition}
\begin{proposition}\label{pr2.8}
Let  $\Pi_-=\bta$ be a boundary triplet for $A^*$ (so that in view
of Proposition \ref{pr2.7}, (1) $n_+(A)\leq n_-(A)$). Denote also
by $\pi_1$ the orthoprojector in $\gH\oplus\gH$ onto $\gH\oplus
\{0\}$. Then:

{\rm (1)} The operators $P_1\G_0\up \wh \gN_\l (A), \;\l\in\bC_+,$
and $\G_0\up \wh \gN_z (A),\; z\in\bC_-,$ isomorphically map
$\wh\gN_\l (A)$ onto $\cH_1$ and $\wh\gN_z(A)$ onto $\cH_0$
respectively. Therefore the equalities
\begin {equation}\label{2.12}
\begin{array}{c}
\g_{+} (\l)=\pi_1(P_1\G_0\up\wh \gN_\l (A))^{-1}, \;\;\l\in\Bbb C_+,\\
 \g_{-} (z)=\pi_1(\G_0\up\wh\gN_z (A))^{-1}, \;\; z\in\Bbb C_-,
\end{array}
\end{equation}
\begin{gather}
M_{+}(\l)h_1=(\G_1-iP_2\G_0)\{\g_+(\l)h_1, \l\g_+(\l)h_1\}, \quad
h_1\in\cH_1, \quad
\l\in\bC_+\label{2.13}\\
M_{-}(z)h_0=\G_1\{\g_-(z)h_0, z\g_-(z)h_0\}, \quad h_0\in\cH_0,
\quad z\in\bC_- \label{2.14}
\end{gather}
correctly define the operator functions $\g_{+}(\cdot):\Bbb
C_+\to[\cH_1,\gH], \; \; \g_{-}(\cdot):\Bbb C_-\to[\cH_0,\gH]$ and
$M_{+}(\cdot):\bC_+\to [\cH_1,\cH_0], \;\; M_{-}(\cdot):\bC_-\to
[\cH_0,\cH_1]$, which are holomorphic on their domains. Moreover,
the equality $M_+^*(\ov\l)=M_-(\l), \;\l\in\bC_-,$ is valid.

{\rm (2)} Assume that
\begin{gather*}
M_+(\l)=(M(\l),N_+(\l))^\top:\cH_1\to\cH_1\oplus\cH_2 , \quad
\l\in\bC_+\\
M_-(z)=(M(z),N_-(z)):\cH_1\oplus\cH_2\to \cH_1, \quad z\in\bC_-
\end{gather*}
are the block representations of $M_+(\l)$ and $M_-(z)$
respectively and let
\begin{gather}
\cM(\l)=\begin{pmatrix}M(\l) & 0 \cr N_+(\l) & \tfrac i 2
I_{\cH_2}
\end{pmatrix}:\cH_1\oplus\cH_2\to \cH_1\oplus\cH_2, \quad
\l\in\bC_+\label{2.17}\\
\cM(\l)=\begin{pmatrix}M(\l) & N_-(\l) \cr  0  & -\tfrac i 2
I_{\cH_2}
\end{pmatrix}:\cH_1\oplus\cH_2\to \cH_1\oplus\cH_2, \quad
\l\in\bC_-.\label{2.18}
\end{gather}
Then $\cM(\cd)$ is a Nevanlinna operator function satisfying the
identity
\begin {equation}\label{2.18a}
\cM(\mu)-\cM^*(\l)=(\mu-\ov\l)\g_-^*(\l)\g_-(\mu), \quad \mu,\l\in
\bC_- .
\end{equation}
\end{proposition}
\begin{definition}\label{def2.9}$\,$\cite{Mog13.2}
The operator functions $\g_\pm(\cd)$ and $M_\pm(\cd)$  are called
the $\g$-fields and the Weyl functions, respectively,
corresponding to the boundary triplet $\Pi_-$.
\end{definition}
It follows from \eqref{2.12} that for each $h_1\in\cH_1$ and
$h_0\in\cH_0$ the following equalities hold
\begin {equation}\label{2.19}
P_1\G_0\{\g_+(\l)h_1,\l \g_+(\l)h_1\}=h_1,  \qquad
\G_0\{\g_-(z)h_0, z \g_-(z)h_0\}=h_0.
\end{equation}
\begin{proposition}\label{pr2.10a}
Let $\Pi_-=\bta$ be a boundary triplet for $A^*$ and let
$\g_\pm(\cd)$ and $M_\pm(\cd)$ be the corresponding $\g$-fields
and Weyl functions respectively. Moreover, let the spaces $\cH_0$
and $\cH_1$ be decomposed as
\begin {equation*}
\cH_1=\wh\cH\oplus\dot\cH_1, \qquad \cH_0=\wh\cH\oplus\dot\cH_0
\end{equation*}
(so  that $\dot\cH_0=\dot\cH_1\oplus\cH_2$) and let
\begin {equation*}
\G_0=( \wh\G_0, \dot\G_0)^\top : A^* \to \wh\cH\oplus\dot\cH_0,
\qquad \G_1=( \wh\G_1, \dot\G_1 )^\top: A^* \to
\wh\cH\oplus\dot\cH_1
\end{equation*}
be the block representations of the operators $\G_0$ and $\G_1$.
Then:

{\rm (1)} The equalities
\begin {gather*}
\wt A=\{\wh f\in A^*: \wh\G_0 \wh f=\dot\G_0\wh f= \dot\G_1\wh
f=0\},\quad \wt A^*=\{\wh f\in A^*: \wh\G_0 \wh f=0\}
\end{gather*}
define a closed symmetric extension $\wt A\in\exa$ and its adjoint
$\wt A^*$.

{\rm(2)} The collection $\dot\Pi_-=\{\dot\cH_0\oplus\dot\cH_1,
\dot \G_0\up \wt A^*, \dot \G_1\up \wt A^*\}$  is a boundary
triplet for $\wt A^*$.

{\rm (3)}  The $\g$-fields $\dot\g_\pm(\cd)$ and the Weyl
functions $\dot M_\pm (\cd)$ corresponding to $\dot\Pi_-$ are
given by
\begin{gather*}
\dot\g_+(\l)=\g_+(\l)\up \dot\cH_1, \qquad \dot
M_+(\l)=P_{\dot\cH_0}M_+(\l)\up
\dot\cH_1, \quad \l\in\bC_+\\
\dot\g_-(\l)=\g_-(\l)\up \dot\cH_0, \qquad \dot
M_-(\l)=P_{\dot\cH_1}M_-(\l)\up \dot\cH_0, \quad \l\in\bC_-.
\end{gather*}
\end{proposition}
The proof of Proposition \ref{pr2.10a} is similar to that of
Proposition 4.1 in \cite{DM00}.
\begin{remark}\label{rem2.10b}
If $\cH_0=\cH_1:=\cH$, then   the boundary triplet in the sense of
Definition \ref{def2.5} turns into the boundary triplet
$\Pi=\{\cH,\G_0,\G_1\}$ for $A^*$ in the sense of
\cite{GorGor,Mal92}. In this case $n_+(A)=n_-(A)(=\dim \cH)$ and
the $\g$-fields $\g_\pm(\cd)$ and Weyl functions $M_\pm(\cd)$ turn
into the $\g$-field $\g(\cd)$ and  Weyl function $M(\cd)$
respectively introduced in \cite{DM91,Mal92}. Observe also that
along with $\Pi_-$ we define in \cite{Mog06.2,Mog13.2}  a boundary
triplet $\Pi_+=\bta$ for $A^*$. Such a triplet is applicable to
symmetric relations $A$ with $n_-(A)\leq n_+(A)$.
\end{remark}
\subsection{Generalized resolvents and exit space extensions}\label{sub2.4}
The following definitions are well known.
\begin{definition}\label{def2.11}
Let $\wt\gH$ be a Hilbert space and let $\gH$ be a subspace in
$\wt \gH$. A relation $\wt A=\wt A^* \in \C (\wt\gH)$ is  called
$\gH$-minimal if $\text{span} \{\gH,(\wt A-\l)^{-1}\gH:
\l\in\CR\}=\wt\gH$.
\end{definition}
\begin{definition}\label{def2.12}
The relations $T_j\in \C (\gH_j), \; j\in\{1,2\},$ are said to be
unitarily equivalent (by means of a unitary operator $U\in
[\gH_1,\gH_2]$) if $T_2=\wt U T_1$ with $\wt U=U\oplus U \in
[\gH_1^2, \gH_2^2]$.
\end{definition}
\begin{definition}\label{def2.13}
Let $A$ be a symmetric relation in a Hilbert space $\gH$.  The
operator functions $R(\cd):\CR\to [\gH]$ and $F(\cd):\bR\to [\gH]$
are called the generalized resolvent and the spectral function of
$A$ respectively if there exist a Hilbert space $\wt
\gH\supset\gH$ and a self-adjoint  relation $\wt A\in \C (\wt\gH)$
such that $A\subset \wt A$ and the following equalities hold:
\begin {gather}
R(\l) =P_\gH (\wt A- \l)^{-1}\up \gH, \quad \l \in \CR\label{2.20}\\
F(t)=P_{\gH}E(t)\up\gH, \quad  t\in\bR\label{2.21}
\end{gather}
(in formula \eqref{2.21} $E(\cd)$ is the  spectral function of
$\wt A$).

The relation $\wt A$ in \eqref{2.20} is called an exit space
extension of $A$.
\end{definition}
%It follows from \eqref{2.20} and \eqref{2.21} that the generalized resolvent
%$R(\cd)$ and the spectral function $F(\cd)$ generated by the same extension
%$\wt A$ of $A$ are connected by
%\begin {equation}\label{2.22}
%R(\l)=\int_{\bR}\frac {d F(t)} {t-\l}, \quad \l\in\bR.
%\end{equation}
According to \cite{LanTex77} each generalized resolvent  of $A$ is
generated by some $\gH$-minimal exit space extension  $\wt A$ of
A. Moreover,  if the $\gH$-minimal exit space  extensions $\wt
A_1\in\C (\wt\gH_1)$ and $\wt A_2\in\C (\wt\gH_2)$ of $A$ induce
the same generalized resolvent $R(\l)$, then
 there exists a unitary operator $ V'\in
[\wt\gH_1\ominus \gH, \wt\gH_2\ominus \gH]$ such that  $\wt A_1$
and $\wt A_2$ are unitarily equivalent by means of $\wt
V=I_{\gH}\oplus V'$.  By using these facts we suppose in the
following that the exit space extension $\wt A$ in \eqref{2.20} is
$\gH$-minimal, so that  $\wt A$  is defined by $R(\cd)$ uniquely
up to the unitary equivalence. Note that in this case the equality
$\mul\wt A=\mul A$ holds for any exit space extension $\wt A=\wt
A^*$ of $A$ if and only if $\mul A=\mul A^*$ or, equivalently, if
and only if the operator $A'$ (the operator part of $A$) is
densely defined.

\section{Boundary triplets for symmetric systems}
\subsection{Notations}
Let $\cI=[ a,b\rangle\; (-\infty < a< b\leq\infty)$ be an interval
of the real line (the symbol $\rangle$ means that the endpoint
$b<\infty$  might be either included  to $\cI$ or not). For a
given finite-dimensional Hilbert space $\bH$ denote by $\AC$ the
set of functions $f(\cd):\cI\to \bH$ which are absolutely
continuous on each segment $[a,\b]\subset \cI$ and let $AC
(\cI):=AC(\cI;\bC)$.

Next assume that $\D(\cd)$ is a $[\bH]$-valued Borel functions on
$\cI$ integrable on each compact interval $[a,\b]\subset \cI$ and
such that $\D(t)\geq 0$. Denote  by $\lI$  the semi-Hilbert  space
of  Borel functions $f(\cd): \cI\to \bH$ satisfying
$||f||_\D^2:=\int\limits_{\cI}(\D (t)f(t),f(t))_\bH \,dt<\infty$
(see e.g. \cite[Chapter 13.5]{DunSch}).  The semi-definite inner
product $(\cd,\cd)_\D$ in $\lI$ is defined by $
(f,g)_\D=\int\limits_{\cI}(\D (t)f(t),g(t))_\bH \,dt,\; f,g\in
\lI$. Moreover, let $\LI$ be the Hilbert space of the equivalence
classes in $\lI$ with respect to the semi-norm $||\cd||_\D$ and
let $\pi$ be the quotient map from $\lI$ onto $\LI$.

For a given finite-dimensional Hilbert space $\cK$ denote by
$\lo{\cK}$ the set of all Borel operator-functions $F(\cd): \cI\to
[\cK,\bH]$ such that $F(t)h\in \lI$ for each $h\in\cK$. It is
clear that the latter condition is equivalent to
$\int\limits_{\cI}||\D^{\frac 1 2} (t)F(t)||^2\,dt <\infty$.

\subsection{Symmetric systems}
In this subsection we provide some known results on symmetric
systems of differential equations.

Let $H$ and $\wh H$ be  finite-dimensional Hilbert spaces and let
\begin {equation}\label{3.0.1}
H_0=H\oplus\wh H, \quad \bH=H_0\oplus  H=H\oplus\wh H \oplus H.
\end{equation}
In the following we put
\begin {equation}\label{3.0.2}
\nu_+:=\dim H, \quad \wh\nu:=\dim \wh H, \quad \nu_-:=\dim
H_0=\nu_+ +\wh \nu, \quad n:=\dim \bH=\nu_+ + \nu_-.
\end{equation}

Let as above $\cI=[ a,b\rangle\; (-\infty < a< b\leq\infty)$ be an
interval in $\bR$ . Moreover, let $B(\cd)$ and $\D(\cd)$ be
$[\bH]$-valued Borel functions on $\cI$ integrable on each compact
interval $[a,\b]\subset \cI$ and satisfying $B(t)=B^*(t)$ and
$\D(t)\geq 0$ a.e. on $\cI$ and let $J\in [\bH]$ be  operator
\eqref{1.3}.

A first-order symmetric  system on an interval $\cI$ (with the
regular endpoint $a$) is a system of differential equations of the
form
\begin {equation}\label{3.1}
J y'(t)-B(t)y(t)=\D(t) f(t), \quad t\in\cI,
\end{equation}
where $f(\cd)\in \lI$. Together with \eqref{3.1} we consider also
the homogeneous  system
\begin {equation}\label{3.2}
J y'(t)-B(t)y(t)=\l \D(t) y(t), \quad t\in\cI, \quad \l\in\bC.
\end{equation}
A function $y\in\AC$ is a solution of \eqref{3.1} (resp.
\eqref{3.2}) if equality \eqref{3.1} (resp. \eqref{3.2} holds a.e.
on $\cI$. Moreover, a function $Y(\cd,\l):\cI\to [\cK,\bH]$ is an
operator solution of  equation \eqref{3.2} if $y(t)=Y(t,\l)h$ is a
(vector) solution of this equation for every $h\in\cK$ (here $\cK$
is a Hilbert space with $\dim\cK<\infty$).

In what follows  we always assume  that  system \eqref{3.1} is
definite in the sense of the following definition.
 \begin{definition}\label{def3.1}$\,$\cite{GK,KogRof75}
Symmetric system \eqref{3.1} is called definite if for each
$\l\in\bC$ and each solution $y$ of \eqref{3.2} the equality
$\D(t)y(t)=0$ (a.e. on $\cI$) implies $y(t)=0, \; t\in\cI$.
\end{definition}

As it is known \cite{Orc, Kac03, LesMal03} symmetric system
\eqref{3.1} gives rise to the \emph{maximal linear relations}
$\tma$ and $\Tma$  in  $\lI$ and $\LI$, respectively. They are
given by
\begin {equation}\label{3.4}
\begin{array}{c}
\tma=\{\{y,f\}\in(\lI)^2 :y\in\AC \;\;\text{and}\;\; \qquad\qquad\qquad\qquad \\
\qquad\qquad\qquad\qquad\qquad\quad  J y'(t)-B(t)y(t)=\D(t)
f(t)\;\;\text{a.e. on}\;\; \cI \}
\end{array}
\end{equation}
and $\Tma=\{\{\pi y,\pi f\}:\{y,f\}\in\tma\}$. Moreover the
Lagrange's identity
\begin {equation}\label{3.6}
(f,z)_\D-(y,g)_\D=[y,z]_b - (J y(a),z(a)),\quad \{y,f\}, \;
\{z,g\} \in\tma.
\end{equation}
holds with
\begin {equation}\label{3.7}
[y,z]_b:=\lim_{t \uparrow b}(J y(t),z(t)), \quad y,z \in\dom\tma.
\end{equation}
Formula \eqref{3.7} defines the skew-Hermitian  bilinear form
$[\cd,\cd]_b $ on $\dom \tma$, which plays a crucial  role in our
considerations. By using this form we define the \emph{minimal
relations} $\tmi$ in $\lI$ and $\Tmi$ in $\LI$ via
\begin {equation*}
\tmi=\{\{y,f\}\in\tma: y(a)=0 \;\; \text{and}\;\;
[y,z]_b=0\;\;\text{for each}\;\; z\in \dom \tma \}.
\end{equation*}
and $\Tmi=\{\{\pi y,\pi f\}:\{y,f\}\in\tmi\} $. According to
\cite{Orc,Kac03, LesMal03} $\Tmi$ is a closed symmetric linear
relation in $\LI$ and $\Tmi^*=\Tma$.
\begin{remark}\label{rem3.1a}
It  is known (see e.g. \cite{LesMal03}) that the maximal relation
$\Tma$ induced by the definite symmetric system \eqref{3.1}
possesses the following
 property: for any $\{\wt y, \wt f \}\in \Tma $ there exists
a unique function $y\in \AC \cap \lI $ such that $y\in \wt y$ and
$\{y,f\}\in \tma$ for any $f\in\wt f$. Below we associate such a
function $ y\in \AC \cap \lI$ with each pair $\{\wt y, \wt
f\}\in\Tma$.
\end{remark}

For any  $\l\in\bC$ denote by $\cN_\l$ the linear space of
solutions of the homogeneous system \eqref{3.2} belonging to
$\lI$. Definition \eqref{3.4} of $\tma$ implies
\begin{equation*}
\cN_\l=\ker (\tma-\l)=\{y\in\lI:\; \{y,\l y\}\in\tma\},
\quad\l\in\bC,
\end{equation*}
and hence $\cN_\l\subset \dom\tma$. As usual, denote by $ n_\pm
(\Tmi ):=\dim \gN_\l (\Tmi), \quad \l\in\bC_\pm,$ the  deficiency
indices of  $\Tmi$.  Since the system \eqref{3.1} is definite, $
\pi\cN_\l=\gN_\l (\Tmi)$ and $\ker(\pi\up\cN_\l)=\{0\},\;\;
\l\in\bC$.  This implies that $\dim \cN_\l=n_\pm (\Tmi), \;
\l\in\bC_\pm$.

The following lemma is obvious.
\begin{lemma}\label{lem3.2}
 If $Y(\cd,\l)\in \lo{\cK}$ is an operator solution  of Eq.
\eqref{3.2}, then the relation
\begin {equation}\label{3.8}
\cK\ni h\to (Y(\l) h)(t)=Y(t,\l)h \in\cN_\l.
\end{equation}
defines the linear mapping $Y(\l):\cK\to \cN_\l$ and, conversely,
for each such a mapping $Y(\l)$ there exists a unique operator
solution $Y(\cd,\l)\in \lo{\cK}$ of Eq.  \eqref{3.2} such that
\eqref{3.8} holds.
\end{lemma}

Next assume that
\begin {equation} \label{3.17.1}
U=\begin{pmatrix} u_1 & u_2 & u_3  \cr u_4 & u_5 & u_6
\end{pmatrix}: H\oplus\wh H\oplus H\to \wh H\oplus H
\end{equation}
is the operator  satisfying the relations
\begin{gather}
\ran U=\wh H\oplus H \label{3.17.2}\\
iu_2u_2^* - u_1u_3^*+u_3u_1^*=iI_{\wh H}, \qquad iu_5u_2^* -
u_4u_3^*+u_6u_1^*=0\label{3.17.3}\\
iu_5u_5^* + u_6u_4^*-u_4u_6^*=0\label{3.17.4}
\end{gather}
One can prove that the operator \eqref{3.17.1} admits an extension
to the $J$-unitary operator
\begin {equation} \label{3.17.5}
\wt U=\begin{pmatrix} u_7 & u_8 & u_9 \cr \hline  u_1 & u_2 & u_3
\cr u_4 & u_5 & u_6
\end{pmatrix}: H\oplus\wh H\oplus H\to H\oplus \wh H\oplus H,
\end{equation}
i.e. the operator satisfying  $\wt U^* J\wt U=J$.

In view of \eqref{3.0.1} each function $y\in\AC$ admits the
representation
\begin {equation}\label{3.18}
y(t)=\{y_0(t),\,\wh y(t), \, y_1(t) \}(\in H\oplus\wh H \oplus H),
\quad t\in\cI.
\end{equation}
Using \eqref{3.17.5} and the representation \eqref{3.18} of $y$ we
introduce the linear mappings $\G_{ja}:\AC\to H, \; j\in \{0,1\},$
and $\wh\G_{a}:\AC\to \wh H$ by setting
\begin{gather}
\G_{0a}y=u_7 y_0(a)+ u_8 \wh y(a)+ u_9 y_1(a),\quad y\in\AC\label{3.23}\\
\wh\G_{a}y=u_1 y_0(a)+ u_2 \wh y(a)+ u_3 y_1(a),\quad \G_{1a}y=u_4
y_0(a)+ u_5 \wh y(a)+ u_6 y_1(a). \label{3.24}
\end{gather}
Clearly, the mappings  $\wh\G_a$ and $\G_{1a}$ are determined by
the operator $U$, while $\G_{0a}$ is determined by the extension
$\wt U$. Moreover, the mapping
\begin {equation}\label{3.24a}
\G_a:=\left(\G_{0a},\, \wh\G_a,\, \G_{1a}\right)^\top :\AC\to
H\oplus \wh H\oplus H
\end{equation}
satisfies $\G_ay=\wt U y(a), \; y\in\AC$. Hence $\G_a$ is
surjective and
\begin {equation}\label{3.24b}
(J y(a),z(a))= -(\G_{1a}y,\G_{0a}z)+ (\G_{0a}y,\G_{1a}z)+i
(\wh\G_a y,\wh\G_a z), \quad y,z\in\AC.
\end{equation}

In what follows we associate with each operator $U$ (see
\eqref{3.17.1}) the operator solution $\f
(\cd,\l)=\f_U(\cd,\l)(\in [H_0,\bH]),\;\l\in\bC,$ of Eq.
\eqref{3.2} with the initial data
\begin {equation}\label{3.25}
\f_U(a,\l)= \begin{pmatrix} u_6^*  & iu_3^* \cr -iu_5^* & u_2^*
\cr -u_4^*  & -iu_1^*\end{pmatrix}:\underbrace { H\oplus\wh
H}_{H_0}\to \underbrace{H\oplus \wh H\oplus H}_{\bH}.
\end{equation}
One can easily verify that for each $J$-unitary extension $\wt U$
of $U$ one has
\begin {equation}\label{3.26}
\wt U \f_U(a,\l)=\begin{pmatrix}I_{H_0} \cr 0
\end{pmatrix}:H_0\to H_0\oplus H.
\end{equation}
The particular case of the operator $U$ and its $J$-unitary
extension $\wt U$ is (cf. \cite{HinSch06})
\begin {equation}\label{3.26a}
U= \begin{pmatrix}  0& I_{\wh H} & 0 \cr  \cos B & 0 & \sin
B\end{pmatrix}, \quad  \wt U=\begin{pmatrix} \sin B & 0 & -\cos B
\cr 0& I_{\wh H} & 0 \cr \cos B & 0 & \sin B\end{pmatrix},
\end{equation}
where $B=B^*\in [H]$. For such $U$ the  solution $\f_U(\cd,\l)$
is defined by the initial data
\begin {equation*}
\f_U(a,\l)=\begin{pmatrix} \sin B & 0  \cr 0& I_{\wh H}  \cr -\cos
B & 0
\end{pmatrix}: H\oplus\wh H\to H\oplus
\wh H\oplus H.
\end{equation*}
\subsection{Decomposing boundary triplets}\label{sub3.3}
According to \cite{BHSW10,Mog12} the  skew-Hermitian bilinear form
\eqref{3.7} has finite indices of inertia  $\nu_{b+} $ and
$\nu_{b-}$ and
\begin {equation} \label{3.27}
n_+(\Tmi)=\nu_+ +\nu_{b+}, \qquad n_-(\Tmi)=\nu_- +\nu_{b-}=\nu_+
+\wh \nu+ \nu_{b-}
\end{equation}
(for  $\nu_\pm$ see \eqref{3.0.2}). Moreover, the following lemma
is immediate from \cite[Lemma 5.1]{Mog12}.
\begin{lemma}\label{lem3.2a}
For any pair of finite-dimensional Hilbert spaces $\cH_b$ and
$\wh\cH_b$ with $\dim\cH_b=\min\{\nu_{b+},\nu_{b-}\}$ and $\dim
\wh\cH_b=|\nu_{b+}- \nu_{b-}|$ there exists a  surjective linear
mapping
\begin{gather}\label{3.29}
\G_b=(\G_{0b},  \wh\G_b,  \G_{1b})^\top:\dom\tma\to
\cH_b\oplus\wh\cH_b\oplus \cH_b
\end{gather}
such that for all $y,z \in \dom\tma$ the following equality is
valid
\begin{gather}\label{3.30}
[y,z]_b=i\cdot\sign (\nu_{b+}-\nu_{b-}) (\wh\G_b y, \wh\G_b
z)-(\G_{1b}y,\G_{0b}z)+(\G_{0b}y,\G_{1b}z).
\end{gather}
\end{lemma}
Note that $\G_b y$ is in fact a singular boundary value of a
function $y\in\dom\tma$ in the sense of \cite[Chapter
13.2]{DunSch} (for more details see Remark 3.5 in \cite{Mog13.1}).

Assume that $n_+(\Tmi) < n_-(\Tmi)$. Then by \eqref{3.27} and
\eqref{3.0.2} $\wh\nu > \nu_{b+}-\nu_{b-}$ and hence the following
two alternative cases may hold:

\underline {\emph{ Case 1}}. $\;\; \wh\nu >\nu_{b+}-\nu_{b-}>0$.

\underline {\emph{ Case 2}}. $\;\; \wh\nu \geq 0 \geq
\nu_{b+}-\nu_{b-}$ and $\wh\nu\neq \nu_{b+}-\nu_{b-} (\neq 0)$.

Below we construct a boundary triplet for $\Tma$ separately in
each of these cases.

In \emph{ Case 1} one has $\dim \wh H(=\wh\nu)> \nu_{b+}-\nu_{b-}>
0$. Let $\wh H_1$ be a subspace in $\wh H$ with $\dim\wh H_1=
\nu_{b+}-\nu_{b-}$ and let $\cH_b$ be a Hilbert space with $\dim
\cH_b=\nu_{b-}$. Then by Lemma \ref{lem3.2a} there exists a
surjective linear mapping
\begin{gather}\label{3.30a}
\G_b=(\G_{0b},  \wh\G_b,  \G_{1b})^\top:\dom\tma\to \cH_b\oplus\wh
H_1\oplus \cH_b
\end{gather}
such that \eqref{3.30} holds for all $y,z \in \dom\tma$. Let $\wh
H$ be decomposed as $\wh H=\wh H_1\oplus\wh H_2$ with $\wh
H_2:=\wh H\ominus\wh H_1$ and let
\begin {equation}\label{3.30b}
\wh\G_a=(\wh\G_{a1}, \, \wh\G_{a2})^\top :\AC\to\wh H_1\oplus\wh
H_2
\end{equation}
be the block representation of the mapping $\wh\G_a$ (see
\eqref{3.24}). Moreover, let $H_0':=H\oplus \wh H_1$, so that in
view of \eqref{3.0.1} $H_0$ admits the representation
\begin {equation}\label{3.31}
H_0=\underbrace {H\oplus \wh H_1}_{H_0'}\oplus\wh H_2= H_0'\oplus
\wh H_2
\end{equation}
In \emph{Case 1} we let
\begin{gather}
\cH_0=H_0'\oplus\wh H_2\oplus\cH_b, \quad \cH_1=H_0'\oplus\cH_b,\label{3.32}\\
\G_0\{\wt y, \wt f\}=\{- \G_{1a}y +i(\wh\G_{a1}-\wh\G_b)y,\, i
\wh\G_{a2}y,\,
\G_{0b}y\} (\in H_0'\oplus \wh H_2 \oplus\cH_b),\label{3.33}\\
\G_1\{\wt y, \wt f\}=\{\G_{0a}y +\tfrac 1
2(\wh\G_{a1}+\wh\G_b)y,\,-\G_{1b}y\} (\in H_0'\oplus \cH_b),\quad
\{\wt y, \wt f\}\in\Tma.\label{3.34}
\end{gather}

Now assume that \emph{Case 2} holds. Let $\cH_b$ and $\wh\cH_b$ be
Hilbert spaces with $\dim\cH_b=\nu_{b+}$ and $\dim
\wh\cH_b=\nu_{b-}-\nu_{b+}$. Then by Lemma \ref{lem3.2a} there is
a surjective linear mapping \eqref{3.29} satisfying \eqref{3.30}.
Let $\wt\cH_b:=\cH_b\oplus \wh\cH_b$ (so that
$\cH_b\subset\wt\cH_b $) and let $\wt\G_{0b}:\dom\tma\to\wt\cH_b$
be the linear mapping given by
\begin {equation}\label{3.35}
\wt\G_{0b}=\G_{0b}+ \wh \G_b.
\end{equation}
In \emph{Case 2} we put
\begin {gather}
\cH_0=H\oplus \wh H \oplus \wt \cH_b, \qquad \cH_1=H \oplus \cH_b\label{3.36}\\
\G_0\{\wt y, \wt f\}=\{- \G_{1a}y,  i \wh\G_a y,\wt\G_{0b} y
\}(\in  H\oplus\wh
H\oplus \wt\cH_b)\label{3.37}\\
\G_1\{\wt y, \wt f\}=\{ \G_{0a}y,  -\G_{1b}y\} (\in
H\oplus\cH_b),\quad \{\wt y, \wt f\}\in\Tma. \label{3.38}
\end{gather}

Note that in both \emph{Cases 1} and \emph{2}  $\cH_1$ is a
subspace in $\cH_0$ and $\G_j$ is an operator from  $\Tma$ to
$\cH_j, \; j\in \{0.1\}$. Moreover,
 $\cH_2(=\cH_0\ominus \cH_1)=\wh H_2$  in \emph{Case
1} and $\cH_2=\wh H\oplus\wh \cH_b$ in \emph{Case 2}.
\begin{proposition}\label{pr3.3}
Assume that  $\wt U$ is  $J$-unitary operator \eqref{3.17.5} and
$\G_{0a}, \; \G_{1a}$ and $\wh\G_a$ are  linear mappings
\eqref{3.23}, \eqref{3.24}. Moreover, let $\G_b$ be surjective
linear mapping given either by\eqref{3.30a} (in Case 1) or
\eqref{3.29} (in Case 2) and satisfying \eqref{3.30}. Then a
collection $\Pi_-=\bta$ defined either  by
\eqref{3.32}--\eqref{3.34} (in {Case 1}) or
\eqref{3.36}--\eqref{3.38} (in {Case 2}) is a boundary triplet for
$\Tma$.
\end{proposition}
\begin{proof}
The immediate calculation with taking \eqref{3.24b} and
\eqref{3.30} into account gives
\begin {gather*}
(\G_1\{\wt y,\wt f\}, \G_0\{\wt z,\wt g\})-(\G_0\{\wt y,\wt f\},
\G_1\{\wt z,\wt g\})-i (P_2 \G_0\{\wt y,\wt f\}, P_2\G_0\{\wt
z,\wt g\} )=\qquad
\qquad\qquad\\
\qquad\qquad \qquad\qquad\qquad\qquad \qquad\qquad=[y,z]_b-(J
y(a), z(a)),\quad \{\wt y,\wt f\},\{\wt z,\wt g\}\in\Tma.
\end{gather*}
This and the Lagrange's identity \eqref{3.6} yield identity
\eqref{2.10} for $\G_0$ and $\G_1$. Moreover, the mapping
$\G=(\G_0,\G_1)^\top$ is surjective, because so are $\G_a$ (see
\eqref{3.24a}) and $\G_b$.
\end{proof}
\begin{definition}\label{def3.4}
The boundary triplet $\Pi_-=\bta$  constructed in Proposition
\ref{pr3.3} is called a decomposing  boundary triplet for $\Tma$.
\end{definition}
\begin{proposition}\label{pr3.5}
Let in \emph{Case 1} $U$ be  operator \eqref{3.17.1} and let
$\wh\G_a$ and $\G_{1a}$ be linear mappings \eqref{3.24}. Moreover,
let $\wh H_1$ be a subspace in $\wh H$, let $\wh H_2=\wh H\ominus
\wh H_1$, let $\wh \G_{aj}, \; j\in\{1,2\},$ be defined by
\eqref{3.30b} and let $\G_b$ be  surjective linear mapping
\eqref{3.30a} satisfying \eqref{3.30}. Then:

{\rm (1)} The equalities
\begin{gather}
T=\{\{\wt y, \wt f\}\in\Tma: \, \G_{1a}y=0, \;\wh\G_{a1} y=\wh\G_b
y,\;
\wh\G_{a2}y=0,\;\G_{0b}y =\G_{1b}y=0 \}\label {3.39}\\
T^*=\{\{\wt y, \wt f\}\in\Tma: \, \G_{1a}y=0, \;\wh\G_{a1}
y=\wh\G_b y \}\label {3.40}
\end{gather}
define a symmetric extension $T$ of $\Tmi$ and its adjoint $T^*$.
Moreover, the deficiency indices of $T$ are $n_+(T)=\nu_{b-}$ and
$n_-(T)=\wh\nu +2\nu_{b-}-\nu_{b+}$.

{\rm (2)} The collection $\dot\Pi_-=\{\dot\cH_0\oplus \cH_b,
\dot\G_0, \dot\G_1\}$ with $\dot\cH_0=\wh H_2\oplus\cH_b$ and the
operators
\begin {equation}\label {3.41}
\dot\G_0 \{\wt y,\wt f\}=\{i\wh \G_{a2} y,\G_{0b}y\}(\in \wh
H_2\oplus\cH_b), \quad \dot\G_1 \{\wt y,\wt f\}=-\G_{1b}y, \quad
\{\wt y,\wt f\}\in T^*,
\end{equation}
is a boundary triplet for $T^*$ and the (maximal symmetric)
relation $A_0(=\ker\dot\G_0)$  is
\begin {equation}\label {3.42}
A_0=\{\{\wt y, \wt f\}\in\Tma: \, \G_{1a}y=0, \;\wh\G_{a1}
y=\wh\G_b y,\; \wh\G_{a2} y=0,\; \G_{0b}y =0 \}.
\end{equation}
\end{proposition}
\begin{proof}
Let $\wt U$ be  $J$-unitary extension \eqref{3.17.5} of $U$, let
$\G_{0a}$ be operator \eqref{3.23} and let $\Pi_-=\bta$ be the
decomposing boundary triplet \eqref{3.32}--\eqref{3.34} for
$\Tma$. Applying to this triplet Proposition \ref{pr2.10a} one
obtains the desired statements.
\end{proof}
\begin{proposition}\label{pr3.6}
Let in Case 2 $U$ be  operator \eqref{3.17.1} and let $\wh\G_a$
and $\G_{1a}$ be  linear mappings \eqref{3.24}. Moreover, let
$\G_b$ be surjective linear mapping \eqref{3.29} satisfying
\eqref{3.30}, let $\wt\cH_b=\cH_b\oplus\wh\cH_b$ and let $\wt
\G_{0b}$ be given by \eqref{3.35}. Then:

{\rm (1)} The equalities
\begin{gather}
T=\{\{\wt y, \wt f\}\in\Tma: \, \G_{1a}y=0, \;\wh\G_{a}
y=0,\;\wt\G_{0b}y =
\G_{1b}y=0 \}\label {3.43}\\
T^*=\{\{\wt y, \wt f\}\in\Tma: \, \G_{1a}y=0 \}\label {3.44}
\end{gather}
define a symmetric extension $T$ of $\Tmi$ and its adjoint $T^*$.
Moreover, $n_+(T)=\nu_{b+}$ and $n_-(T)=\wh\nu +\nu_{b-}$.

{\rm (2)} The collection $\dot\Pi_-=\{\dot\cH_0\oplus \cH_b,
\dot\G_0, \dot\G_1\}$ with $\dot\cH_0=\wh H\oplus\wt\cH_b$ and the
operators
\begin {equation}\label {3.45}
\dot\G_0 \{\wt y,\wt f\}=\{i\wh \G_{a} y,\wt\G_{0b}y\}(\in \wh
H\oplus\wt\cH_b), \quad \dot\G_1 \{\wt y,\wt f\}=-\G_{1b}y, \quad
\{\wt y,\wt f\}\in T^*
\end{equation}
is a boundary triplet for $T^*$ and the (maximal symmetric)
relation $A_0(=\ker\dot\G_0)$  is
\begin {equation}\label {3.46}
A_0=\{\{\wt y, \wt f\}\in\Tma: \, \G_{1a}y=0, \;\wh\G_{a} y=0,\;
\wt\G_{0b}y =0 \}.
\end{equation}
\end{proposition}
We omit the proof of this proposition, because it is similar to
that of Proposition \ref{pr3.5}.
\section{$\cL_\D^2$-solutions of boundary value problems}
\subsection{Basic assumption}\label{sub4.1}
In what follows we  suppose (unless otherwise stated) that  system
\eqref{3.1} satisfies $n_-(\Tmi) < n_+(\Tmi)$ and the following
assumptions are fulfilled:
\begin{itemize}
\item[(A1)]
 $U$ is the operator \eqref{3.17.1} satisfying  \eqref{3.17.2} -
\eqref{3.17.4} and $\wh\G_a$ and $\G_{1a}$ are the linear mappings
\eqref{3.24}.

\item[(A2)]
In \emph {Case 1}  $\wh H$ is decomposed as $\wh H=\wh
H_1\oplus\wh H_2$, $\wh\G_{aj}:\AC\to\wh H_j, \; j\in \{1,2\}, $
are the operators given by \eqref{3.30b}, $\cH_b$ is a
finite-dimensional Hilbert space and $\G_b$ is a surjective
operator \eqref{3.30a} satisfying \eqref{3.30}.

\item[(A3)]
In \emph {Case 2} $\cH_b$ and $\wh\cH_b$ are finite-dimensional
Hilbert spaces, $\G_b$ is a surjective operator \eqref{3.29}
satisfying  \eqref{3.30}, $\wt\cH_b=\cH_b\oplus \wh\cH_b$ and
$\wt\G_{0b}:\dom\tma\to \wt\cH_b$ is the operator \eqref{3.35}.
\end{itemize}

\subsection{Case 1}
\begin{definition}\label{def4.1}
In \Ca{1} a boundary parameter  is a collection $\pair\in\wt
R_-(\wh H_2\oplus\cH_b,\cH_b)$. A truncated boundary parameter is
a boundary parameter $\pair$ satisfying $\tau_-(\l)\subset
(\{0\}\oplus\cH_b)\oplus \cH_b,\; \l\in\bC_-$.
\end{definition}
\noindent According to Subsection \ref{sub2.2} a boundary
parameter $\pair$ admits the representation
\begin {equation}\label{4.1}
\tau_+(\l)=\{(C_0(\l),C_1(\l));\cH_b\}, \;\;\l\in\bC_+; \quad
\tau_-(\l)=\{(D_0(\l),D_1(\l));\wh H_2\oplus\cH_b\},
\;\;\l\in\bC_-
\end{equation}
with holomorphic operator functions $C_0(\l)(\in [\wh
H_2\oplus\cH_b,\cH_b]), \; C_1(\l)(\in [\cH_b])$ and $D_0(\l)(\in
[\wh H_2\oplus\cH_b]), \; D_1(\l)(\in [\cH_b,\wh H_2\oplus
\cH_b])$. Moreover, a truncated boundary parameter $\pair$ admits
the representation \eqref{4.1} with
\begin {equation}\label{4.1.1}
D_0(\l)=\begin{pmatrix} I_{\wh H_2}& 0 \cr  0 & \ov
D_0(\l)\end{pmatrix}:\wh H_2\oplus\cH_b \to \wh H_2\oplus\cH_b,\;
\;\;D_1(\l)=\begin{pmatrix}0 \cr \ov D_1(\l)
\end{pmatrix}:\cH_b\to \wh H_2\oplus\cH_b.
\end{equation}

Let $\pair$ be a boundary parameter \eqref{4.1} and let
\begin {equation}\label{4.1a}
\begin{array}{c}
C_0(\l)=(\wh C_{02}(\l),C_{0b}(\l)):\wh H_2\oplus\cH_b\to \cH_b, \\
D_0(\l)=(\wh D_{02}(\l),D_{0b}(\l)):\wh H_2\oplus\cH_b\to(\wh
H_2\oplus \cH_b)
\end{array}
\end{equation}
be the block representations of $C_0(\l)$ and $D_0(\l)$. For a
given function $f\in\lI$ consider the following boundary value
problem:
\begin{gather}
J y'-B(t)y=\l \D(t)y+\D(t)f(t), \quad t\in\cI,\label{4.2}\\
\G_{1a}y=0, \quad \wh \G_{a1} y= \wh\G_b y,\quad \l\in\CR,\label{4.3}\\
i \wh C_{02}(\l)\wh\G_{a2}y+ C_{0b}(\l)\G_{0b}y
+C_1(\l)\G_{1b}y=0,
\quad \l\in\bC_+ \label{4.4}\\
i \wh D_{02}(\l)\wh\G_{a2}y+ D_{0b}(\l)\G_{0b}y
+D_1(\l)\G_{1b}y=0, \quad \l\in\bC_-.\label{4.5}
\end{gather}
A function $y(\cd,\cd):\cI\tm (\CR)\to\bH$ is called a solution of
this problem if for each $\l\in\CR$ the function $y(\cd,\l)$
belongs to $\AC\cap\lI$ and satisfies the equation \eqref{4.2}
a.e. on $\cI$ (so that $y\in\dom\tma$) and the boundary conditions
\eqref{4.3} -- \eqref{4.5}.
\begin{theorem}\label{th4.2}
Let in \Ca{1} $T$ be a symmetric relation in $\LI$ defined by
\eqref{3.39}. If $\pair$ is a boundary parameter \eqref{4.1}, then
for every $f\in\lI$ the boundary problem \eqref{4.2} - \eqref{4.5}
has a unique solution $y(t,\l)=y_f(t,\l) $ and the equality
\begin {equation*}
R(\l)\wt f = \pi(y_f(\cd,\l)), \quad \wt f\in \LI, \quad f\in\wt
f, \quad \l\in\CR,
\end{equation*}
defines a generalized resolvent $R(\l)=:R_\tau(\l)$ of $T$.
Conversely, for each generalized resolvent $R(\l)$ of $T$ there
exists a unique boundary parameter $\tau$ such that
$R(\l)=R_\tau(\l)$.
\end{theorem}
\begin{proof}
Let $\dot \Pi_-$ be the boundary triplet  for $T^*$ defined in
Proposition \ref{pr3.5}. Applying to this triplet \cite[Theorem
3.11]{Mog13.2} we obtain the required statements.
\end{proof}
\begin{remark}\label{rem4.3}
Let $\tau_0=\{\tau_{+},\tau_{-}\}$ be a boundary parameter
\eqref{4.1}  with
\begin {equation}\label{4.6}
C_0(\l)\equiv P_{\cH_b}, \quad C_1(\l)\equiv 0 , \quad D_0(\l)
\equiv I_{\wh H_2\oplus \cH_b},\quad D_1(\l)\equiv 0,
\end{equation}
 and let $A_0$ be a symmetric relation \eqref{3.42}. Then
\begin {equation}\label{4.7}
R_{\tau_0}(\l)=(A_0^*-\l)^{-1},\;\l\in\bC_+ \;\;\text{and}\;\;
R_{\tau_0}(\l)=(A_0-\l)^{-1},\;\l\in\bC_-.
\end{equation}
\end{remark}
\begin{proposition}\label{pr4.4}
Let in \Ca{1} $P_H, \; P_{\wh H_1} $ and $P_{\wh H_2}$ be the
orthoprojectors in $H_0$ onto $H,\; \wh H_1$ and $\wh H_2$
respectively (see \eqref{3.31}) and let $\wt P_{\wh H_2}$
($P_{\cH_b}$) be the orthoprojector in $\wh H_2\oplus\cH_b$ onto
$\wh H_2$ (resp. $\cH_b$). Then:

{\rm (1)} For any $\l\in\CR$ there exists a unique operator
solution $v_0(\cd,\l)\in\lo{H_0}$ of Eq. \eqref{3.2} satisfying
\begin{gather}
\G_{1a} v_0(\l)=-P_H, \;\;\; i(\wh\G_{a1}-\wh\G_b)v_0(\l)=P_{\wh
H_1},
\;\;\;\G_{0b} v_0(\l)=0,\;\;\; \l\in\CR \label{4.8}\\
i\wh\G_{a2}v_0(\l)=P_{\wh H_2}, \quad \l\in\bC_-\label{4.9}
\end{gather}

{\rm(2)} For any $\l\in\bC_+ \; (\l\in\bC_-)$ there exists a
unique operator solution $u_+(\cd,\l)\in\lo{\cH_b}$ (resp.
$u_-(\cd,\l)\in\lo{\wh H_2\oplus\cH_b}$) of Eq. \eqref{3.2} such
that
\begin{gather}
\G_{1a}u_\pm(\l)=0, \qquad i(\wh\G_{a1}-\wh\G_b)u_\pm(\l)=0,\quad
\l\in\bC_\pm \label{4.10}\\
\G_{0b}u_+(\l)=I_{\cH_b},\;\; \l\in\bC_+ \label{4.11}\\
\G_{0b}u_-(\l)=P_{\cH_b}, \qquad i \wh\G_{a2}u_-(\l)=\wt P_{\wh
H_2},\quad \l\in\bC_-.\label{4.11a}
\end{gather}
In formulas \eqref{4.8}-- \eqref{4.11a} $v_0(\l)$ and $u_\pm(\l)$
are linear mappings from Lemma \ref{lem3.2} corresponding to the
solutions $v_0(\cd,\l)$ and $u_{\pm}(\cd,\l)$ respectively.
\end{proposition}
\begin{proof}
Let $\wt U$ be the $J$-unitary extension \eqref{3.17.5} of $U$,
let $\G_{0a}$ be the operator \eqref{3.23} and let $\Pi_-=\bta$ be
the decomposing boundary triplet \eqref{3.32}-\eqref{3.34} for
$\Tma$. Assume also that $\g_\pm(\cd)$ are the $\g$-fields  of
$\Pi_-$. Since the quotient mapping $\pi$ isomorphically maps
$\cN_\l$ onto $\gN_\l (\Tmi)$, it follows that for every
$\l\in\bC_+\;(\l\in\bC_-)$ there exists an isomorphism
$Z_+(\l):\cH_1\to \cN_\l$ (resp. $Z_-(\l):\cH_0\to \cN_\l$) such
that
\begin {equation}\label{4.12}
\g_+(\l)=\pi Z_+(\l), \;\;\;\l\in\bC_+; \qquad \g_-(\l)=\pi
Z_-(\l), \;\;\;\l\in\bC_-.
\end{equation}
Let $\G_0'$ and $\G_1'$ be the linear mappings given by
\begin {equation}\label{4.13}
\begin{array}{l}
\G_0'=\begin{pmatrix} - \G_{1a} +i(\wh\G_{a1}-\wh\G_b)\cr
i\wh\G_{a2} \cr \G_{0b}
\end{pmatrix}:\dom\tma \to H_0'\oplus \wh H_2\oplus\cH_b,\\
\G_1'=\begin{pmatrix} \G_{0a} + \tfrac 1 2(\wh\G_{a1}+\wh\G_b)\cr
-\G_{1b}
\end{pmatrix}:\dom\tma \to H_0'\oplus\cH_b.
\end{array}
\end{equation}
Then by \eqref{3.33} and \eqref{3.34} one has $ \G_j\{\pi y, \l
\pi y\}=\G_j' y, \; y\in\cN_\l, \; j \in \{0,1\}. $ Combining of
this equality with \eqref{4.12} and \eqref{2.19} gives
\begin{gather}\label{4.15}
P_{H_0'\oplus \cH_b}\G_0' Z_+(\l)=I_{\cH_1},\quad
\l\in\bC_+;\qquad \G_0' Z_-(\l)= I_{\cH_0}, \quad \l\in\bC_-,
\end{gather}
which in view of \eqref{4.13} can be written as
\begin{gather}
\begin{pmatrix}-\G_{1a}+i (\wh\G_{a1}- \wh\G_b) \cr \G_{0b}
\end{pmatrix} Z_+(\l)=\begin{pmatrix} I_{H_0'} & 0 \cr 0 & I_{\cH_b}
\end{pmatrix}, \;\; \l\in\bC_+ \label{4.15.1}\\
\begin{pmatrix}-\G_{1a}+i (\wh\G_{a1}- \wh\G_b) \cr i\wh\G_{a2}\cr \G_{0b}
\end{pmatrix} Z_-(\l)=\begin{pmatrix} I_{H_0'} & 0  & 0 \cr 0 & I_{\wh H_2}&0
\cr 0& 0 & I_{\cH_b}
\end{pmatrix}, \;\; \l\in\bC_- \label{4.15.2}
\end{gather}
It follows from \eqref{4.15.1} and \eqref{4.15.2} that
\begin{gather}
\G_{1a}Z_+(\l)=(-P_H,\; 0),\;\;\;\tfrac 1 2 (\wh\G_{a1}- \wh\G_b)
Z_+(\l)=(-\tfrac i 2 P_{\wh H_1},\; 0), \;\;\;\G_{0b} Z_+(\l)=
(0,\, I_{\cH_b})
\label{4.15.3}\\
\G_{1a}Z_-(\l)=(-P_H,\; 0,\;0),\;\;\;\tfrac 1 2 (\wh\G_{a1}-
\wh\G_b)
Z_-(\l)=(-\tfrac i 2 P_{\wh H_1},\; 0, \; 0)\label{4.15.4}\\
\wh\G_{a2}Z_-(\l)=(0,\; -iI_{\wh H_2}, \; 0),\quad \G_{0b}
Z_-(\l)= (0,\, 0,\, I_{ \cH_b})\label{4.15.5}
\end{gather}
Next assume that
\begin {gather}
Z_+(\l)=(r(\l),\,u_+(\l)):H_0'\oplus\cH_b\to
\cN_\l,\;\;\;\l\in\bC_+\label{4.15.6}\\
 Z_-(\l)=(r(\l),\, \omega_-(\l),\, \wt u_-(\l)):H_0'\oplus\wh H_2\oplus \cH_b\to
\cN_\l,\;\;\;\l\in\bC_-\label{4.15.7}
\end{gather}
are the block representations of $Z_\pm(\l)$ and let
\begin {gather}
v_0(\l):= (r(\l),\, 0):H_0'\oplus\wh H_2\to
\cN_\l,\;\;\;\l\in\bC_+\label{4.15.8}\\
v_0(\l):= (r(\l),\, \omega_-(\l)):H_0'\oplus \wh H_2\to
\cN_\l,\;\;\;\l\in\bC_-\label{4.15.9}
\end{gather}
Then in view of \eqref{3.31} the equalities \eqref{4.15.8} and
\eqref{4.15.9} define the operator $v_0(\l)\in
[H_0,\cN_\l],\;\l\in\CR$. Moreover, \eqref{4.15.6} induces the
operator $u_+(\l)\in [\cH_b,\cN_\l]$. Introduce also the operator
\begin {equation}\label{4.15.10}
 u_-(\l)=(\omega_-(\l),\, \wt u_-(\l)): \wh H_2\oplus \cH_b\to \cN_\l,
\;\;\;\l\in\bC_-.
\end{equation}
Now assume that $v_0(\cd,\l)\in \lo{H_0}, \; u_+(\cd,\l)\in
\lo{\cH_b}$ and $u_-(\cd,\l)\in \lo{\wh H_2\oplus\cH_b} $ are the
operator solutions of Eq. \eqref{3.2} corresponding to $v_0(\l),
\; u_+(\l)$ and $u_-(\l)$ in accordance with Lemma \ref{lem3.2}.
Then  combining of \eqref{4.15.6}--\eqref{4.15.10} with
\eqref{4.15.3}--\eqref{4.15.5} yields the relations
\eqref{4.8}--\eqref{4.11a} for $v_0(\cd,\l)$ and $u_\pm(\cd,\l)$.
Finally, by using uniqueness of the solution of the boundary value
problem \eqref{4.2}--\eqref{4.5} (with $f=0$) one proves
uniqueness of $v_0(\cd,\l)$ and $u_\pm(\cd,\l)$ in the same way as
in \cite[Proposition 4.3]{Mog13.1}.
\end{proof}
Let $\wt U$ be a $J$-unitary extension \eqref{3.17.5} of $U$ and
let $\G_{0a}$ be the mapping \eqref{3.23}. By using the solutions
$v_0(\cd,\l)$ and $u_\pm(\cd,\l)$ we define the operator functions
\begin {gather}
X_+(\l)=\begin{pmatrix}m_0(\l) & \Phi_+(\l) \cr \Psi_+(\l) & \dot
M_+(\l)
\end{pmatrix}:H_0\oplus\cH_b\to H_0\oplus (\wh H_2\oplus\cH_b), \quad
\l\in\bC_+\label{4.16}\\
X_-(\l)=\begin{pmatrix}m_0(\l) & \Phi_-(\l) \cr \Psi_-(\l) & \dot
M_-(\l)
\end{pmatrix}:H_0\oplus (\wh H_2\oplus\cH_b)\to H_0\oplus\cH_b , \quad
\l\in\bC_-,\label{4.17}
\end{gather}
where
\begin {gather}
m_0(\l)=(\G_{0a}+\wh\G_a)v_0(\l)+\tfrac i 2 P_{\wh H},\;\;\;\l\in\CR\label{4.18}\\
\Phi_\pm (\l)=(\G_{0a}+\wh\G_a)u_\pm(\l), \;\;\;\l\in\bC_\pm\label{4.19}\\
\Psi_+(\l)=(\wh\G_{a2}-\G_{1b})v_0(\l)+i P_{\wh H_2}, \quad\dot
M_+(\l)=(\wh\G_{a2}-\G_{1b})u_+(\l), \;\;\; \l\in\bC_+\label{4.20}\\
\Psi_-(\l)=-\G_{1b}v_0(\l), \quad \dot M_-(\l)=-\G_{1b}u_-(\l),
\;\;\; \l\in\bC_-\label{4.21}
\end{gather}
In the following proposition we specify a connection between the
operator functions $X_\pm(\cd)$ and the Weyl functions
$M_\pm(\cd)$ of the decomposing boundary triplet $\Pi_-$ for
$\Tma$.
\begin{proposition}\label{pr4.5}
Let $\Pi_-=\bta$ be the decomposing boundary triplet
\eqref{3.32}--\eqref{3.34} for $\Tma$ and let
\begin{gather}
M_+(\l)=\begin{pmatrix} M_1(\l )& M_{2+}(\l) \cr N_{1+}(\l) &
N_{2+}(\l) \cr M_{3+}(\l) & M_{4+}(\l)\end{pmatrix}:
H_0'\oplus\cH_b\to H_0'\oplus\wh H_2\oplus\cH_b,
\;\;\;\l\in\bC_+\label{4.22}\\
M_-(\l)=\begin{pmatrix} M_1(\l )& N_{1-}(\l) & M_{2-}(\l) \cr
M_{3-}(\l) & N_{2-}(\l) & M_{4-}(\l)
\end{pmatrix}: H_0'\oplus \wh H_2\oplus\cH_b\to H_0'\oplus \cH_b,
\;\;\;\l\in\bC_-\label{4.23}
\end{gather}
be the block  representations of the corresponding Weyl functions
$M_\pm(\cd)$. Then the entries of the operator matrices
\eqref{4.16} and \eqref{4.17} are holomorphic on their domains and
have the following block matrix representations:
\begin{gather}
m_0(\l)=\begin{pmatrix}M_1(\l) & 0 \cr N_{1+}(\l) & \tfrac i 2
I_{\wh H_2}\end{pmatrix}:\underbrace{H_0'\oplus\wh H_2}_{H_0}\to
\underbrace{H_0'\oplus\wh H_2}_{H_0}, \quad \l\in\bC_+\label{4.24}\\
m_0(\l)=\begin{pmatrix}M_1(\l) &  N_{1-}(\l) \cr  0   & -\tfrac i
2 I_{\wh H_2}\end{pmatrix}:\underbrace{H_0'\oplus\wh H_2}_{H_0}\to
\underbrace{H_0'\oplus\wh H_2}_{H_0}, \quad \l\in\bC_-\label{4.25}\\
\F_+(\l)=\begin{pmatrix}M_{2+}(\l)\cr
N_{2+}(\l)\end{pmatrix},\;\; \l\in\bC_+; \quad
\F_-(\l)=\begin{pmatrix}N_{1-}(\l) & M_{2-}(\l) \cr -i I &
0\end{pmatrix}, \quad \l\in\bC_-\label{4.26}\\
\Psi_+(\l)=\begin{pmatrix}N_{1+}(\l) & i I \cr   M_{3+}(\l) &
0\end{pmatrix}, \quad\l\in\bC_+; \quad  \Psi_-(\l)=(M_{3-}(\l),\,
N_{2-}(\l)),
\quad\l\in\bC_-\label{4.27}\\
\dot M_+(\l)=(N_{2+}(\l), \, M_{4+}(\l) )^\top, \;\;
\l\in\bC_+;\quad  \dot M_-(\l)=(N_{2-}(\l), \,M_{4-}(\l) ),\;\;
\l\in\bC_-.\label{4.28}
\end{gather}
Moreover, the following equalities hold
\begin {equation}\label{4.29}
m_0^*(\ov\l)=m_0(\l), \;\;  \F_+^*(\ov\l)=\Psi_-(\l), \;\;
\Psi_+^*(\ov\l)=\F_-(\l), \;\; \dot M_+^*(\ov\l)=\dot M_-(\l),
\;\; \l\in\bC_-.
\end{equation}
\end{proposition}
\begin{proof}
 Let $\G_0'$ and $\G_1'$ be given by \eqref{4.13} and let $Z_{\pm}(\cd)$ be
the same as in Proposition \ref{pr4.4}. Then by \eqref{4.12} and
\eqref{2.13}, \eqref{2.14} one has
\begin {equation*}
(\G_1'-i P_{\wh H_2}\G_0')Z_+(\l)= M_+(\l), \quad \l\in\bC_+;
\qquad \G_1' Z_-(\l)=M_-(\l), \quad \l\in\bC_-,
\end{equation*}
which in view of \eqref{4.22} and \eqref{4.23} can be represented
as
\begin{gather*}
\begin{pmatrix}\G_{0a}+\tfrac 1 2  (\wh\G_{a1}+ \wh\G_b)\cr
\wh\G_{a2} \cr -\G_{1b}\end{pmatrix} Z_+(\l)=\begin{pmatrix}
M_1(\l )& M_{2+}(\l)\cr N_{1+}(\l) & N_{2+}(\l) \cr M_{3+}(\l) &
M_{4+}(\l)
\end{pmatrix}, \;\; \l\in\bC_+ \\
\begin{pmatrix}\G_{0a}+\tfrac 1 2  (\wh\G_{a1}+ \wh\G_b) \cr
-\G_{1b}\end{pmatrix} Z_-(\l)=\begin{pmatrix} M_1(\l )& N_{1-}(\l)
& M_{2-}(\l) \cr M_{3-}(\l) & N_{2-}(\l) &
M_{4-}(\l)\end{pmatrix}, \;\; \l\in\bC_-.
\end{gather*}
This implies that
\begin{gather}
\G_{0a}Z_+(\l)= P_H( M_1(\l),\,  M_{2+}(\l)),\;\;\;\tfrac 1 2
(\wh\G_{a1}
+\wh\G_b) Z_+(\l)=P_{\wh H_1}(M_1(\l),\,M_{2+}(\l))\label{4.30}\\
\wh\G_{a2} Z_+(\l)= (N_{1+}(\l),\, N_{2+}(\l)),\;\;\;
\G_{1b}Z_+(\l)=(-M_{3+}(\l), \,  -M_{4+}(\l)) \label{4.31}\\
\G_{0a}Z_-(\l)= P_H( M_1(\l),\, N_{1-}(\l),\,  M_{2-}(\l))\label{4.32}\\
\tfrac 1 2 (\wh\G_{a1} +\wh\G_b) Z_-(\l)=P_{\wh H_1}(M_1(\l),\,
N_{1-}(\l),\,
M_{2-}(\l))\label{4.32a}\\
\G_{1b}Z_-(\l)=(-M_{3-}(\l), \, - N_{2-}(\l),\,  -M_{4-}(\l)).
\label{4.33}
\end{gather}
Summing up the second equality in \eqref{4.15.3} with the
equalities \eqref{4.30} and the first equality in \eqref{4.31}
one obtains
\begin {equation}\label{4.34}
(\G_{0a}+\wh\G_a)Z_+(\l)=(M_1(\l)+N_{1+}(\l) - \tfrac i 2 P_{\wh
H_1},\,M_{2+}(\l)+N_{2+}(\l) ):H_0'\oplus\cH_b\to H_0.
\end{equation}
Similarly, summing up  the second equality in \eqref{4.15.4}, the
first equality in \eqref{4.15.5} and the equalities \eqref{4.32}
and \eqref{4.32a} one gets
\begin {equation}\label{4.35}
(\G_{0a}+\wh\G_a)Z_-(\l)= (M_1(\l) - \tfrac i 2 P_{\wh H_1},\,
N_{1-}(\l) - i I_{\wh H_2},\, M_{2-}(\l) ).
\end{equation}
Combining \eqref{4.34} and \eqref{4.31} with the block
representation \eqref{4.15.6} of $Z_+(\l)$ and taking definition
\eqref{4.15.8}  of $v_0(\l)$ into account we obtain
\begin{gather}
(\G_{0a}+\wh\G_a)v_0(\l)=(M_1(\l)+N_{1+}(\l) - \tfrac i 2 P_{\wh
H_1})P_{H_0'},
\quad \l\in\bC_+\label{4.36}\\
(\G_{0a}+\wh\G_a)u_+(\l)=M_{2+}(\l)+N_{2+}(\l), \quad\l\in\bC_+\label{4.36.0}\\
\wh\G_{a2} v_0(\l)= N_{1+}(\l)P_{H_0'}, \quad \wh\G_{a2} u_+(\l)=
N_{2+}(\l),\quad \l\in\bC_+\label{4.36.1}\\
\G_{1b} v_0(\l)= -M_{3+}(\l)P_{H_0'}, \quad  \G_{1b} u_+(\l)=
-M_{4+}(\l),\quad \l\in\bC_+.\label{4.36.2}
\end{gather}
Moreover, \eqref{4.35} and \eqref{4.33} together with
\eqref{4.15.7}, \eqref{4.15.9} and \eqref{4.15.10} give
\begin{gather}
(\G_{0a}+\wh\G_a)v_0(\l)=(M_1(\l)- \tfrac i 2 P_{\wh
H_1})P_{H_0'}+(N_{1-}(\l)
-i I_{\wh H_2})P_{\wh H_2}, \quad \l\in\bC_-\label{4.37}\\
(\G_{0a}+\wh\G_a)u_-(\l)=(N_{1-}(\l)-i I_{\wh H_2},\, M_{2-}(\l)):
\wh
H_2\oplus\cH_b\to H_0, \quad \l\in\bC_-\label{4.38}\\
\G_{1b} v_0(\l)= - (M_{3-}(\l), \, N_{2-}(\l)), \quad  \G_{1b}
u_-(\l)= -(N_{2-}(\l), \, M_{4-}(\l)),\quad
\l\in\bC_-.\label{4.39}
\end{gather}
Combining of \eqref{4.18} with  \eqref{4.36} and \eqref{4.37}
yields
\begin{gather*}
m_0(\l)=(M_1(\l)+N_{1+}(\l) - \tfrac i 2 P_{\wh
H_1})P_{H_0'}+\tfrac i 2 P_{\wh
H}=(M_1(\l)+N_{1+}(\l))P_{H_0'}-\tfrac i 2 P_{\wh H_1}+\\
+\tfrac i 2 (P_{\wh H_1}+P_{\wh
H_2})=(M_1(\l)+N_{1+}(\l))P_{H_0'}+\tfrac i 2
P_{\wh H_2}, \;\;\;\l\in\bC_+\\
m_0(\l)=(M_1(\l)- \tfrac i 2 P_{\wh H_1})P_{H_0'}+(N_{1-}(\l)-i
I_{\wh
H_2})P_{\wh H_2}+\tfrac i 2P_{\wh H}=M_1(\l)P_{H_0'}+N_{1-}(\l)P_{\wh H_2}-\\
-\tfrac i 2 P_{\wh H_1}-i P_{\wh H_2}+\tfrac i 2 (P_{\wh
H_1}+P_{\wh H_2})= M_1(\l)P_{H_0'}+N_{1-}(\l)P_{\wh H_2}-\tfrac i
2 P_{\wh H_2},\;\;\;\l\in\bC_-,
\end{gather*}
which is equivalent to the block representations \eqref{4.24} and
\eqref{4.25} of $m_0(\l)$. Moreover, \eqref{4.19} together with
\eqref{4.36.0} and \eqref{4.38} gives the block representations
\eqref{4.26} of $\Phi_+(\l)$ and $\Phi_-(\l)$. Combining
\eqref{4.20} and \eqref{4.21} with \eqref{4.36.1}, \eqref{4.36.2}
and \eqref{4.39} we arrive at the block representations
\eqref{4.27} of $\Psi_\pm(\l)$ and \eqref{4.28} of $\dot M_\pm
(\l)$. Finally, holomorphy of the operator functions $m_0(\cd), \;
\Phi_\pm(\cd),\; \Psi_\pm(\cd)$ and $\dot M_\pm(\cd)$ and the
equalities \eqref{4.29} are implied by \eqref{4.22}--\eqref{4.28}
and the equality $M_+^*(\ov\l)=M_-(\l),\; \l\in\bC_- $.
\end{proof}
\begin{theorem}\label{th4.6}
Let in Case 1 $\pair$ be a boundary parameter \eqref{4.1},
\eqref{4.1a}. Then:

{\rm (1)} For each $\l\in\CR$ there exists a unique operator
solution $v_\tau(\cd,\l)\in\lo{H_0}$ of Eq. \eqref{3.2} satisfying
the boundary conditions
\begin {gather}
\G_{1a}v_\tau(\l)=-P_H, \quad \l\in\CR\label {4.41}\\
i(\wh\G_{a1}-\wh\G_b)v_\tau(\l)=P_{\wh H_1},\quad \l\in\CR\label {4.42}\\
\wh C_{02}(\l)(i\wh\G_{a2} v_\tau(\l)-P_{\wh H_2})+
C_{0b}(\l)\G_{0b}v_\tau(\l)+C_1(\l)\G_{1b}v_\tau(\l)=0,
\;\;\;\;\l\in\bC_+
\label {4.43}\\
\wh D_{02}(\l)(i\wh\G_{a2} v_\tau(\l)-P_{\wh H_2})+
D_{0b}(\l)\G_{0b}v_\tau(\l)+D_1(\l)\G_{1b}v_\tau(\l)=0,
\;\;\;\;\l\in\bC_- \label {4.44}
\end{gather}
(here  $v_\tau(\l)$ is the linear map from Lemma \ref{lem3.2}
corresponding to the solution $v_\tau(\cd,\l)$).

{\rm (2)}  $v_\tau(\cd,\l)$ is connected with the solutions
$v_0(\cd,\l)$ and
 $u_\pm(\cd,\l)$ from Proposition \ref{pr4.4} by
\begin {gather}
v_\tau(t,\l)=v_0(t,\l)-u_+(t,\l)(\tau_-^*(\ov\l)+\dot
M_{+}(\l))^{-1} \Psi_+(\l),
\quad \l\in\bC_+,\label{4.45}\\
v_\tau(t,\l)=v_0(t,\l)-u_-(t,\l)(\tau_-(\l)+ \dot
M_{-}(\l))^{-1}\Psi_-(\l), \quad \l\in\bC_-,\label{4.46}
\end{gather}
where $\Psi_\pm(\l)$ and $\dot M_\pm (\l)$ are the operator
functions defined in \eqref{4.20} and \eqref{4.21}.
\end{theorem}
\begin{proof}
It follows from \eqref{4.28} and \eqref{4.23} that $\dot
M_-(\l)=P_{\cH_b}M_-(\l)\up \wh H_2\oplus\cH_b, \;\l\in\bC_-$.
Therefore by Proposition \ref{pr2.10a}, (3) $\dot M_-(\l)$ is the
Weyl function of the boundary triplet $\dot\Pi_-$ for $T^*$  and
in view of \cite[Theorem 3.11]{Mog13.2} one has $0\in\rho
(\tau_-(\l)+ \dot M_{-}(\l)), \;\l\in\bC_-$. Hence for each
$\l\in\CR$ the equalities \eqref{4.45} and \eqref{4.46} correctly
define the solution $v_\tau(\cd,\l)\in\lo{H_0}$ of Eq. \eqref{3.2}
and to prove the theorem it is sufficient to show that such
$v_\tau(\cd,\l)$ is a unique solution of \eqref{3.2} belonging to
$\lo{H_0}$ and  satisfying \eqref{4.41}--\eqref{4.44}.

Combining \eqref{4.45} and \eqref{4.46} with \eqref{4.8} and
\eqref{4.10} one gets the equalities \eqref{4.41} and
\eqref{4.42}.  To prove \eqref{4.43} and \eqref{4.44} we let
\begin {equation}\label{4.46a}
T_+(\l)=(\tau_-^*(\ov\l)+\dot M_{+}(\l))^{-1},\;\l\in\bC_+,\;
\quad T_-(\l)=(\tau_-(\l) + \dot M_{-}(\l))^{-1},\;\;\l\in\bC_-,
\end{equation}
so that $\tau_-^*(\ov\l)$ and $\tau_-(\l)$ can be written as
\begin {gather}
\tau_-^*(\ov\l)=\{\{T_+(\l)h, (I-\dot M_{+}(\l)T_+(\l))h\}
:h\in\wh
H_2\oplus\cH_b\}, \label{4.47}\\
 \tau_-(\l)=\{\{T_-(\l)h, (I-\dot M_{-}(\l)T_-(\l))h\}:h\in\cH_b\}.\label{4.48}
\end{gather}
By using definition of the  class $\wt R_-(\cH_0,\cH_1)$ in
\cite{Mog13.2} one can easily show that
\begin {equation*}
\tau_+(\l)=\{\{-h_1+i P_2h_0, -P_1h_0\}:\{h_1,h_0\}\in
\tau_-^*(\ov\l)\}.
\end{equation*}
This and \eqref{4.47} yield
\begin {equation}\label{4.49}
\tau_+(\l)=\{\{(-T_+(\l)+i\wt P_{\wh H_2}-i\wt P_{\wh H_2} \dot
M_{+}(\l)T_+(\l))h, (-P_{\cH_b} + P_{\cH_b} \dot
M_{+}(\l)T_+(\l))h\}\},
\end{equation}
where $\wt P_{\wh H_2}$ and $P_{\cH_b}$ are the same  as in
Proposition \ref{pr4.4} and $h$ runs over $\wh H_2\oplus\cH_b.$

It follows from \eqref{4.20} that
\begin {gather}
\wh\G_{a2}v_0(\l)=\wt P_{\wh H_2}\Psi_+(\l)-i P_{\wh H_2}, \quad
\G_{1b}v_0(\l)=-P_{\cH_b}\Psi_+(\l), \quad \l\in\bC_+\label{4.50}\\
\wh\G_{a2}u_+(\l)=\wt P_{\wh H_2}\dot M_+(\l),
\quad\G_{1b}u_+(\l)=-P_{\cH_b}\dot M_+(\l), \quad
\l\in\bC_+.\label{4.51}
\end{gather}
Combining \eqref{4.45} with \eqref{4.50}, \eqref{4.51},
\eqref{4.11} and the last equality in \eqref{4.8}  one gets
\begin {gather}
(i\wh\G_{a2} v_\tau(\l)- P_{\wh
H_2})+\G_{0b}v_\tau(\l)=(-T_+(\l)+i\wt P_{\wh
H_2}-i\wt P_{\wh H_2}\dot M_+(\l)T_+(\l) )\Psi_+(\l),\nonumber\\
\G_{1b}v_\tau(\l)=(-P_{\cH_b}+P_{\cH_b}\dot
M_+(\l)T_+(\l))\Psi_+(\l), \quad\l\in\bC_+. \label{4.52}
\end{gather}
Moreover,  \eqref{4.46} together with  \eqref{4.9}, \eqref{4.11a},
\eqref{4.21} and the last equality in  \eqref{4.8}  yields
\begin {gather}
(i\wh\G_{a2} v_\tau(\l)- P_{\wh H_2})+\G_{0b}v_\tau(\l)=-\wt
P_{\wh H_2}T_-(\l) \Psi_-(\l)-P_{\cH_b} T_-(\l)
\Psi_-(\l)=-T_-(\l)\Psi_-(\l),
\nonumber\\
\G_{1b}v_\tau(\l)=-(I-\dot
M_-(\l)T_-(\l))\Psi_-(\l),\quad\l\in\bC_-.\label{4.53}
\end{gather}
Hence by \eqref{4.48} and  \eqref{4.49} one has
\begin {equation*}
\{(i\wh\G_{a2} v_\tau(\l)- P_{\wh H_2})h_0+\G_{0b}v_\tau(\l)h_0,
\G_{1b}v_\tau(\l)h_0 \}\in \tau_\pm(\l), \;\;\;h_0\in H_0, \;\;\;
\l\in\bC_\pm,
\end{equation*}
which in view of \eqref{2.2}  and  the block representations
\eqref{4.1a} yields \eqref{4.43} and \eqref{4.44}. Finally,
uniqueness of $v_\tau (\cd,\l)$ is implied by uniqueness of the
solution of the boundary value problem \eqref{4.2}--\eqref{4.5}
(see Theorem \ref{th4.2}).
\end{proof}
\subsection{Case 2}
\begin{definition}\label{def4.7}
A boundary parameter in \emph{Case 2} is a collection $\pair\in\wt
R_-(\wh H\oplus\wt\cH_b,\cH_b)$. A truncated boundary parameter is
a boundary parameter $\pair$ satisfying $\tau_-(\l)\subset
(\{0\}\oplus\wt\cH_b)\oplus \cH_b,\; \l\in\bC_-$.
\end{definition}
\noindent According to Subsection \ref{sub2.2} a boundary
parameter $\tau= \{\tau_+, \tau_-\}$ admits the representation
\begin {equation}\label{4.54}
\tau_+(\l)=\{(C_0(\l),C_1(\l));\cH_b\}, \;\;\l\in\bC_+; \quad
\tau_-(\l)=\{(D_0(\l),D_1(\l));\wh H\oplus\wt\cH_b\},
\;\;\l\in\bC_-
\end{equation}
with holomorphic operator functions $C_0(\l)(\in [\wh
H\oplus\wt\cH_b,\cH_b]), \; C_1(\l)(\in [\cH_b])$ and $D_0(\l)(\in
[\wh H\oplus\wt\cH_b]), \; D_1(\l)(\in [\cH_b,\wh H\oplus
\wt\cH_b])$. Moreover, a truncated boundary parameter $\pair$
admits the representation \eqref{4.54} with
\begin {equation}\label{4.54.1}
D_0(\l)=\begin{pmatrix} I_{\wh H}& 0 \cr  0 & \ov
D_0(\l)\end{pmatrix}:\wh H\oplus\wt\cH_b \to \wh
H\oplus\wt\cH_b,\; \;\;D_1(\l)=\begin{pmatrix}0 \cr \ov D_1(\l)
\end{pmatrix}:\cH_b\to \wh H\oplus\wt\cH_b.
\end{equation}

Let $\pair$ be a boundary parameter \eqref{4.54} and let
\begin {equation}\label{4.55}
\begin{array}{c}
C_0(\l)=(\wh C_0(\l),\wt C_{0b}(\l)):\wh H\oplus\wt\cH_b\to \cH_b, \\
D_0(\l)=(\wh D_{0}(\l),\wt D_{0b}(\l)):\wh H\oplus\wt\cH_b\to(\wh
H\oplus \wt\cH_b)
\end{array}
\end{equation}
be the block representations of $C_0(\l)$ and $D_0(\l)$. For a
given function $f\in\lI$ consider the following boundary value
problem:
\begin{gather}
J y'-B(t)y=\l \D(t)y+\D(t)f(t), \quad t\in\cI,\label{4.59}\\
\G_{1a}y=0, \quad \l\in\CR,
\label{4.60}\\
i \wh C_{0}(\l)\wh \G_a y+  \wt C_{0b}(\l)\wt\G_{0b}y
+C_1(\l)\G_{1b}y=0, \quad
\l\in\bC_+,\label{4.61.1}\\
i \wh D_{0}(\l)\wh \G_a y+  \wt D_{0b}(\l)\wt\G_{0b}y
+D_1(\l)\G_{1b}y=0, \quad \l\in\bC_-.\label{4.61.2}
\end{gather}
One proves the following theorem in the same way as Theorem
\ref{th4.2}.
\begin{theorem}\label{th4.8}
Let in Case 2 $T$ be a symmetric relation in $\LI$ defined by
\eqref{3.43}. Then statements of Theorem \ref{th4.2} hold with a
boundary parameter $\pair$ of the form \eqref{4.54} and the
boundary value problem \eqref{4.59} - \eqref{4.61.2} in place of
\eqref{4.2} - \eqref{4.5}.
\end{theorem}
\begin{remark}\label{rem4.9}
Let $\tau_0=\{\tau_{+},\tau_{-}\}$ be a boundary parameter
\eqref{4.54}  with
\begin {equation}\label{4.61a}
C_0(\l) \equiv P_{\cH_b}(\in [\wh H\oplus\wt\cH_b,\cH_b]), \quad
C_1(\l)\equiv 0, \;\;\;
  D_0(\l)\equiv I_{\wh H\oplus\wt\cH_b},\quad D_1(\l)\equiv 0
\end{equation}
and let $A_0$ be the symmetric relation \eqref{3.46}. Then
$R_{\tau_0}(\l)$ is of the form \eqref{4.7}.
\end{remark}
\begin{proposition}\label{pr4.10}
Let in Case 2 $P_H$ (resp. $P_{\wh H}$)  be the orthoprojector in
$H_0$ onto $H$ (resp. $\wh H$)  and let $\wt P_{\wh H}$ (resp.
$P_{\wt\cH_b}$) be the orthoprojector in $\wh H\oplus\wt\cH_b$
onto $\wh H$ (resp. $\wt\cH_b$). Then:

{\rm (1)} For any $\l\in\CR$ there exists a unique operator
solution $v_0(\cd,\l)\in\lo{H_0}$ of Eq. \eqref{3.2} satisfying
\begin{gather}
\G_{1a} v_0(\l)=-P_H, \;\;\;  \l\in\CR \label{4.62}\\
\G_{0b}v_0(\l)=0, \;\;\l\in\bC_+;\qquad   i\wh\G_{a}v_0(\l)=P_{\wh
H}, \quad \wt\G_{0b}v_0(\l)=0,\;\; \l\in\bC_-\label{4.63}
\end{gather}

{\rm(2)} For any $\l\in\bC_+ \; (\l\in\bC_-)$ there exists a
unique operator solution $u_+(\cd,\l)\in\lo{\cH_b}$ (resp.
$u_-(\cd,\l)\in\lo{\wh H\oplus\wt\cH_b}$) of Eq. \eqref{3.2}
satisfying
\begin{gather}
\G_{1a}u_\pm(\l)=0, \quad \l\in\bC_\pm \label{4.64}\\
\G_{0b}u_+(\l)=I_{\cH_b},\;\; \l\in\bC_+;\qquad i
\wh\G_{a}u_-(\l)=\wt P_{\wh H},\quad
\wt\G_{0b}u_-(\l)=P_{\wt\cH_b},\;\;\; \l\in\bC_-. \label{4.65}
\end{gather}
\end{proposition}
\begin{proof}
let $\Pi_-$ be the decomposing boundary triplet for  $\Tma$
constructed with the aid of some $J$-unitary extension $\wt U$ of
$U$ (see Proposition \ref{pr3.3}). By using $\g$-fields of this
triplet one proves in the same way as in Proposition \ref{pr4.4}
the existence of isomorphisms $Z_+(\l):\cH_1\to \cN_\l,
\;\l\in\bC_+,$ and $Z_-(\l):\cH_0\to \cN_\l, \; \l\in\bC_-,$ such
that
\begin{gather}
\begin{pmatrix}-\G_{1a} \cr \G_{0b}\end{pmatrix} Z_+(\l)=\begin{pmatrix}
I_{H} & 0 \cr 0 & I_{\cH_b}\end{pmatrix},\quad \l\in\bC_+\label{4.67}\\
\begin{pmatrix}-\G_{1a}\cr i \wh\G_a \cr \wt\G_{0b}\end{pmatrix} Z_-(\l)=
\begin{pmatrix} I_{H} & 0 & 0 \cr
0 & I_{\wh H} & 0 \cr 0 & 0 & I_{\wt\cH_b}\end{pmatrix}, \;\;\;\;
\l\in\bC_- \label{4.68}
\end{gather}
It follows from \eqref{4.67} and \eqref{4.68} that
\begin{gather}
\G_{1a}Z_+(\l)=(-I_H,\; 0),\;\;\;\;\;\G_{0b} Z_+(\l)= (0,\,
I_{\cH_b})
\label{4.69}\\
\G_{1a}Z_-(\l)=(-I_H,\; 0,\;0),\;\;\;\; i\wh\G_{a}Z_-(\l)=(0,\,
I_{\wh H}, \, 0),\;\;\;\; \wt\G_{0b} Z_-(\l)= (0,\, 0,\, I_{
\wt\cH_b}).\label{4.70}
\end{gather}
Assume that
\begin {gather}
Z_+(\l)=(r(\l), \,u_+(\l)):H\oplus\cH_b\to\cN_\l,\;\;\;\l\in\bC_+
\label{4.76}\\
Z_-(\l)=(r(\l),\,\omega_-(\l),\,  \wt u_-(\l)):H\oplus\wh H\oplus
\wt\cH_b\to \cN_\l,\;\;\;\l\in\bC_-\label{4.77}
\end{gather}
are the block representations of $Z_\pm (\l)$ (see  \eqref{3.36} )
and let $v_0(\l):H_0\to\cN_\l, \; \l\in\CR,$ be the operator given
by
\begin {gather}
v_0(\l):= (r(\l),\; 0):H\oplus\wh H\to \cN_\l,\;\;\;\l\in\bC_+\label{4.78}\\
v_0(\l):= (r(\l),\;\omega_-(\l) ):H\oplus\wh H\to
\cN_\l,\;\;\;\l\in\bC_-. \label{4.78.1}
\end{gather}
It is clear that \eqref{4.76} induces the operator $u_+(\l)\in
[\cH_b,\cN_\l]$. Introduce also the operator
\begin {equation}\label{4.79}
 u_-(\l)=(\omega_-(\l),\, \wt u_-(\l)): \wh H\oplus \wt\cH_b\to \cN_\l,
\;\;\;\l\in\bC_-.
\end{equation}
Let $v_0(\cd,\l)\in \lo{H_0}, \; u_+(\cd,\l)\in \lo{\cH_b}$ and
$u_-(\cd,\l)\in \lo{\wh H\oplus\wt\cH_b} $ be the operator
solutions of Eq. \eqref{3.2} corresponding to $v_0(\l), \;
u_+(\l)$ and $u_-(\l)$ respectively (see Lemma \ref{lem3.2}). Then
combining of \eqref{4.76}--\eqref{4.79} with \eqref{4.69} and
\eqref{4.70} gives the relations \eqref{4.62}--\eqref{4.65}.
Finally, uniqueness of $v_0(\cd,\l)$ and $u_\pm(\cd,\l)$ follows
from uniqueness of the solution of the boundary value problem
\eqref{4.59}--\eqref{4.61.2}.
\end{proof}
Let $\wt U$ be a $J$-unitary extension \eqref{3.17.5} of $U$ and
let $\G_{0a}$ be the mapping \eqref{3.23}.  Introduce the operator
functions
\begin {gather}
X_+(\l)=\begin{pmatrix}m_0(\l) & \Phi_+(\l) \cr \Psi_+(\l) & \dot
M_+(\l)
\end{pmatrix}:H_0\oplus\cH_b\to H_0\oplus (\wh H\oplus\wt\cH_b), \quad
\l\in\bC_+\label{4.80}\\
X_-(\l)=\begin{pmatrix}m_0(\l) & \Phi_-(\l) \cr \Psi_-(\l) & \dot
M_-(\l)
\end{pmatrix}:H_0\oplus (\wh H\oplus\wt\cH_b)\to H_0\oplus\cH_b , \quad
\l\in\bC_-,\label{4.81}
\end{gather}
with the entries defined in terms of the solutions $v_0(\cd,\l)$
and $u_\pm(\cd,\l)$ as follows: $m_0(\l),\; \Phi_\pm(\l)$,  $
\Psi_-(\l)$ and $ \dot M_-(\l)$ are given by \eqref{4.18},
\eqref{4.19} and  \eqref{4.21}, while
\begin {equation}\label{4.81a}
\Psi_+(\l)=(\wh\G_{a}-\G_{1b}-i \wh\G_b)v_0(\l)+i P_{\wh H},
\quad\dot M_+(\l)=(\wh\G_{a}-\G_{1b}-i \wh\G_b)u_+(\l), \;\;\;
\l\in\bC_+.
\end{equation}
\begin{proposition}\label{pr4.11}
Let $\Pi_-=\bta$ be the decomposing boundary triplet
\eqref{3.36}--\eqref{3.38} for $\Tma$ and let
\begin{gather}
M_+(\l)=\begin{pmatrix} M_1(\l )& M_{2+}(\l) \cr N_{1+}(\l) &
N_{2+}(\l) \cr M_{3+}(\l) & M_{4+}(\l)\end{pmatrix}:
H\oplus\cH_b\to H\oplus\wh
H\oplus\wt\cH_b,\;\;\;\l\in\bC_+\label{4.82}\\
M_-(\l)=\begin{pmatrix} M_1(\l )& N_{1-}(\l) & M_{2-}(\l) \cr
M_{3-}(\l) & N_{2-}(\l) & M_{4-}(\l)
\end{pmatrix}: H\oplus \wh H\oplus\wt\cH_b\to H\oplus \cH_b,
\;\;\;\l\in\bC_-\label{4.83}
\end{gather}
be the block matrix representations of the corresponding Weyl
functions.  Then the entries of the operator matrices \eqref{4.80}
and \eqref{4.81} are holomorphic on their  domains and  satisfy
\begin{gather}
m_0(\l)=\begin{pmatrix}M_1(\l) & 0 \cr N_{1+}(\l) & \tfrac i 2
I_{\wh H}\end{pmatrix}:\underbrace{H\oplus\wh H}_{H_0}\to
\underbrace{H\oplus\wh H}_{H_0}, \quad \l\in\bC_+\label{4.84}\\
m_0(\l)=\begin{pmatrix}M_1(\l) &  N_{1-}(\l) \cr  0   & -\tfrac i
2 I_{\wh H}\end{pmatrix}:\underbrace{H\oplus\wh H}_{H_0}\to
\underbrace{H\oplus\wh H}_{H_0}, \quad \l\in\bC_-\label{4.85}
\end{gather}
and \eqref{4.26}--\eqref{4.28}. Moreover, the equalities
\eqref{4.29} are valid.
\end{proposition}
\begin{proof}
 Let $Z_\pm(\cd)$ be the same as in Proposition \ref{pr4.10}. By using the
reasonings similar to those in the proof of Proposition
\ref{pr4.5} one proves the equalities
\begin{gather}
\begin{pmatrix}\G_{0a} \cr \wh\G_a \cr  -\G_{1b}-i\wh \G_b\end{pmatrix}
 Z_+(\l)=\begin{pmatrix} M_1(\l )&
M_{2+}(\l)\cr N_{1+}(\l) & N_{2+}(\l) \cr M_{3+}(\l) & M_{4+}(\l)
\end{pmatrix}, \;\;\; \l\in\bC_+,\label{4.86}\\
\begin{pmatrix}\G_{0a}\cr -\G_{1b}\end{pmatrix} Z_-(\l)=\begin{pmatrix}
M_1(\l )& N_{1-}(\l) & M_{2-}(\l) \cr M_{3-}(\l) & N_{2-}(\l) &
M_{4-}(\l)\end{pmatrix}, \;\;\;\; \l\in\bC_-. \label{4.87}
\end{gather}
It follows from \eqref{4.86} that
\begin{gather*}
(\G_{0a}+\wh\G_a)Z_+(\l)=(M_1(\l)+N_{1+}(\l),\,M_{2+}(\l)+N_{2+}(\l) ),\\
(\wh\G_a-\G_{1b}-i\wh\G_b)Z_+(\l)=(N_{1+}(\l)+M_{3+}(\l),\,N_{2+}(\l)+
M_{4+}(\l) ),\quad \l\in\bC_+.
\end{gather*}
Moreover, \eqref{4.87} and the second equality in \eqref{4.70}
yield
\begin{gather*}
(\G_{0a}+\wh\G_a)Z_-(\l)= (M_1(\l),\, N_{1-}(\l) - i I_{\wh H},\,
M_{2-}(\l)
),\\
-\G_{1b}Z_-(\l)=(M_{3-}(\l), \, N_{2-}(\l),\,  M_{4-}(\l)),\quad
\l\in\bC_-.
\end{gather*}
Combining these equalities  with the block representations
\eqref{4.76} and \eqref{4.77} of $Z_\pm(\l)$ and taking
\eqref{4.78}--\eqref{4.79} into account we obtain
\begin{gather*}
(\G_{0a}+\wh\G_a)v_0(\l)=(M_1(\l)+N_{1+}(\l))P_H,\;\;\;
 (\G_{0a}+\wh\G_a)u_+(\l)=M_{2+}(\l)+N_{2+}(\l),\;\; \l\in\bC_+\\
(\wh\G_a-\G_{1b}-i\wh\G_b)v_0(\l)=(N_{1+}(\l)+M_{3+}(\l))P_H,\;\;\;
\l\in\bC_+\\
(\wh\G_a-\G_{1b}-i\wh\G_b)u_+(\l)=N_{2+}(\l)+M_{4+}(\l),\;\;\;
\l\in\bC_+
\\
(\G_{0a}+\wh\G_a)v_0(\l)=M_1(\l)P_H+N_{1-}(\l)P_{\wh H}-iP_{\wh
H},\quad
\l\in\bC_-\\
(\G_{0a}+\wh\G_a)u_-(\l)=(N_{1-}(\l)-iI_{\wh H},\, M_{2-}(\l)):\wh
H\oplus\wt
\cH_b\to H_0, \quad\l\in\bC_-\\
-\G_{1b} v_0(\l)=(M_{3-}(\l), \, N_{2-}(\l)), \qquad -\G_{1b}
u_-(\l)=(N_{2-}(\l), \, M_{4-}(\l)), \quad\l\in\bC_-
\end{gather*}
This and definitions \eqref{4.18},  \eqref{4.19}, \eqref{4.21} and
\eqref{4.81a} of $m_0(\cd), \; \Phi_\pm (\cd), \; \Psi_\pm(\cd)$
and $ \dot M_\pm (\cd)$ yield the equalities \eqref{4.84},
\eqref{4.85} and \eqref{4.26}--\eqref{4.28}. Finally one proves
\eqref{4.29} in the same way as in Proposition \ref{pr4.5}.
\end{proof}
\begin{theorem}\label{th4.12}
Let in Case 2 $\pair$ be a boundary parameter defined by
\eqref{4.54} and \eqref{4.55}. Then:

{\rm (1)} Statement {\rm (1)} of Theorem \ref{th4.6} holds with
the boundary condition \eqref{4.41} and the following boundary
conditions in place of \eqref{4.42}--\eqref{4.44}:
\begin {gather}
\wh C_{0}(\l)(i\wh\G_a v_\tau(\l)-P_{\wh H})+ \wt
C_{0b}(\l)\wt\G_{0b}v_\tau(\l)+C_1(\l)\G_{1b}v_\tau(\l)=0,
\;\;\;\;\l\in\bC_+
\label {4.94}\\
\wh D_{0}(\l)(i\wh\G_a v_\tau(\l)-P_{\wh H})+ \wt
D_{0b}(\l)\wt\G_{0b}v_\tau(\l)+D_1(\l)\G_{1b}v_\tau(\l)=0,\;\;\;\;\l\in\bC_-\label
{4.95}
\end{gather}

{\rm (2)} The solution $v_\tau(\cd,\l)$ is of the form
\eqref{4.45}, \eqref{4.46}, where $v_0(\cd,\l)$ and
$u_\pm(\cd,\l)$ are defined in Proposition \ref{pr4.10} and
$\Psi_\pm(\l)$ and $\dot M_\pm(\l)$ are given by \eqref{4.21} and
\eqref{4.81a}.
\end{theorem}
\begin{proof}
As in Theorem \ref{th4.6} one proves that \eqref{4.45} and
\eqref{4.46} correctly define the solution
$v_\tau(\cd,\l)\in\lo{H_0}$ of Eq. \eqref{3.2}. Therefore it
remains to show that such $v_\tau(\cd,\l)$ is a unique solution of
\eqref{3.2} belonging to $\lo{H_0}$ and  satisfying \eqref{4.41},
\eqref{4.94} and \eqref{4.95}.

First observe that  \eqref{4.62} and \eqref{4.64} yield
\eqref{4.41}. Next assume that $T_\pm(\l)$ are defined by
\eqref{4.46a}. Then in the same way as in Theorem \ref{th4.6} one
gets   \eqref{4.48} and the equality
\begin {gather}
\tau_+(\l)=\{\{(-T_+(\l)+iP_{\wh H\oplus\wh\cH_b}- iP_{\wh
H\oplus\wh\cH_b}\dot M_{+}(\l)T_+(\l))h ,\qquad \qquad\quad
\qquad\qquad\label{4.96}\\
\qquad\qquad \qquad\qquad\qquad \qquad\qquad
\qquad\qquad(-P_{\cH_b} + P_{\cH_b}\dot
M_{+}(\l)T_+(\l))h\}:h\in\wh H\oplus\wt\cH_b\},\nonumber
\end{gather}
where $P_{\wh H\oplus\wh\cH_b}$ is the orthoprojector in $\wh
H\oplus\wt\cH_b$ onto $\wh H\oplus\wh\cH_b(=( \wh
H\oplus\wt\cH_b)\ominus \cH_b)$.

It follows from \eqref{4.81a} that
\begin {gather}
\wh\G_a v_0(\l)=\wt P_{\wh H}\Psi_+(\l)-iP_{\wh H}, \quad
\G_{1b}v_0(\l)=-P_{\cH_b}\Psi_+(\l),\;\;\;\l\in\bC_+ \label{4.97}\\
 \wh \G_b v_0(\l)=iP_{\wh\cH_b}\Psi_+(\l),
\;\;\;\l\in\bC_+\label{4.97a}\\
\wh\G_a u_+(\l)=\wt P_{\wh H}\dot M_+(\l), \quad
\G_{1b}u_+(\l)=-P_{\cH_b}\dot M _+(\l), \quad \wh \G_b
u_+(\l)=iP_{\wh\cH_b}\dot M_+(\l),\label{4.98}
\end{gather}
where $\wt P_{\wh H}, \; P_{\cH_b}$ and $P_{\wh\cH_b}$ are the
orthoprojectors in $\wh H\oplus\wt\cH_b$ onto $\wh H, \; \cH_b$
and $\wh \cH_b$ respectively. Moreover, \eqref{4.97a}, the last
equality in \eqref{4.98} and  the first equalities in \eqref{4.63}
and \eqref{4.65} give
\begin {equation}\label{4.99}
\wt\G_{0b}v_0(\l)=iP_{\wh\cH_b}\Psi_+(\l),
\;\;\;\;\wt\G_{0b}u_+(\l)=I_{\cH_b}+i P_{\wh\cH_b} \dot M_+(\l),
\;\;\; \l\in\bC_+.
\end{equation}
Now combining \eqref{4.45} with  \eqref{4.97}--\eqref{4.99} one
gets \eqref{4.52} and the equality
\begin {equation*}
(i\wh\G_a v_\tau(\l)-P_{\wh
H})+\wt\G_{0b}v_\tau(\l)=(-T_+(\l)+iP_{\wh
H\oplus\wh\cH_b}-iP_{\wh H\oplus\wh\cH_b}\dot
M_+(\l)T_+(\l))\Psi_+(\l),\;\;\;\l\in\bC_+.
\end{equation*}
Moreover, \eqref{4.46} together with   \eqref{4.63}, \eqref{4.65}
and \eqref{4.21} leads to \eqref{4.53} and the equality
\begin {equation*}
(i\wh\G_a v_\tau(\l)-P_{\wh H})+\wt\G_{0b}v_\tau(\l)=-\wt P_{\wh
H}T_-(\l)\Psi_-(\l)-P_{\wt\cH_b}T_-(\l)\Psi_-(\l)=-T_-(\l)\Psi_-(\l),
\;\;\;\l\in\bC_-.
\end{equation*}
Therefore in view of \eqref{4.48} and \eqref{4.96} one has
\begin{equation*}
\{(i\wh\G_a v_\tau(\l)-P_{\wh H})h_0+\wt\G_{0b}v_\tau(\l)h_0,
\G_{1b}v_\tau(\l)h_0\}\in\tau_\pm(\l), \quad h_0\in H_0, \quad
\l\in\bC_\pm,
\end{equation*}
which implies \eqref{4.94} and \eqref{4.95}. Finally, uniqueness
of $v_\tau(\cd,\l)$ follows from uniqueness of the solution  of
the boundary value problem \eqref{4.59}--\eqref{4.61.2}.
\end{proof}
\section{$m$-functions and eigenfunction expansions}\label{sect5}
\subsection{$m$-functions}
In this subsection we assume that $\wt U$ is a $J$-unitary
extension \eqref{3.17.5} of $U$ and $\G_{0a}$ is the mapping
\eqref{3.23}.

Let $\tau$ be a boundary parameter and let $v_\tau
(\cd,\l)\in\lo{H_0}$ be the  operator solution of Eq. \eqref{3.2}
defined in Theorems \ref{th4.6} and \ref{th4.12}. Similarly to the
case $n_-(\Tmi)\leq n_+(\Tmi)$ (see \cite{Mog13.1}) we introduce
the following definition.
\begin{definition}\label{def5.2}
The operator function $m_\tau(\cd):\CR\to [H_0]$ defined by
\begin {equation}\label{5.0}
m_\tau(\l)=(\G_{0a}+\wh \G_a)v_\tau (\l)+\tfrac i 2 P_{\wh H},
\quad\l\in\CR,
\end{equation}
is called the $m$-function corresponding to the boundary parameter
$\tau$ or, equivalently, to the boundary value problem
\eqref{4.2}--\eqref{4.5} (in \emph{Case 1}) or
\eqref{4.59}--\eqref{4.61.2} (in \emph{Case 2}).
\end{definition}
From \eqref{4.41} it follows from   that $m_\tau(\cd)$ satisfies
the equality
\begin {equation}\label{5.1}
\wt U v_{\tau}(a,\l)\left(=\begin{pmatrix} \G_{0a}+\wh \G_a\cr
\G_{1a}
\end{pmatrix} v_\tau(\l)\right )=\begin{pmatrix} m_\tau(\l)-\tfrac i 2 P_{\wh
H}\cr -P_H
\end{pmatrix}:H_0\to H_0\oplus H, \quad \l\in\CR.
\end{equation}

It is easily seen that Proposition 5.3 in \cite{Mog13.1} remains
valid in the case $n_+(\Tmi)<n_-(\Tmi) $. This means that for a
given operator $U$ (see \eqref{3.17.1}) and a boundary parameter
$\tau$ the $m$-functions $m_\tau^{(1)}(\cd)$ and
$m_\tau^{(2)}(\cd)$ corresponding to $J$-unitary extensions $\wt
U_1$ and $\wt U_2$ of $U$ are connected by
$m_\tau^{(1)}(\l)=m_\tau^{(2)}(\l)+B, \; \l\in\CR,$ with some
$B=B^*\in [H_0]$.

A definition of the  $m$-function $m_\tau$ in somewhat other terms
ia given in the following proposition, which directly follows from
Theorems \ref{th4.6} and \ref{th4.12}.
\begin{proposition}\label{pr5.3}
Let  $\pair$ be a boundary parameter \eqref{4.1}, \eqref{4.1a} in
Case 1 (resp. \eqref{4.54}, \eqref{4.55} in Case 2), let
$\f_U(\cd,\l)(\in [H_0,\bH]),\;\l\in\bC,$ be the operator solution
of Eq. \eqref{3.2} defined by \eqref{3.25} and let
$\psi(\cd,\l)(\in [H_0,\bH]), \;\l\in\bC,$ be the operator
solutions of Eq. \eqref{3.2} with
\begin {equation*}
\wt U \psi (a,\l)=\begin{pmatrix} -\tfrac i 2 P_{\wh H}\cr
-P_H\end{pmatrix}:H_0\to H_0\oplus H.
\end{equation*}
Then there exists a unique operator function
$m(\cd)=m_\tau(\cd):\CR\to [H_0]$ such that for any $\l\in\CR$ the
operator solution $v(\cd,\l)=v_\tau(\cd,\l)$ of Eq. \eqref{3.2}
given by
\begin {equation}\label{5.3}
v(t,\l):=\f_U (t,\l)m(\l)+\psi (t,\l)
\end{equation}
belongs to $\lo{H_0}$ and satisfies the  boundary conditions
\eqref{4.42}--\eqref{4.44} in Case 1 and \eqref{4.94},
\eqref{4.95} in Case 2.
\end{proposition}
In the following  theorem we give a description of all
$m$-functions immediately in terms of the boundary parameter
$\tau$.
\begin{theorem}\label{th5.4}
 Assume  the following hypothesis:

{\rm (i)} $\tau_0$ is a boundary parameter from Remark
\ref{rem4.3} in Case 1 and Remark \ref{rem4.9} in Case 2.

{\rm (ii)} $m_0(\cd), \; \Phi_-(\cd), \; \Psi_-(\cd)$ and $\dot
M_-(\cd)$ are the operator functions given by \eqref{4.18},
\eqref{4.19} and \eqref{4.21} (this functions form the operator
matrix $X_-(\cd)$ given by \eqref{4.17} in Case 1 and \eqref{4.81}
in Case 2).

{\rm (iii)} $M_1(\cd)$ and $M_{j-}(\cd), \; j\in\{2,3,4\}, $ are
the operator functions defined by the block representations
\eqref{4.23} (in Case 1) and \eqref{4.83}(in Case 2) of the Weyl
function $M_-(\cd)$ (according to Propositions \ref{pr4.5} and
\ref{pr4.11} these functions can be also defined as the entries of
the block matrix representations \eqref{4.24}--\eqref{4.28} of
$m_0(\cd), \; \Phi_-(\cd),\; \Psi_-(\cd)$ and $\dot M_-(\cd)$).

Then: {\rm (1)} $m_0(\l)=m_{\tau_0}(\l),\;\l\in\CR,$ and for any
boundary parameter $\pair$ defined by \eqref{4.1} in Case 1 and
\eqref{4.54} in Case 2 the corresponding $m$-function
$m_\tau(\cd)$ is
\begin {equation}\label{5.4}
m_\tau(\l)=m_0(\l)+\F_-(\l)(D_0(\l)-D_1(\l)\dot M_-(\l))^{-1}
D_1(\l) \Psi_-(\l), \quad\l\in\bC_-.
\end{equation}
{\rm (2)} For each truncated boundary parameter $\pair$ defined by
\eqref{4.1}, \eqref{4.1.1} in Case 1  and \eqref{4.54},
\eqref{4.54.1} in Case 2 the corresponding $m$-function
$m_{\tau}(\cd)$ has the triangular block matrix representation
\begin {equation}\label{5.5}
m_{\tau}(\l)=\begin{pmatrix} m_{\tau, 1}(\l)  & m_{\tau, 2}(\l)
\cr 0 & -\tfrac i 2 I\end{pmatrix}, \quad \l\in\bC_-,
\end{equation}
with respect to the decomposition $H_0=H_0'\oplus\wh H_2$ in Case
1 (see \eqref{3.31}) and $H_0=H\oplus\wh H$ in Case 2. Moreover,
the operator function $m_{\tau, 1}(\cd)$ in \eqref{5.5} is
\begin {equation}\label{5.6}
m_{\tau, 1}(\l)=M_1(\l)+M_{2-}(\l)(\ov D_0(\l)-\ov
D_1(\l)M_{4-}(\l))^{-1} \ov D_1(\l)M_{3-}(\l), \quad \l\in\bC_-.
\end{equation}
\end{theorem}
\begin{proof}
We prove the theorem only for  \emph{Case 1}, because in
\emph{Case 2} the proof is similar.

 (1) It is easily seen  that  $v_0(t,\l)=v_ {\tau_0}(t,\l)$, where
$v_0(\cd,\l)$ is defined in Proposition \ref{pr4.4}.  Hence by
\eqref{4.18} one has $m_0(\l)=m_{\tau_0}(\l)$. Next, applying the
operator $\G_{0a}+\wh\G_a$ to the equalities \eqref{4.45} and
\eqref{4.46} with taking \eqref{4.18} and \eqref{4.19} into
account one gets
\begin {gather}
m_\tau(\l)=m_0(\l)-\Phi_+(\l)(\tau_-^*(\ov\l)+\dot
M_{+}(\l))^{-1}\Psi_+(\l),
\quad\l\in\bC_+, \label{5.7}\\
m_\tau(\l)=m_0(\l)- \Phi_-(\l)(\tau_-(\l)+\dot
M_{-}(\l))^{-1}S_{-}(\l), \quad\l\in\bC_-. \label{5.8}
\end{gather}
Moreover, according to \cite[Lemma 2.1]{MalMog02}
$0\in\rho(D_0(\l)-D_1(\l)\dot M_{-}(\l))$ and
\begin {gather*}
-(\tau_-(\l)+\dot M_{-}(\l))^{-1}=(D_0(\l)-D_1(\l)\dot
M_{-}(\l))^{-1}D_1(\l), \quad\l\in\bC_- ,
\end{gather*}
which together with \eqref{5.8} yields \eqref{5.4}.

{\rm (2)} It follows from \eqref{4.1.1} and the second equality in
\eqref{4.28} that $(D_0(\l)-D_1(\l)\dot M_-(\l))^{-1}=$

$\begin{pmatrix}I & 0 \cr -\ov D_1(\l) N_{2-}(\l) & \ov
D_0(\l)-\ov D_1(\l) M_{4-}(\l)
\end{pmatrix}^{-1}=\begin{pmatrix}I & 0 \cr * & (\ov D_0(\l)-
\ov D_1(\l) M_{4-}(\l))^{-1} \end{pmatrix}$

\noindent (the entry $*$ does not matter). Combining this equality
with \eqref{5.4} and taking \eqref{4.25}--\eqref{4.27} into
account one gets the equalities \eqref{5.5} and \eqref{5.6}.
\end{proof}
By using the reasonings similar to those in the proof of
Proposition 5.7 in \cite{Mog13.1} one proves the following
proposition.
\begin{proposition}\label{pr5.6}
The $m$-function $m_\tau(\cd)$ is a Nevanlinna operator function
satisfying
\begin {equation*}
(\im \,\l)^{-1}\cd \im\, m_\tau(\l)\geq \int_\cI
v_\tau^*(t,\l)\D(t) v_\tau(t,\l)\, dt, \quad \l\in\bC_-.
\end{equation*}
\end{proposition}
\subsection{Green's function}
In the sequel we put $\gH:=\LI$ and denote by $\gH_b$ the set of
all $\wt f\in\gH$ with the following property: there exists
$\b_{\wt f}\in\cI$
 such that for some (and  hence for all) function
$f\in\wt f$ the equality $\D(t)f(t)=0 $ holds a.e. on $(\b_{\wt
f}, b)$.

Let $\f_U(\cd,\l)$ be the operator-valued solution \eqref{3.25},
let $\tau$ be a boundary parameter and let
$v_\tau(\cd,\l)\in\lo{H_0}$ be the operator-valued  solution of
Eq. \eqref{3.2} defined in Theorems \ref{th4.6} and \ref{th4.12}.
\begin{definition}\label{def6.1}
The operator function $G_\tau (\cd,\cd,\l):\cI\times\cI\to [\bH]$
given by
\begin {equation} \label{6.1}
G_\tau (x,t,\l)=\begin{cases} v_\tau(x,\l)\, \f_U^*(t,\ov\l), \;\;
x>t \cr \f_U(x,\l)\, v_\tau^* (t,\ov\l), \;\; x<t \end{cases},
\quad \l\in \CR
\end{equation}
will be called the Green's function corresponding to the boundary
parameter $\tau$.
\end{definition}
\begin{theorem}\label{th6.2}
Let $\tau$ be a boundary parameter and let $R_\tau(\cd)$ be the
corresponding generalized resolvent of the relation $T$ (see
Theorems \ref{th4.2} and \ref{th4.8}). Then
\begin {equation}\label{6.2}
R_\tau(\l)\wt f=\pi\left(\int_\cI  G_\tau (\cd,t,\l)\D(t)f(t)\, dt
\right ), \quad \wt f\in\gH, \quad  f\in\wt f.
\end{equation}
\end{theorem}
\begin{proof}
As in \cite[Theorem 6.2]{Mog13.1} one proves that for each
$f\in\lI$ the inequality

\noindent $\int\limits_\cI ||G_\tau(x,t,\l)\D(t)f(t)||\, dt
<\infty$ is valid. This implies that formula
\begin {equation} \label{6.3}
y_f=y_f(x,\l):=\int_\cI  G_\tau (x,t,\l)\D(t)f(t)\, dt,  \quad
\l\in\CR
\end{equation}
correctly defines the function $y_f(\cd,\cd):\cI\times \CR\to
\bH$. Therefore \eqref{6.2} is equivalent to the following
statement: for each $\wt f\in\gH$
\begin {equation} \label{6.4}
y_f(\cd,\l)\in\lI \;\;\text{and}\;\; R_\tau(\l)\wt f = \pi
(y_f(\cd,\l)), \quad f\in\wt f, \;\;\;\l\in\CR.
\end{equation}
To prove \eqref{6.4} first assume that $\wt f\in\gH_b$. We show
that  the function $y_f(\cd,\l)$ given by \eqref{6.3} is a
solution of the boundary problem \eqref{4.2}--\eqref{4.5} in \emph
{Case 1} (resp. \eqref{4.59}--\eqref{4.61.2} in \emph{Case 2} ).
As in \cite[Theorem 6.2]{Mog13.1} one proves that $y_f(\cd,\l)$
belongs to $ \AC\cap \lI$ and satisfies \eqref{4.2} a.e. on $\cI$.
Now it remains to show that $y_f$ satisfies the boundary
conditions \eqref{4.3}--\eqref{4.5} in \emph {Case 1} (resp.
\eqref{4.60}--\eqref{4.61.2} in \emph{Case 2} ).

It follows from \eqref{6.3} and \eqref{6.1} that
\begin{gather}
y_f(a,\l)=\f_U(a,\l)\int_\cI  v_\tau^*(t,\ov\l)\D(t)f(t)\,dt,\label{6.5}\\
y_f(x,\l)=v_\tau(x,\l)\int_\cI \f_U^*(t,\ov\l)\D(t) f(t)\, dt,
\quad x\in (\b_{\wt f},b).\label{6.6}
\end{gather}
Let $\wt U$ be a $J$-unitary extension \eqref{3.17.5} of $U$ and
let $\G_a$ be the operator \eqref{3.24a}. Then $\G_a y_f=\wt U
y_f(a,\l)$ and in view of \eqref{6.5} one has
\begin {equation*}
\G_a y_f=\wt U \f_U(a,\l)\int_\cI  v_\tau^*(t,\ov\l)\D(t)f(t)\,
dt.
\end{equation*}
This and \eqref{3.26} imply
\begin {equation} \label{6.7}
\wh\G_a y_f=P_{\wh H} \int _\cI v_\tau^*(t,\ov\l)\D(t)f(t)\,
dt,\qquad \G_{1a}y_f=0.
\end{equation}
Next assume that
\begin {equation} \label{6.8}
h_{\wt f}:=\int_\cI \f_U^*(t,\ov\l)\D(t) f(t)\, dt (\in H_0)
\end{equation}
Then in view of \eqref{6.6} one has (see \cite[Remark 3.5
(1)]{Mog13.1})
\begin {equation} \label{6.9}
\wh \G_b y_f=\wh \G_b v_\tau(\l)h_{\wt f}, \quad \G_{0b} y_f=
\G_{0b} v_\tau(\l)h_{\wt f}, \quad \G_{1b} y_f= \G_{1b}
v_\tau(\l)h_{\wt f}.
\end{equation}
Let us also prove  the equality
\begin {equation} \label{6.10}
\wh \G_a y_f=(\wh \G_a v_\tau(\l)+iP_{\wh H})h_{\wt f}.
\end{equation}
It follows from \eqref{5.1} that
\begin {equation*}
\wt U v_\tau(a,\l)\up\wh H=\begin{pmatrix} (m_\tau(\l)-\tfrac i 2
I_{H_0})\up \wh H \cr 0 \end{pmatrix}:\wh H\to H_0\oplus H, \quad
\l\in\CR.
\end{equation*}
Therefore by \eqref{3.26} one has $v_\tau(t,\l)\up\wh H=
\f_U(t,\l)(m_\tau(\l)-\tfrac i 2 I_{H_0})\up \wh H$. Moreover,
according to Proposition \ref{pr5.6} $m_\tau^*(\ov\l)=m_\tau
(\l)$. This implies that
\begin {equation*}
P_{\wh H}v_\tau^*(t,\ov\l)= P_{\wh H}(m_\tau(\l)+\tfrac i 2
I_{H_0})\f_U^*(t,\ov\l) =(P_{\wh H} m_\tau(\l) + \tfrac i 2 P_{\wh
H}) \f_U^*(t,\ov\l)
\end{equation*}
and \eqref{6.7}, \eqref{6.8} yield
\begin {equation} \label{6.11}
\wh \G_a y_f=(P_{\wh H} m_\tau(\l) + \tfrac i 2 P_{\wh H})\int_\cI
\f_U^*(t,\ov\l)\D(t) f(t)\, dt=(P_{\wh H} m_\tau(\l) + \tfrac i 2
P_{\wh H})h_{\wt f}.
\end{equation}
Since in view of \eqref{5.0} $P_{\wh H} m_\tau(\l)=\wh \G_a
v_\tau(\l)+\tfrac i 2 P_{\wh H}$, the equality \eqref{6.11} gives
\eqref{6.10}.

The second equality in \eqref{6.7} gives the first condition in
\eqref{4.3} and the condition \eqref{4.60}. Next assume  \emph
{Case 1} and the hypothesis (A2). Then by \eqref{3.30b} and
\eqref{6.10}
\begin {equation} \label{6.12}
\wh \G_{a1} y_f=(\wh \G_{a1} v_\tau(\l)+iP_{\wh H_1})h_{\wt f},
\quad \wh \G_{a2} y_f=(\wh \G_{a2} v_\tau(\l)+iP_{\wh H_2})h_{\wt
f}
\end{equation}
and combining of \eqref{6.12} and \eqref{6.9} with
\eqref{4.42}--\eqref{4.44} shows that the second condition in
\eqref{4.3} and the conditions \eqref{4.4} and  \eqref{4.5} are
fulfilled for $y_f$.

Now assume  \emph{Case 2}   and the hypothesis (A3). Then by
\eqref{3.35} and \eqref{6.9} one has $\wt\G_{0b} y_f= \wt\G_{0b}
v_\tau(\l)h_{\wt f}$. Combining of this equality and \eqref{6.9},
\eqref{6.10} with  \eqref{4.94} and \eqref{4.95} implies
fulfillment of the conditions \eqref{4.61.1} and \eqref{4.61.2}
for $y_f$.

Thus $y_f(\cd,\l)$ is a solution of the boundary value problem
\eqref{4.2}--\eqref{4.5} in \emph  {Case 1} (resp.
\eqref{4.59}--\eqref{4.61.2} in \emph{Case 2} ) and by Theorems
\ref{th4.2} and \ref{th4.8} relations \eqref{6.4} hold (for $\wt
f\in\gH_b$). Finally, one proves \eqref{6.4} for arbitrary $\wt
f\in \gH $ in the same way as in \cite[Theorem 6.2]{Mog13.1}.
\end{proof}
\begin{remark}\label{rem6.3}
Theorem \ref{th6.2}  is a generalization of similar results
obtained in \cite{DLS93,HinSch93,Kra89} for Hamiltonian systems
(i.e., for systems \eqref{3.1} with $\wh H=\{0\}$). Moreover, for
non-Hamiltonian systems \eqref{3.1} in the case of minimally
possible deficiency indices $n_\pm(\Tmi)=\nu_\pm$ formulas
\eqref{6.1} and \eqref{6.2} were proved in \cite{HinSch06}.
\end{remark}
\subsection{Spectral functions and the Fourier transform}
Let $T$ be a symmetric relation in $\gH$ defined by \eqref{3.39}
in \emph {Case 1} and \eqref{3.43} in \emph{Case 2}  and let
$\pair $ be a boundary parameter given by \eqref{4.1} in \emph
{Case 1} and \eqref{4.54} in \emph{Case 2} . According to Theorems
\ref{th4.2} and \ref{th4.8} the boundary value problems
\eqref{4.2}--\eqref{4.5}  and \eqref{4.59}--\eqref{4.61.2}
establish a bijective correspondence between boundary parameters
$\tau$ and generalized resolvents $R(\cd)=R_\tau(\cd)$ of $T$.
Denote by $\wt T^\tau$ the ($\gH$-minimal) exit space self-adjoint
extension of $T$ in the Hilbert space $\wt\gH\supset\gH$
generating $R_\tau(\cd)$:
\begin {equation} \label{6.17}
R_\tau(\l)=P_\gH (\wt T^\tau-\l)^{-1}\up \gH, \quad \l\in\CR.
\end{equation}
Clearly, formula \eqref{6.17} gives a parametrization of all such
extensions $\wt T=\wt T^\tau$ by means of a boundary parameter
$\tau$. Denote  by $F_\tau(\cd)$ the  spectral function of $T$
corresponding to $\wt T^\tau$.

In the following we assume that  a certain  $J$-unitary extension
$\wt U$ of $U$ is fixed and hence the $m$-function $m_\tau(\cd)$
is defined (see \eqref{5.0}).  Note that in view of the assertion
just after \eqref{5.1} a choice of $\wt U$ does not matter in our
further considerations.

For each $\wt f\in\gH_b$ introduce the Fourier transform $\wh
f(\cd):\bR\to H_0$ by setting
\begin {equation} \label{6.18}
\wh f(s)=\int_\cI \f_U^*(t,s)\D(t)f(t)\, dt, \quad f\in \wt f.
\end{equation}

\begin{definition}\label{def6.4}
A distribution function $\Si_\tau(\cd):\bR\to [H_0]$ is called a
spectral function of the boundary value problem
\eqref{4.2}--\eqref{4.5} in \emph  {Case 1} (resp.
\eqref{4.59}--\eqref{4.61.2} in  \emph{Case 2} ) if, for each $\wt
f\in\gH_b$ and for each finite interval $[\a,\b) \subset \bR$, the
Fourier transform \eqref{6.18} satisfies the equality
\begin {equation} \label{6.19}
((F_\tau(\b)-F_\tau(\a))\wt f,\wt
f)_\gH=\int_{[\a,\b)}(d\Si_\tau(s)\wh f(s),\wh f(s)).
\end{equation}
\end{definition}
In the following   "the boundary value problem"  means either  the
boundary value problem \eqref{4.2}--\eqref{4.5} (in \emph  {Case
1})  or the boundary value problem \eqref{4.59}--\eqref{4.61.2}
(in \emph{Case 2} ).

Let $\St$ be the spectral function of the boundary value problem.
Then in view of \eqref{6.19} $\wh f\in \LS$ and the same
reasonings as in \cite{Mog13.1} give the Bessel inequality $||\wh
f||_{\LS}\leq ||\wt f||_\gH$ (for the Hilbert space $L^2(\Si;\cH)$
see Subsection \ref{sub2.1a}). Therefore for each $\wt f\in\gH$
there exists a function $\wh f\in\LS$ (the Fourier transform of
$\wt f$) such that
\begin {equation*}
\lim\limits_{\b\uparrow b}||\wh f - \int_a^\b
\f_U^*(t,\cd)\D(t)f(t)\, dt
 ||_{\LS}=0, \quad f\in\wt f,
\end{equation*}
and the equality $V\wt f=\wh f, \; \wt f\in\gH,$ defines the
contraction $V: \gH\to \LS$
\begin{theorem}\label{th6.5}
For each boundary parameter $\tau$ there exists a unique spectral
function $\Si_\tau(\cd)$ of the boundary value problem.  This
function is defined by the Stieltjes inversion formula
\begin {equation} \label{6.24}
\Si_\tau(s)=-\lim\limits_{\d\to+0}\lim\limits_{\varepsilon\to +0}
\frac 1 \pi \int_{-\d}^{s-\d}\im \,m_\tau(\s-i\varepsilon)\, d\s.
\end{equation}
If in addition $\tau$ is a truncated boundary parameter, then the
corresponding spectral function $\Si_{\tau}(\cd)$ has the block
matrix representation
\begin {equation} \label{6.25}
\Si_{\tau}(s)=\begin{pmatrix} \Si_{\tau,1}(s) & \Si_{\tau, 2}(s)
\cr \Si_{\tau, 3}(s) & \tfrac 1 {2\pi} s I \end{pmatrix}, \quad
s\in\bR
\end{equation}
with respect to the decomposition $H_0=H_0'\oplus \wh H_2$ in Case
1 (resp. $H_0=H\oplus\wh H$ in  Case 2). In \eqref{6.25}
$\Si_{\tau,1}(\cd)$ is an $[H_0']$-valued (resp. $[H]$-valued)
distribution function defined by
\begin {equation} \label{6.26}
\Si_{\tau,1}(s)=-\lim\limits_{\d\to+0}\lim\limits_{\varepsilon\to
+0} \frac 1 \pi \int_{-\d}^{s-\d}\im \,m_{\tau,
1}(\s-i\varepsilon)\, d\s
\end{equation}
with $m_{\tau, 1}(\cd)$ taken from \eqref{5.5}
\end{theorem}
\begin{proof}
In the case of an arbitrary boundary parameter $\tau$ one proves
the required statement by using Theorem \ref{th6.2} and the
Stieltjes-Liv\u{s}ic inversion formula  in the same way as in
\cite[Theorem 6.5]{Mog13.1}. The statements of the theorem for a
truncated boundary parameter $\tau$ are implied by \eqref{6.24}
and the block matrix representation \eqref{5.5} of
$m_{\tau}(\cd)$.
\end{proof}

Let $\gH$ be decomposed as
\begin {equation}\label{6.27}
\gH=\gH_0\oplus \mul T
\end{equation}
and let $T'$ be the operator part of $T$. Moreover, let $\tau$ be
a boundary parameter, let $\wt T^{\tau}=(\wt T^{\tau})^*$ be an
exit space extension of $T$ in the Hilbert space $\wt\gH\supset
\gH$, let $\wt\gH$ be decomposed as
\begin {equation}\label{6.28}
\wt\gH=\wt\gH_0\oplus \mul \wt T^\tau
\end{equation}
and let $T^\tau$ be the operator part of $\wt T^{\tau}$ (so that
$T^\tau$ is a self-adjoint operator in $\wt \gH_0$). Assume also
that $\St$ is a spectral function of the boundary value problem
and $V\in [\gH,\LS]$ is the Fourier transform. As in
\cite{Mog13.1} one proves that $V\up \mul T=0$. Therefore by
\eqref{6.27} $\gH_0$ is the maximally possible subspace of $\gH$
on which
 $V$ may by isometric.
\begin{definition}\label{def6.7}
A spectral function $\St$ is referred to the class $SF_0$ if the
operator $V_0:=V\up\gH_0$ is an isometry from $\gH_0$ to $\LS$.
\end{definition}
Similarly to \cite{Mog13.1} the following  equivalence is valid:
\begin {equation}\label{6.29}
\St\in SF_0\iff \mul\Tt=\mul T.
\end{equation}
Hence  all spectral functions $\St$ of the boundary value problem
belong to $SF_0$ if and only if $\mul T=\mul T^*$ or,
equivalently, if and only if the operator $T'$ is densely defined.

As in \cite{Mog13.1} one proves the following statement: if
$\St\in SF_0$, then for each $\wt f\in \gH_0$ the inverse Fourier
transform is
\begin {equation}\label{6.30}
\wt f=\pi\left(\int_\bR \f_U(\cd,s)\,d\Si_\tau(s) \wh f(s)\right),
\end{equation}
where the integral converges in the semi-norm of $\lI$.

The following theorem can be proved in the same way as Theorem 6.9
in \cite{Mog13.1}.
\begin{theorem}\label{th6.9}
Let $\tau$ be a boundary parameter such that $\St\in SF_0$ and let
$V$ be the corresponding Fourier transform. Then $\gH_0\subset
\wt\gH_0$ and there exists a unitary extension  $\wt V\in
[\wt\gH_0, \LS]$ of the isometry $V_0(=V\up\gH_0)$ such that  the
operator $T^\tau$ and the multiplication operator
$\L=\L_{\Si_\tau}$ in $\LS$ are unitarily equivalent by means of
$\wt V$.

Moreover, if $\mul T=\mul T^* $, then the statements of the
theorem hold for arbitrary boundary parameter $\tau$ (that is, for
any spectral function $\St$).
 \end{theorem}
Since in the case of a densely defined $T$ one has $\mul T=\mul
T^*=\{0\}$, the following theorem is implied by Theorem
\ref{th6.9}.
\begin{theorem}\label{th6.11}
If  $T$ is a densely defined operator, then for each boundary
parameter $\tau$ and the corresponding spectral function $\St$ the
following hold: {\rm (i)} $\Tt$ is an operator, that is, $\Tt=
T^\tau$;   {\rm (ii)} the  Fourier transform $V$ is an isometry;
{\rm (iii) } there exists a unitary extension $\wt V\in [\wt\gH,
\LS]$ of $V$ such that $\Tt$ and $\L$ are unitarily equivalent by
means of $\wt V$.
\end{theorem}
Since $n_+(\Tmi)\neq n_-(\Tmi)$, the equality $\s (T^\tau)=\bR$
holds for any boundary parameter $\tau$. Moreover, Theorem
\ref{th6.9}  yields the following corollary.
\begin{corollary}\label{cor6.12}
{\rm (1)} If $\tau $ is a boundary parameter such that $\St \in
SF_0$, then the spectral multiplicity of the operator $T^\tau$
does not exceed $\nu_-(=\dim H_0)$.

{\rm (2)} If $\mul T=\mul T^* $, then the above statement about
the spectral multiplicity of $T^\tau$ holds for any boundary
parameter $\tau$.
\end{corollary}
In the case of the truncated boundary parameter $\tau$ there is a
somewhat more information about the spectrum of $T^{\tau}$.
Namely, the following theorem is valid.
\begin{theorem}\label{th6.13}
{\rm (1)} Let $\tau$ be a truncated boundary parameter such that
$\Si_{\tau}(\cd)\in SF_0$ and let $\Si_{\tau,1}(\cd)$ be a
distribution function given by the block matrix representation
\eqref{6.25} of $\Si_{\tau}(\cd)$. Moreover, let $E(\cd)$ be the
orthogonal spectral measure of the operator $T^{\tau} $ and let
$E_s(\cd)$ be the singular part of $E(\cd)$. Then
\begin {equation*}
\s_{ac}(T^{\tau})=\bR, \qquad \s_{s}(T^{\tau})=S_s(\Si_{\tau,1})
\end{equation*}
and the  multiplicity  of the orthogonal spectral measure
$E_s(\cd)$ does not exceed $\nu_+ + \nu_{b+}-\nu_{b-}(=\dim H_0')$
in  Case 1 and $\nu_+ (=\dim H)$ in  Case 2 (this implies that the
same estimates are valid for
 multiplicity of each eigenvalue  $\l_0$ of $T^{\tau} $).

{\rm (2)} If $\mul T=\mul T^*$, then the above statements about
the spectrum of $T^{\tau} $ hold for each truncated boundary
parameter $\tau$.
\end{theorem}
\begin{proof}
(1) let $\Si_{\tau,ac}(\cd), \;\Si_{\tau,s}(\cd) $ and $\Si_{\tau
1,ac}(\cd), \;\Si_{\tau 1,s} (\cd)$ be absolutely continuous and
singular parts of $\Si_{\tau} (\cd)$ and $\Si_{\tau,1}(\cd)$
respectively. Then in view of \eqref{6.25} the block matrix
representations
\begin {equation*}
\Si_{\tau,ac}(s)=\begin{pmatrix} \Si_{\tau 1,ac}(s) &
\Si_{\tau,2}(s) \cr \Si_{\tau,3}(s) & \tfrac 1 {2\pi}s I
\end{pmatrix}, \qquad \Si_{\tau,s}(s)=\begin{pmatrix}
\Si_{\tau 1,s} (s) & 0 \cr 0 & 0
\end{pmatrix}
\end{equation*}
hold with respect to the same decompositions of $H_0$ as in
Theorem
 \ref{th6.5}. This and Theorems \ref{th6.9} and \ref{th2.1} give the required
statements about the spectrum of $T^{\tau}$.

The statement (2) of the theorem follows from the fact that in the
case $\mul T=\mul T^*$ the inclusion $\St\in SF_0$ holds for any
boundary parameter $\tau$ and hence for any truncated boundary
parameter $\tau$.
\end{proof}

In the next theorem we give a parametrization of all spectral
functions $\St$ immediately in terms of the boundary parameter
$\tau$.
\begin{theorem}\label{th6.14}
Let $X_-(\cd)$ be the operator matrix \eqref{4.17} in  Case 1
(resp. \eqref{4.81} in  Case 2). Then for each boundary parameter
$\pair$ of the form \eqref{4.1} in  Case 1 (resp. \eqref{4.54} in
 Case 2) the equality
\begin {equation}\label{6.31}
m_\tau(\l)=m_0(\l)+\F_-(\l)(D_0(\l)-D_1(\l)\dot M_-(\l))^{-1}
D_1(\l) \Psi_-(\l), \quad\l\in\bC_-
\end{equation}
together with \eqref{6.24} defines a (unique) spectral function
$\St$ of the boundary value problem. If in addition $\pair$ is a
truncated boundary parameter defined by \eqref{4.1}, \eqref{4.1.1}
in Case 1 and \eqref{4.54}, \eqref{4.54.1} in Case 2, then the
distribution function $F_{\tau}(\cd)$ in \eqref{6.25} is defined
by \eqref{5.6} and \eqref{6.26}. Moreover, the following
statements are valid:

{\rm (1)} Let in Case 1 $\wh C_{02}(\l), \; C_{0b}(\l)$ and
$M_{4+}(\l), \; N_{2+}(\l) $ be defined by \eqref{4.1a} and
\eqref{4.28}  and let $P_{\cH_b}$ be the orthoprojector in $\wh
H_2\oplus\cH_b$ onto $\cH_b$. Then there exist the limits
\begin {gather*}
\cB_\tau:=\lim_{y\to +\infty} \tfrac 1 {i y}(C_{0b}(i y)-C_1(i y)
M_{4+}(i
y)+i\wh C_{02}(i y)N_{2+}(i y))^{-1} C_1(i y)=\qquad\qquad\\
=\lim_{y\to -\infty} \tfrac 1 {i y} P_{\cH_b}(D_0(i y)-D_1(i y
)\dot M_-(i y
))^{-1}D_1(i y)\qquad\qquad\qquad\qquad\quad\\
\wh\cB_\tau: =\lim_{y\to +\infty} \tfrac 1 {i y}M_{4+}(i
y)(C_{0b}(i y)-C_1(i y)M_{4+}(i y)+i\wh C_{02}(i y)N_{2+}(i
y))^{-1}C_{0b}(i y)\qquad
\end{gather*}
and the following equivalence holds:
\begin {equation}\label{6.32}
\St\in SF_0 \iff \cB_\tau=\wh \cB_\tau=0
\end{equation}

{\rm (2)} Let in  Case 2 $C_0(\l)$ and $\dot M_+(\l)$ has the
block representations (cf. \eqref{4.55} and \eqref{4.28})
\begin {gather*}
C_0(\l)=(C_{02}(\l),\, C_{0b}(\l)): (\wh H\oplus\wh \cH_b)\oplus
\cH_b \to
\cH_b, \quad \l\in\bC_+\\
\dot M_+(\l)=(N_{2+}(\l),\, N_{4+}(\l),\,
M_4(\l))^\top:\cH_b\to\wh H\oplus\wh \cH_b\oplus \cH_b,
\quad\l\in\bC_+
\end{gather*}
and let $N_+(\l)=(N_{2+}(\l),\, N_{4+}(\l))^\top \in [\cH_b,\wh
H\oplus\wh \cH_b ], \;\l\in\bC_+$. Then there exist limits
\begin {gather*}
\cB_\tau:=\lim_{y\to +\infty} \tfrac 1 {i y}(C_{0b}(i y)-C_1(i y)
M_{4}(i
y)+i C_{02}(i y)N_{+}(i y))^{-1} C_1(i y)=\qquad\qquad\\
=\lim_{y\to -\infty} \tfrac 1 {i y} P_{\cH_b}(D_0(i y)-D_1(i y
)\dot M_-(i y
))^{-1}D_1(i y)\qquad\qquad\qquad\qquad\quad\\
\wh\cB_\tau: =\lim_{y\to +\infty} \tfrac 1 {i y}M_{4}(i
y)(C_{0b}(i y)-C_1(i y)M_{4}(i y)+i C_{02}(i y)N_{+}(i
y))^{-1}C_{0b}(i y)\qquad
\end{gather*}
and the equivalence \eqref{6.32} holds.

{\rm (3)} Each spectral function $\St$ belongs to $SF_0$ if and
only if $\lim\limits_{y\to +\infty} \tfrac 1 {i y} M_{4+}(iy)=0$
in  Case 1 (resp. $\lim\limits_{y\to +\infty} \tfrac 1 {i y}
M_{4}(iy)=0$ in  Case 2) and
\begin {gather*}
\lim_{y\to-\infty} y(\im (\dot M_-(i y)h,h)-\tfrac 1 2
||P_2h||^2)=+\infty, \quad h\neq 0,
\end{gather*}
where $h\in \wh H_2\oplus \cH_b$ and $P_2$ is the orthoprojector
in  $\wh H_2\oplus \cH_b$ onto $\wh H_2$ in Case 1 (resp. $h\in
\wh H_2\oplus \wt\cH_b$ and $P_2$ is the orthoprojector in $\wh
H_2\oplus \wt\cH_b$ onto $\wh H_2\oplus\wh\cH_b$ in  Case 2).
\end{theorem}
\begin{proof}
The main statement of the theorem directly follows from Theorems
\ref{th5.4} and \ref{th6.5}.

Next, consider the boundary triplet $\dot\Pi_-=\{\dot\cH_{0}\oplus
\cH_{b}, \dot\G_0,\dot\G_1\}$ for $T^*$ defined in Propositions
\ref{pr3.5} and \ref{pr3.6}. Since the Weyl functions $M_\pm(\cd)$
of the decomposing boundary triplet $\Pi_-$ for $\Tma$ have the
block representations \eqref{4.22}, \eqref{4.23} in \emph  {Case
1} and \eqref{4.82}, \eqref{4.83} in \emph{Case 2} , it follows
from \eqref{4.28} and Proposition \ref{pr2.10a}, (3) that the Weyl
functions of the triplet $\dot\Pi_-$ are $\dot M_+ (\l)$ and $\dot
M_-(\l)$.  Now applying to the boundary triplet $\dot\Pi_-$
Theorems 4.12 and 4.13 from \cite{Mog13.2} and taking equivalence
\eqref{6.29} into account one obtains statements (1)--(3) of the
theorem.
\end{proof}
\subsection{The case of minimal deficiency indices}\label{sub5.4}
It follows from \eqref{3.27} that for a given system \eqref{3.1}
minimally possible deficiency indices of the linear relation
$\Tmi$ are
\begin {equation}\label{6.37}
n_+(\Tmi)=\nu_+, \qquad n_-(\Tmi)=\nu_-
\end{equation}
and the first (second) equality in \eqref{6.37} holds if and only
if $\nu_{b+}=0$ (resp. $\nu_{b-}=0$). If $n_+(\Tmi)=\nu_+$, then
by \eqref{3.27} one has $n_+(\Tmi)\leq n_-(\Tmi)$; moreover,
$n_+(\Tmi)< n_-(\Tmi)$ if in addition $\wh\nu >0(\Leftrightarrow
\wh H\neq \{0\})$.

In the case of the minimal  deficiency index $n_+(\Tmi)$ the above
results  can be somewhat simplified. Namely, assume the first
equality in \eqref{6.37}  and let $n_+(\Tmi)< n_-(\Tmi)$. Then the
equality $\nu_{b+}=0$ implies \emph{Case 2}. Moreover, Lemma
\ref{lem3.2a} gives $\cH_b=\{0\}$ and hence $\G_{0b}= \G_{1b}=0$.
Therefore $\wt\cH_b= \wh\cH_b$ and \eqref{3.35} yields
$\wt\G_{0b}=\wh\G_b$.

 If  the assumption (A1) from Subsection
\ref{sub4.1} is satisfied, then by Proposition \ref{pr3.6} the
equality
\begin {equation} \label{6.38}
T=\{\{\wt y, \wt f\}\in\Tma: \, \G_{1a}y=0, \;\wh\G_a y=0,\;
\wh\G_b y=0\}.
\end{equation}
defines a maximal symmetric relation $T$ in $\gH(=\LI)$ with the
deficiency indices $n_+(T)=0$ and $n_-(T)=\wh\nu+\wh\nu_{b-}$.
Moreover, in this case  $T$ coincides with the relation $A_0$
defined by \eqref{3.46}. Hence $T$ has a unique generalized
resolvent $R(\l)$ of the form \eqref{4.7}, which in view of
Theorem \ref{th4.8} and Remark \ref{rem4.9} is  given by the
boundary value problem
\begin{gather}
J y'-B(t)y=\l\D (t)y +\D(t) f(t), \quad t\in\cI,\label{6.39}\\
\G_{1a}y =0, \;\;\;\l\in\bC_+; \qquad \G_{1a}y =0, \; \wh\G_a
y=0,\; \wh\G_b y=0, \;\;\;\l\in\bC_-.\label{6.40}
\end{gather}

 If $\wt U$ is a $J$-unitary extension \eqref{3.17.5}
of $U$, then according to Proposition \ref{pr5.3} the $m$-function
$m(\cd)$ of the  problem \eqref{6.39}, \eqref{6.40} is given by
the relations
\begin{equation*}
v(t,\l):=\f_U (t,\l)m(\l)+\psi (t,\l)\in\lo{H_0},\;\;\;\l\in\CR
\end{equation*}
and  $ i\wh\G_a v(\l)=P_{\wh H},\;\;\wh\G_b v(\l)=0,
\;\;\l\in\bC_-$.

Since $\cH_b=\{0\}$, the decomposing boundary triplet
\eqref{3.36}--\eqref{3.38} for $\Tma$ takes the form
$\Pi_-=\{\cH_0\oplus H,\G_0,\G_1\}$, where
\begin {gather*}
\cH_0=H\oplus\wh H\oplus\wh\cH_b=H_0\oplus\wh\cH_b,\\
\G_0\{\wt y,\wt f\}=\{-\G_{1a}y, i \wh\G_a y,\wh\G_b y\}\,(\in
H\oplus \wh H\oplus \wh\cH_b), \quad \G_1\{\wt y,\wt f\}=\G_{0a}y,
\quad \{\wt y,\wt f\}\in \Tma.
\end{gather*}
Assume that the Weyl function $M_-(\cd)$ of $\Pi_-$ has the block
matrix representation (cf. \eqref{4.83})
\begin{gather}\label{6.41}
M_-(\l)=(M(\l), N_-(\l), M_-(\l)) :H\oplus \wh H\oplus\wh\cH_b\to
H , \quad\l\in\bC_-.
\end{gather}
According to Theorem \ref{th5.4} $m(\l)=m_0(\l)$ and by
\eqref{4.85} one has
\begin{gather}\label{6.42}
m(\l)=\begin{pmatrix} M(\l)  &   N_-(\l)  \cr 0 & -\tfrac i 2
I_{\wh H}\end{pmatrix}: \underbrace{H\oplus\wh H}_{H_0}\to
\underbrace{H\oplus\wh H}_{H_0}, \quad \l\in\bC_-.
\end{gather}
Let $\cM(\cd)$ be the operator function \eqref{2.17}, \eqref{2.18}
corresponding to  the decomposing boundary triplet $\Pi_-$. Then
by \eqref{6.41} and \eqref{6.42}
\begin {equation}\label{6.43}
\cM(\l)=\begin{pmatrix} m(\l)  &  \cM_1(\l)  \cr 0 & -\tfrac i 2
I_{\wh \cH_b}\end{pmatrix}: \underbrace{H_0\oplus\wh
\cH_b}_{\cH_0}\to\underbrace{H_0\oplus\wh \cH_b}_{\cH_0} , \quad
\l\in\bC_-,
\end{equation}
where $\cM_1(\l)=(M_-(\l),\, 0)^\top(\in [\wh\cH_b,H\oplus\wh
H])$. Using \eqref{6.43} and \eqref{2.18a} one can  show that
\begin{gather*}
m(\mu)-m^*(\l)=(\mu-\ov\l)\int_\cI v^*(t,\l)\D(t)v(t,\mu)\,dt,
\quad \mu,\l\in\bC_-.
\end{gather*}
The spectral function $\Si(\cd)$ of the problem \eqref{6.39},
\eqref{6.40} has the block matrix representation
\begin {equation*}
\Si (s)=\begin{pmatrix} F(s) & \Si_{1}(s) \cr \Si_{2}(s) & \tfrac
1 {2\pi} s I_{\wh H}\end{pmatrix}: H\oplus\wh H\to H\oplus\wh H,
\quad s\in\bR,
\end{equation*}
where $F(s)$ is an $[H]$-valued distribution function defined by
the Stieltjes formula \eqref{6.26} with $M(\l)$ in place of
$m_{\tau,1}(\l)$. Moreover, since $T$ is maximal symmetric, it
follows that $ \mul T=\mul T^*$ and, consequently, $\Si(\cd)\in
SF_0$. Hence the operator part $T_0'$ of the (unique) exit space
self-adjoint extension $T_0$ of $T$ satisfies statements of
Theorem \ref{th6.13} with $T_0'$ and $F(\cd)$ in place of $T^\tau$
and $F_\tau(\cd)$ respectively.

If in addition to the above assumptions $n_-(\Tmi)=\nu_-$ (i.e.,
both the deficiency indices $n_\pm(\Tmi)$ are minimal), then
$\wh\cH_b=\{0\}$ and $\wh\G_b=0$. This implies the corresponding
modification of the results of this subsection. In particular,
formula \eqref{6.43} takes the form $m(\l)=\cM(\l),\; \l\in\CR$.

\section{Differential operators of an odd order}
In this section we apply the above results to ordinary
differential operators of an odd order on an interval
$\cI=[a,b\rangle \; (-\infty<a <b\leq \infty)$ with the regular
endpoint $a$.

Assume that $H$ is a finite dimensional Hilbert space and
\begin {equation}\label{7.1}
l[y]= \sum_{k=0}^m  (-1)^k \left( \tfrac {i}{2}
[(q_{n-k}y^{(k)})^{(k+1)}+(q_{n-k} y^{(k+1)})^{(k)}] +
(p_{n-k}y^{(k)})^{(k)}\right)
\end{equation}
is a differential expression  of  an odd order $2m+1$ with
sufficiently smooth operator valued coefficients $p_k(\cd),\,
q_k(\cd):\cI\to [H]$ such that $p_k(t)=p_k^*(t), \;
q_k(t)=q_k^*(t)$ and $0\in\rho (q_0(t))$. Denote by $y^{[k]}(\cd),
\; k\in \{0,\; \dots,\; 2m+1\},$ the quasi-derivatives of $y\in
AC(\cI;H)$ and let $\dom l$ be the set of all functions $y\in
AC(\cI;H)$ for which $l[y]:=y^{[2m+1]}$ makes sense
\cite{Wei,KogRof75}.

As is known $H=H_t^+\oplus H_t^-$, where $H_t^+\; (H_t^-)$ is an
invariant subspace of the operator $q_0(t)$, on which $q_0(t)$ is
strictly positive (resp. negative). We put
\begin {equation*}
\nu_{0+}=\dim H_t^+, \qquad \nu_{0-}=\dim H_t^-
\end{equation*}
(this numbers does not depend on $t$); moreover, we assume for
definiteness that $\nu_{0-}\leq\nu_{0+}$.

By using formula (1.27) in \cite{KogRof75} one can easily show
that there exist finite dimensional Hilbert spaces $H'$ and $\wh
{\bold H}$ and an absolutely continuous operator function
\begin {equation*}
Q(t)=(Q_1(t),\, \wh Q(t), \, Q_2(t))^\top : H\to H' \oplus \wh
{\bold H} \oplus H', \quad t\in\cI,
\end{equation*}
such that $0\in\rho (Q(t))$ and the following holds:
\begin{gather}
i q_0(t)=-Q_1^*(t)Q_2(t)+Q_2^*(t)Q_1(t)+i\wh Q^*(t)\wh Q(t), \quad
t\in\cI\nonumber\\
\dim H'=\nu_{0-}, \qquad \dim\wh {\bold H}=\nu_{0+}-\nu_{0-}.
\label{7.1a}
\end{gather}

Introduce the finite dimensional Hilbert spaces (cf.
\eqref{3.0.1})
\begin {gather}
\bold H=\underbrace{H\oplus\dots\oplus H}_{m\;\;\rm{terms}}\oplus
H'
 (=H^m \oplus H') \label{7.2}\\
\bold H_0=\bold H\oplus \wh{\bold H}=H^m\oplus H'\oplus\wh{\bold
H},\quad \bH=\bold H_0\oplus\bold H=\bold H\oplus\wh {\bold
H}\oplus \bold H\nonumber
\end{gather}
Clearly, the space $\bH$ admits the representation
\begin {equation}\label{7.3}
\bH=\overbrace{\underbrace{H\oplus\dots\oplus H}_{m\;\;\rm{terms}}
\oplus H'}^{\bold H}\oplus \wh{\bold
H}\oplus\overbrace{\underbrace{H\oplus\dots\oplus
H}_{m\;\;\rm{terms}} \oplus H'}^{\bold H}
\end{equation}

For each function $y\in \dom l$ we let
\begin{gather}
\bold y_0(t)=\{y(t),\, \dots, \, y^{[m-1]}(t),\,
Q_1(t)y^{(m)}(t)\}(\in\bold
H)\label{7.3a}\\
\bold y_1(t)=\{y^{[2m]}(t),\, \dots, \, y^{[m+1]}(t),\,
Q_2(t)y^{(m)}(t)\}(\in\bold H)\label{7.3b}\\
\bold y(t)=\{\bold y_0(t), \, \wh Q(t)y^{(m)}(t), \, \bold y_1(t)
\}(\in \bold H\oplus\wh {\bold H}\oplus \bold H=\bH)\label{7.4}.
\end{gather}

Let $\cK$ be a finite dimensional Hilbert space. For an operator
valued solution $Y(\cd):\cI\to [\cK,H]$ of the equation
\begin {equation}\label{7.5}
l[y]=\l y\quad (\l\in\bC)
\end{equation}
we define the operator function $\bold Y(\cd ):\cI\to [\cK,\bH]$
by the following relations: if $h\in\cK$ and $Y (t)h=y(t)$, then
$\bold Y(t)h=\bold y(t)$.

Next assume that $\gH':=L^2(\cI)$ is the Hilbert space of all
Borel $H$-valued functions $f(\cd)$ on $\cI$ satisfying
$\int\limits_{\cI} ||f(t)||^2\, dt<\infty$. Denote also by $\cL^2
[\cK,H]$ the set of all operator functions $Y(\cd):\cI\to [\cK,H]$
such that $Y(t)h\in \gH', \;h\in\cK$. Moreover, by using
\eqref{7.3} we associate with a function $f(\cd)\in\gH'$ the
$\bH$-valued functions $\dot f(\cd)$ on $\cI$ given by $ \dot
f(t)=\{f(t),\, 0,\,\dots,\, 0\}, \; t\in\cI $.

According to \cite{Wei} expression \eqref{7.1} induces in $\gH'$
the maximal operator $\Lma$ and the minimal operator $\Lmi$.
Moreover, $\Lmi$ is a closed densely defined symmetric operator
and $\Lmi^*=\Lma$.

It turns out that the expression $l[y]$ is equivalent in fact to a
certain symmetric system. More precisely, the following
proposition is implied by the results of \cite{KogRof75}.
\begin{proposition}\label{pr7.1}
Let $l[y]$ be the expression \eqref{7.1} and let
\begin {equation*}
\bold J=\begin{pmatrix} 0 & 0& -I_{\bold H} \cr 0 & i I_{\wh{\bold
H}} & 0 \cr I_{\bold H} & 0 & 0
\end{pmatrix}\in [\bold H\oplus\wh{\bold H}\oplus\bold H], \qquad \bold\D(t)=
\begin{pmatrix}
\bold P & 0 & 0 \cr 0 & 0 & 0 \cr 0 & 0 & 0\end{pmatrix} \in
[\bold H\oplus\wh{\bold H}\oplus\bold H],
\end{equation*}
where $\bold P$ is the orthoprojector  in $\bold H$ onto the first
term in the right hand side of  \eqref{7.2}. Then there exists a
continuous  operator function $\bold B(t)=\bold B^*(t)(\in [\bH]),
\; t\in \cI,$ (defined in terms of $p_j$ and $q_j$) such that the
first-order symmetric system
\begin {equation}\label{7.6}
\bold J y'(t)-\bold B(t)y(t)=\bold\D(t) f(t), \quad t\in\cI
\end{equation}
and the corresponding homogeneous system
\begin {equation}\label{7.7}
\bold J y'(t)-\bold B(t)y(t)=\l \bold\D(t) y(t), \quad t\in\cI,
\;\;\l\in\bC
\end{equation}
possess the following properties:

{\rm (1)} There is a bijective correspondence $Y(\cd,
\l)\leftrightarrow \wt Y(\cd,\l)=\bold Y(\cd,\l)$ between all
$[\cK,H]$-valued operator solutions $Y(\cd, \l)$ of Eq.
\eqref{7.5} and all $[\cK,\bH]$-valued operator solutions $ \wt
Y(\cd,\l)$ of the system \eqref{7.7}. Moreover, $Y(\cd,
\l)\in\cL^2 [\cK,H]$ if and only if $\bold Y(\cd,\l)\in \cL_{\bold
\D}^2 [\cK,\bH]$.

{\rm (2)} Let $\tma$ be the maximal linear relation in
$\cL_{\bold\D}^2(\cI)$ induced by the system \eqref{7.6}. Then the
equality $(V_1 y)(t) = \bold y(t),\; y(\cd)\in\dom\Lma,$ defines a
bijective linear mapping $V_1$ from $\dom\Lma$ onto $\dom\tma$
satisfying $(V_1 y,V_1 z)_{\bold\D}=(y,z)_{\gH'}, \;
y,z\in\dom\Lma$.

{\rm (3)} Let $\Tmi$ and $\Tma$ be minimal and maximal relations
in $\gH:=L_{\bold\D}^2(\cI)$ induced by system \eqref{7.6}. Then
$\Tmi$ is a densely defined operator and the equality $V_2 f=\pi
(\dot f(\cd)), \; f=f(\cd)\in\gH',$ defines a unitary operator
$V_2\in [\gH',\gH]$ such that $(V_2\oplus V_2)\,( {\rm gr}\,\Lmi)
={\rm gr}\,\Tmi$ and $(V_2\oplus V_2)\, ({\rm gr}\,\Lma )={\rm
gr}\,\Tma$ (i.e., the operators $\Lmi$ and $\Tmi$ as well as
$\Lma$ and $\Tma$ are unitarily equivalent by means of $V_2$).
\end{proposition}
By using Proposition \ref{pr7.1} one can easily translate all the
results of \cite{Mog13.1} and the present paper to the expression
\eqref{7.1}. Below we specify only the basic points in this
direction.

Let $\wt U\in [\bH]$ be a $\bold J$-unitary operator given by
\eqref{3.17.5} with $\bold H$ and $\wh {\bold H}$ in place of $H$
and $\wh H$ respectively. Using \eqref{7.3a}--\eqref{7.4}
introduce the linear mappings $\G_{ja}:\dom l\to \bold H, \;
j\in\{0,1\},$ and $\wh \G_a: \dom l\to \wh{\bold H}$ by setting
\begin{gather}
\G_{0a}y = u_7 \bold y_0(a)+ u_8 \wh Q (a) y^{(m)}(a)+ u_9 \bold
y_1(a), \quad
y\in\dom l \label{7.8}\\
\begin{array}{c}\label{7.9}
\wh\G_ay = u_1 \bold y_0(a)+ u_2 \wh Q (a) y^{(m)}(a)+ u_3 \bold y_1(a),\\
\G_{1a}y = u_4 \bold y_0(a)+ u_5 \wh Q (a) y^{(m)}(a)+ u_6 \bold
y_1(a).
\end{array}
\end{gather}
Next, assume that $\nu_{b+}$ and $\nu_{b-}$ are indices of inertia
of the bilinear form
\begin{equation*}
[y,z]_b:=\lim_{t\uparrow b} (\bold J \bold y(t), \bold z(t)),
\quad y,z\in\dom\Lma.
\end{equation*}
Combining \eqref{3.27} with \eqref{7.2} and \eqref{7.1a} and
taking Proposition \ref{pr7.1}, (3) into account one gets the
following equality for deficiency indices $d_{\pm}=n_\pm(\Lmi)$ of
the operator $\Lmi$:
\begin {equation}\label{7.10}
d_+=m\cdot \dim H+\nu_{0-}+\nu_{b+}, \qquad d_-=m\cdot \dim
H+\nu_{0+}+\nu_{b-}.
\end{equation}
Therefore $d_+ < d_-$ if and only if one of the following two
alternative cases holds:

 \underline {\emph{ Case 1.}} $\;\;\nu_{0+}- \nu_{0-}
>\nu_{b+}-\nu_{b-}>0$.

 \underline {\emph{ Case 2.}} $\;\;\nu_{0+}- \nu_{0-}
\geq 0 \geq \nu_{b+}-\nu_{b-}$ and $\nu_{0+}- \nu_{0-}\neq
\nu_{b+}-\nu_{b-}(\neq 0)$

Proposition \ref{pr7.1}, (2) enables one to identify $\dom \Lma$
and $\dom\tma$. Therefore in the case $d_-\leq d_+$ we may assume
that the linear mapping $\G_b$ in \cite[Lemma 3.4]{Mog13.1} is of
the form
\begin {equation}\label{7.10.1}
\G_b=(\G_{0b},\, \wh\G_b, \, \G_{1b})^\top: \dom\Lma\to
\cH_{0b}\oplus \wh{\bold H}\oplus\cH_{1b},
\end{equation}
where $\cH_{0b}$ is a finite dimensional Hilbert space and
$\cH_{1b}$ is a subspace in $ \cH_{0b}$. Similarly in the case
$d_+ < d_-$ we define  on $\dom\Lma$ the linear mappings
$\G_{0b},\G_{1b}, \wh\G_b$ (\Ca{1}) and $\wt\G_{0b},\G_{1b}$
(\Ca{2}) in the same way as in Subsection \ref{sub3.3}.

In the following we suppose that: (i) $U$ is the operator
\eqref{3.17.1} (with $\bold H$ and $\wh{\bold H}$ in place of $H$
and $\wh H$) satisfying \eqref{3.17.2}--\eqref{3.17.4}; (ii)
$\wh\G_a$ and $\G_{1a}$ are the mappings \eqref{7.9}. Then the
same boundary conditions as in the right hand sides of
\eqref{3.39}, \eqref{3.43} and \cite[(3.40)]{Mog13.1} give a
symmetric operator $T(\supset \Lmi)$ in $\gH'$. Moreover, we
define a boundary parameter $\tau$ in the same way as for
symmetric systems in \cite[Definition 5.1]{Mog13.1} (the case
$d_-\leq d_+$) and Definitions \ref{def4.1} and \ref{def4.7} (the
case $d_+ < d_-$). A boundary parameter $\tau$ gives a
parametrization of all generalized resolvents $R_\tau(\cd)$ and,
consequently, all spectral functions $F_\tau(\cd)$ of $T$. Such a
parametrization is generated by means of a boundary value problem
involving the equation
\begin {equation}\label{7.10.2}
l[y]-\l y=f(t), \quad t\in \cI
\end{equation}
and the boundary conditions   \cite[(4.2) and (4.3)]{Mog13.1}
($d_-\leq d_+$), \eqref{4.3}--\eqref{4.5} ($d_+ <d_-$, \Ca{1}) or
\eqref{4.60}--\eqref{4.61.2} ($d_+ <d_-, \Ca{2}$). In the
following "the boundary value problem for $l[y]$" means one of the
listed above boundary value problems for the expression $l[y]$.

If $d_+ =d_-$, then in \eqref{7.10.1} $\cH_{0b}=\cH_{1b}=:\cH_b$
and $R_\tau(\cd)$ is a canonical resolvent of $T$ if and only if
$\tau=\{(C_0,C_1);\cH_b\}$ is a self-adjoint operator pair (this
means that $C_0,C_1\in [\cH_b], \; \im C_1 C_0^*=0$ and $0\in\rho
(C_0\pm i C_1)$). In this case the corresponding boundary problem
involves equation \eqref{7.10.2} and self-adjoint boundary
conditions
\begin {equation}\label{7.10.3}
\G_{1a}y=0, \quad \wh\G_a y=\wh\G_b y, \quad  C_0\G_{0b}y+C_1
\G_{1b}y=0,
\end{equation}
where $\G_{0b}, \; \G_{1b}$ and $\wh\G_b$ are taken from
\eqref{7.10.1}. Moreover, $R_\tau(\l)=(T^\tau - \l)^{-1}$, where
$T^\tau=\Lma\up \dom T^\tau$ is a self-adjoint extension of $\Lmi$
with the domain
\begin {equation}\label{7.10.4}
\dom T^\tau=\{y\in\dom\Lma:\G_{1a}y=0, \;  \wh\G_a y=\wh\G_b y, \;
C_0\G_{0b}y+C_1 \G_{1b}y=0\} .
\end{equation}

Next assume that $\wt U\in [\bold H\oplus\wh{\bold H}\oplus\bold
H]$ is a $\bold J$-unitary extension \eqref{3.17.5} of $U$ and
$\G_{0a}$ is the mapping \eqref{7.8}.
\begin{definition}\label{def7.2}
The $m$-function of the expression $l[y]$ corresponding to the
boundary parameter $\tau$ is the $m$-function $m_\tau(\cd)$ of the
equivalent system \eqref{7.6}.
\end{definition}
Definition \ref{def7.2} means that $m_\tau(\cd):\CR\to [\bold
H_0]$ is a unique operator function such that for any  $\l\in\CR$
the operator solution $v_\tau(\cd,\l)$ of Eq. \eqref{7.5} given by
\begin {equation}\label{7.11}
v_\tau(t,\l):=\f_U (t,\l)m_\tau(\l)+\psi (t,\l)
\end{equation}
belongs to $\cL^2[\bold H_0,H]$ and satisfies the  boundary
conditions \cite[(4.41)--(4.43)]{Mog13.1} ($d_-\leq d_+$),
\eqref{4.42}--\eqref{4.44} ($d_+< d_-$, \Ca{1}) or \eqref{4.94}
and \eqref{4.95} ($d_+< d_-$, \Ca{2}). In \eqref{7.11}
$\f_U(\cd,\l)(\in [\bold H_0,H])$ and $\psi (\cd,\l)(\in [\bold
H_0,H])$ are the operator solutions of Eq. \eqref{7.5} with  the
initial data
\begin {equation*}
\wt U \pmb \f_U(a,\l)=\begin{pmatrix} I_{\bold H_0}\cr 0
\end{pmatrix}:\bold H_0\to \bold H_0 \oplus \bold H, \qquad \wt U
\pmb \psi(a,\l)=\begin{pmatrix} -\tfrac i 2 P_{\wh{\bold H}} \cr
-P_{\bold H} \end{pmatrix}:\bold H_0\to \bold H_0 \oplus \bold H
\end{equation*}
(note that $\f_U(\cd,\l)$ does not depend on a choice of the
extension $\wt U\supset U $).

Let $\gH_b'$ be the set of all functions $f(\cd)\in\gH'$ such that
$f(t)=0$ on some interval $(\b, b)$ (depending on $f$). For a
function $f(\cd)\in\gH'_b$ the Fourier transform $\wh
f(\cd):\bR\to \bold H_0$ is
\begin {equation} \label{7.13}
\wh f(s)=\int_\cI \f_U^*(t,s) f(t)\, dt.
\end{equation}
\begin{definition}\label{def7.3}
Let $\tau$ be a boundary parameter. A distribution function
$\St:\bR \to [\bold H_0 ] $ is called a spectral function of the
boundary value problem for $l[y]$ if  for each $f(\cd)\in\gH_b'$
the Fourier transform \eqref{7.13} satisfies \eqref{6.19} (with
$\gH'$ in place of $\gH$).
\end{definition}
Since $\Lmi$ is densely defined, formula \eqref{6.19} yields the
Parseval equality $||\wh f||_{\LSB}=||f||_{\gH'}$. Therefore for
each $f\in\gH'$ there exists the fourier transform \eqref{7.13}
(the integral converges in the norm of $\LSB$) and the equality
$Vf=\wh f,\; f\in\gH',$ define an isometry $V:\gH'\to \LSB$.

In view of Proposition \ref{pr7.1} each spectral function $\St$ of
the expression $l[y]$ is a spectral function of the equivalent
symmetric system \eqref{7.6} and vice versa. Hence by
\cite[Theorem 6.5]{Mog13.1} and Theorem \ref{th6.5} for each
boundary parameter $\tau$ there exits a unique spectral function
$\St$ of the boundary value problem for $l[y]$ and this function
is defined by the Stieltjes inversion formula \eqref{6.24}.
Moreover, for each $f\in\gH'$ the inverse Fourier transform is
\begin {equation*}
 f(t)=\int_\bR \f_U(t,s)\,d\Si_\tau(s) \wh f(s).
\end{equation*}

For a given boundary parameter $\tau$ denote by $T^\tau$ the
($\gH$-minimal) exit space self-adjoint extension of $T$
generating $R_\tau(\l)$. Since $T$ is densely defined, $T^\tau$ is
a self-adjoint operator in the Hilbert space $\wt\gH\supset \gH'$.

The following theorem is implied by Theorem \ref{th6.11} and
\cite[Theorem 6.11]{Mog13.1}.
\begin {theorem}\label{th7.4}
Let $\tau$ be a boundary parameter and let $\St$ and  $V$ be the
corresponding spectral function and  Fourier transform
respectively. Then there exists a unitary extension $\wt V\in
[\wt\gH,\LSB]$ of $V$ such that $T^\tau$ and the multiplication
operator $\L$ in $\LSB$ are unitarily equivalent by means of $\wt
V$. Moreover, the following statements are equivalent:

{\rm (1)} $d_+=d_-$ and $ \tau =\{(C_0,C_1);\cH_b\}$ is a
self-adjoint operator pair, so that $T^\tau$ is the canonical
self-adjoint extension \eqref{7.10.4} of $T$;

{\rm (2)} $V\gH' =\LSB$, that is the fourier transform $V$ is a
unitary operator.

If the statement {\rm (1)} (and hence {\rm (2)}) is valid, then
the operators $T^\tau$ and $\L$ are unitarily equivalent by means
of $V$.
\end{theorem}
Theorem \ref{th7.4} immediately implies that the spectral
multiplicity of $T^\tau$ does not exceed $m\cd\dim
H+\nu_{0+}(=\dim\bold H_0)$.

Similarly one can translate to the expression $l[y]$ Theorem
\ref{th6.13} and parametrization of all spectral functions $\St$
by means of \eqref{6.31}, \cite[(5.22)]{Mog13.2} and the Stieltjes
inversion formula.

It follows from \eqref{7.10} that for a given expression
\eqref{7.1} minimally possible deficiency indices of the operator
$\Lmi$ are
\begin {equation*}
n_+(\Lmi)=m\cd\dim H+\nu_{0-}, \qquad n_-(\Lmi)=m\cd\dim
H+\nu_{0+}.
\end{equation*}
The routine reformulation of the results of Subsection
\ref{sub5.4} to the case of the expression $l[y]$  with minimal
deficiency index $n_+(\Lmi)$ of the operator $\Lmi$ is left to the
reader.

\begin{remark}\label{rem7.5}
Let $H=\bC$, so that $l[y]$ is a scalar  expression with real
valued coefficients  $p_k(\cd)$ and $q_k(\cd)$. Then $q_0(t)>0, \;
\wh Q(t)=q_0^{\frac 1 2} (t), \; \wh{\bold H}=\bC$ and the
equalities \eqref{7.2}--\eqref{7.4} take the form
\begin{gather*}
\bold H=\bC^m, \qquad \bold H_0=\bC^m\oplus\bC=\bC^{m+1}, \qquad
\bH=\bC^{m+1}\oplus \bC^m=\bC^{2m+1}\\
\bold y_0(t)=\{y(t),\, \dots, \, y^{[m-1]}(t)\}(\in\bC^m),\quad
\bold
y_1(t)=\{y^{[2m]}(t),\, \dots, \, y^{[m+1]}(t)\}(\in\bC^m),\\
\bold y(t)=\{\bold y_0(t), \,q_0^{\frac 1 2} (t)y^{(m)}(t), \,
\bold y_1(t) \}(\in \bC^{2m+1}).
\end{gather*}
In this case $\f_U(\cd,\l)$ is the $(m+1)$-component operator
solution
\begin {equation*}
\f_U(t,\l)=(\f_1(t,\l), \, \f_2(t,\l),\,\dots, \,\f_m(t,\l),\,
\f_{m+1}(t,\l)):\bC^{m+1}\to\bC
\end{equation*}
of Eq. \eqref{7.5} with the initial data
\begin {equation*}
\wt U\pmb \f_U(a,\l)=\begin{pmatrix} I_{m+1}\cr 0  \end{pmatrix}:
\bC^{m+1}\to \bC^{m+1} \oplus \bC^m
\end{equation*}
and \eqref{7.13} defines the Fourier transform $\wh f(\cd):\bR\to
\bC^{m+1}$. Moreover, in a fixed basis of $\bC^{m+1}$ the
$m$-function $m_\tau(\cd)$ and spectral function $\St$ can be
represented as $(m+1)\times (m+1)$-matrix valued functions
$m_\tau(\l)=(m_{ij}(\l))_{i,j=1}^{m+1}$ and
$\Sigma_\tau(\l)=(\Sigma_{ij}(\l))_{i,j=1}^{m+1}$ respectively.

In conclusion note that for a scalar expression \eqref{7.1} one
has $\nu_{0+}=1$  and $\nu_{0-}=0$. Therefore for a scalar
expression $l[y]$ with $n_+(\Lmi)< n_-(\Lmi)$ \Ca{1} is
impossible. In other words, for such an expression either
$n_-(\Lmi)\leq n_+(\Lmi)$ or $n_+(\Lmi)< n_-(\Lmi)$ and \Ca{2}
holds.
\end{remark}

\end{document}